\documentclass[a4paper,11pt,twoside]{article}
\usepackage[utf8]{inputenc}
\usepackage{amsmath}  
\usepackage{amsthm,amsfonts,amssymb,esint,cases,bbm,mathtools}
\usepackage[margin=1.5in, marginparwidth=1.2in]{geometry}
\usepackage{enumitem}
\usepackage{graphicx}
\usepackage{microtype}
\usepackage{fancyhdr}
\fancypagestyle{plain}{

\fancyhf{}
\fancyfoot[L]{\footnotesize \today}
\fancyfoot[C]{\thepage}
}
\usepackage{tocloft}
\usepackage{appendix}
\usepackage[
pdfauthor={Markus Tempelmayr and Hendrik Weber}, 
pdftitle={A priori bounds for stochastic porous media equations via regularity structures}, pdftex, hypertexnames=false]{hyperref}
\hypersetup{colorlinks=false, pdfborder={0 0 0}}
\usepackage{autonum} 
\newcommand{\1}{\mathbbm{1}}
\newcommand{\D}{\mathcal{D}}
\newcommand{\Da}{\mathcal{D}^{\bar a}}
\newcommand{\Dd}{\mathcal{D}^{2^{-q}}}
\renewcommand{\epsilon}{\varepsilon}
\newcommand{\E}{\mathbb{E}}
\newcommand{\K}{\bar{K}}
\newcommand{\Kn}{K^{(n)}}
\newcommand{\lesssimdata}{\lesssim_{\textrm{data}}}
\newcommand{\N}{\mathbb{N}}
\newcommand{\R}{\mathbb{R}}
\newcommand{\s}{\mathfrak{s}}
\renewcommand{\S}{\mathcal{S}_\mathfrak{s}}
\newcommand{\T}{\mathbb{T}}
\newcommand{\x}{\mathbf{x}}
\newcommand{\y}{\mathbf{y}}
\newcommand{\yy}{\bar{\mathbf{y}}}
\newcommand{\z}{\mathbf{z}}
\newcommand{\Z}{\mathbb{Z}}
\theoremstyle{plain}
\newtheorem{theorem}{Theorem}[section]
\newtheorem*{theorem*}{Theorem}
\newtheorem{proposition}[theorem]{Proposition}

\newtheorem{lemma}[theorem]{Lemma}
\theoremstyle{definition}

\newtheorem{assumption}[theorem]{Assumption}
\newtheorem{remark}[theorem]{Remark}

\newtheorem{example}[theorem]{Example}
\theoremstyle{remark}
\numberwithin{equation}{section}
\usepackage{bbding}

\usepackage{istgame}
\newcommand{\noise}{
\hspace{-1.1ex}
\raisebox{0.25ex}{
\begin{istgame}
\istroot(0) \endist
\end{istgame}}
\hspace{-.3ex}
}
\newcommand{\anoise}{
a\hspace{-1.2ex}\raisebox{-.3ex}{\scalebox{.75}{\noise}}\hspace{-1.2ex}
}
\newcommand{\fnoise}{
f\hspace{-1.3ex}\raisebox{-.3ex}{\scalebox{.75}{\noise}}\hspace{-1.2ex}
}
\newcommand{\gnoise}{
g\hspace{-1.3ex}\raisebox{-.3ex}{\scalebox{.75}{\noise}}\hspace{-1.2ex}
}
\newcommand{\snoise}{
\sigma\hspace{-1.5ex}\raisebox{-.3ex}{\scalebox{.75}{\noise}}\hspace{-1.2ex}
}
\newcommand{\ssnoise}{
\sigma^2\hspace{-2.4ex}\raisebox{-.3ex}{\scalebox{.75}{\noise}}\hspace{-.8ex}
}
\newcommand{\lolly}{
\hspace{-1.1ex}
\raisebox{-0.5ex}{
\begin{istgame}
\xtdistance{2.mm}{1.5mm}
\setistgrowdirection{north}
\istroot(0)[null node] \istb* \endist
\end{istgame}}
\hspace{-.3ex}
}
\newcommand{\lollysmall}{
\hspace{-1.7ex}
\raisebox{-0.5ex}{
\scalebox{.75}{
\begin{istgame}
\xtdistance{2.mm}{1.5mm}
\setistgrowdirection{north}
\istroot(0)[null node] \istb* \endist
\end{istgame}}}
\hspace{-.3ex}
}
\newcommand{\alolly}{
a\hspace{-.9ex}\raisebox{-.3ex}{\scalebox{0.6}{\lolly}}\hspace{-1.ex}
}
\newcommand{\flolly}{
f\hspace{-1.2ex}\raisebox{-.3ex}{\scalebox{0.6}{\lolly}}\hspace{-1.ex}
}
\newcommand{\slolly}{
\sigma\hspace{-.9ex}\raisebox{-.3ex}{\scalebox{0.6}{\lolly}}\hspace{-1.ex}
}
\newcommand{\dumb}{
\hspace{-1.1ex}
\raisebox{-0.3ex}{
\begin{istgame}
\xtdistance{2.mm}{1.5mm}
\setistgrowdirection{north}
\istroot(0) \istb* \endist
\end{istgame}}
\hspace{-.3ex}
}
\newcommand{\dumbsmall}{
\hspace{-1.7ex}
\raisebox{-0.3ex}{
\scalebox{.75}{
\begin{istgame}
\xtdistance{2.mm}{1.5mm}
\setistgrowdirection{north}
\istroot(0) \istb* \endist
\end{istgame}}}
\hspace{-.3ex}
}
\newcommand{\xnoise}{
X\hspace{-1.4ex}
\raisebox{0.4ex}{
\begin{istgame}
\istroot(0) \endist
\end{istgame}}
\hspace{-.3ex}
}
\newcommand{\derivativecherry}{
\hspace{-1.1ex}
\raisebox{-0.5ex}{
\begin{istgame}
\xtdistance{2.mm}{1.5mm}
\setistgrowdirection{north}
\istroot(0)[null node] \istb*[double] \istb*[double] \endist
\end{istgame}}
\hspace{-.3ex}
}
\newcommand{\cherrysmall}{
\hspace{-1.7ex}
\raisebox{-0.5ex}{
\scalebox{.75}{
\begin{istgame}
\xtdistance{2.mm}{1.5mm}
\setistgrowdirection{north}
\istroot(0)[null node] \istb*[double] \istb*[double] \endist
\end{istgame}}}
\hspace{-.3ex}
}
\newcommand{\cdumb}{
C\hspace{-1ex}\raisebox{-.5ex}{\scalebox{.6}{\dumb}}\hspace{-1.3ex}
}
\newcommand{\cdumbeps}{
\cdumb_{\, ,\epsilon}
}
\newcommand{\ccherry}{
C\hspace{-1ex}\raisebox{-.5ex}{\scalebox{.6}{\derivativecherry}}\hspace{-1.5ex}
}
\newcommand{\ccherryeps}{
\ccherry_{\,,\epsilon}
}
\title{\vspace{-2ex}\bf\Large A priori bounds for stochastic porous media equations via regularity structures}
\author{Markus Tempelmayr and Hendrik Weber \\[1ex] 
{\footnotesize\textsc{Universität Münster}}
}
\date{}
\renewenvironment{abstract}
{
\begin{center}
\begin{minipage}{.9\textwidth}\small\textbf{Abstract.}
}
{
\end{minipage}
\end{center}
}
\newenvironment{keywords}
{
\begin{center}
\begin{minipage}{.9\textwidth}\small\textbf{Keywords:}%
}
{
\end{minipage}
\end{center}
}
\newenvironment{msc}
{
\begin{center}
\begin{minipage}{.9\textwidth}\small\textbf{MSC 2020:}%
}
{
\end{minipage}
\end{center}
}
\begin{document}

\maketitle

\begin{abstract}
We prove a priori bounds for solutions of singular stochastic porous media equations with multiplicative noise in their natural $L^1$-based regularity class. 
We consider the first singular regime, 
i.e.~noise of space-time regularity $\alpha-2$ for $\alpha\in(2/3,1)$,
and prove modelledness of the solution in the sense of regularity structures
with respect to the solution of the corresponding linear stochastic heat equation. 

The proof relies on the kinetic formulation of the equation 
and a novel renormalized energy inequality. 
A careful analysis allows to balance the degeneracy of the diffusion coefficient against sufficiently strong damping of the multiplicative noise for small values of the solution.
\end{abstract}




\begin{msc} 
60H17, 
35K65, 
60L30, 
35B45. 
\end{msc}


\tableofcontents


\section{Introduction}\label{sec:intro}

We are interested in the stochastic partial differential equation (SPDE)
\begin{alignat}{2}
\partial_t u - \nabla\cdot(a(u)\nabla u) 
&= \sigma(u)\xi \quad&&\textnormal{ on } (0,T)\times\T^d, \\
u(t=0,\cdot) 
&= u_0 &&\textnormal{ on } \T^d,
\end{alignat}
in the following situation: 
on the one hand we consider $a\geq0$ such that the equation is \emph{degenerate}, 
and on the other hand we consider irregular random $\xi$ such that the equation is a \emph{singular} SPDE. 
A prominent representative of this class of equations is the stochastic porous medium equation corresponding to $a(u)=M|u|^{M-1}$ for~$M>1$.

In the non-degenerate setting $a\geq\lambda>0$, the equation is singular 
if the realizations of $\xi$ take values in the parabolic H\"older space $C^{\alpha-2}$ for $\alpha<1$: 
in this case $u$ and hence $\sigma(u)$ for sufficiently smooth $\sigma$ are expected to take values in $C^\alpha$, 
and consequently $\sigma(u)\xi$ is a singular product ($\alpha+\alpha-2<0$). 
For simplicity we restrict ourselves to $\alpha>2/3$, 
including e.g.~space white noise on $\T^2$. 
In particular, $a(u)\nabla u$ is well-defined ($\alpha+\alpha-1>0$) in this range of $\alpha$. 

As is commonly done in the literature of singular SPDEs, 
we mollify $\xi$ and include a counterterm (depending on the mollification) in the equation. 
This should be done in such a way, that the sequence of resulting solutions converges to a (non-trivial) limit as the mollification is removed. 
However, in the degenerate setting, 
and since we do not necessarily want to assume differentiability of $a$ at $0$ 
(e.g.~we want to cover $a(u)=M|u|^{M-1}$ for $M>1$), 
even for smooth noise there is not necessarily a classical solution of the equation.
We therefore also regularize $a$ to be smooth and bounded away from $0$. 
Similarly, we aim for $u_0\in L^{p}(\T^d)$ for $p>1$, 
which for convenience we approximate by smooth initial conditions. 
For such $a,\xi,u_0$ we consider henceforth the regularized and renormalized version 
\begin{equation}\label{eq:spm}
\begin{alignedat}{2}
\partial_t u - \nabla\cdot(a(u)\nabla u) 
&= \sigma(u)\xi - \sigma'(u)\sigma(u) C^{a(u)} 
\quad&&\textnormal{ on } (0,T)\times\T^d, \\
u(t=0,\cdot) 
&= u_{0} && \textnormal{ on } \T^d,
\end{alignedat}
\end{equation}
with a counterterm $C^{a(u)}$ to be specified in Assumption~\ref{ass:noise} below. 
Note that this regularized version admits a unique classical local in time solution \cite[Chapter V, Theorem~8.1]{LSU68}. 
To avoid any large scale issues, 
we shall assume furthermore that $\sigma$ is compactly supported. 
As a consequence, solutions of \eqref{eq:spm} are global in time.
The main result, an a-priori regularity estimate of the solution of \eqref{eq:spm}, 
is however uniform in the qualitative smoothness of $a,\xi,u_0$.

The analysis to obtain this a-priori estimate requires a careful balancing between the two nonlinearities $a$ and $\sigma$: 
whenever $a$ vanishes and the smoothing effect of the heat operator is lost, 
the noise intensity has to be damped sufficiently by $\sigma$. 
This is made precise in the following assumption. 

\begin{assumption}[Nonlinearities]\label{ass:nonlinearities}
~
\begin{itemize}
\item[(i)](qualitative). 
Assume that $a\in C^2(\R)$, 
and that there exists $a_0>0$ such that 
\begin{equation}\label{eq:apositive}
a\geq a_0 \, .
\end{equation}
\item[(ii)](quantitative). 
Assume that there exist $C>0$ and $M>1$ such that 
\begin{equation}\label{eq:alower}
|v|^{M-1}\leq C a(v) \, , \
|a'(v)|\leq C |v|^{M-2} \, , \
|a''(v)| \leq C |v|^{M-3} \, ,
\end{equation}
the former for $v\in\R$ and the two latter for $v\in\R\setminus\{0\}$.
Furthermore, assume that $\sigma\in C^2(\R)$ is compactly supported in $[-C,C]$,
and that there exists $N\geq M+1$ such that for $v\in\R$
\begin{equation}\label{eq:supper}
|\sigma(v)|\leq C|v|^{N} \, , \quad
|\sigma'(v)|\leq C|v|^{N-1} \, , \quad
|\sigma''(v)|\leq C|v|^{N-2} \, .
\end{equation}
\end{itemize}
\end{assumption}

\noindent
The quantitative part entails in particular $a\in C(\R)\cap C^2(\R\setminus\{0\})$ and $a\geq0$.
We refer to Example~\ref{ex:pm} below to see how $a(u)=Mu^{M-1}$ fits into this assumption. 

We turn to a discussion of the noise. 
As is typical for a singular SPDE, 
the mere information $\xi\in C^{\alpha-2}$ is not enough for a well posedness theory of \eqref{eq:spm}, 
and a certain ``enhancement'' of the noise is required. 
Adopting the point of view of regularity structures\footnote{The reader not familiar with the theory of regularity structures need not be scared -- we use its terminology and ideas, however the article is self contained.}, 
we thus assume that 
a (parameter dependent) model $\Pi_\x$ is defined on the list of symbols 
$\mathcal{T}
=\{\noise,
\lolly,
\dumb,
\derivativecherry,
\xnoise
\}$ by 
\begin{align}
\Pi_\x[\noise](\y) &= \xi(\y) \, , \\
\Pi_\x[\lolly;\bar{a}](\y) &= \int_{(0,\infty)\times\R^d} d\z
\Big( \Psi\big(\bar{a},\y-\z\big) - \Psi\big(\bar{a},\x-\z\big) \Big) \xi(\z) \, , \label{eq:lolly} \\
\Pi_\x[\dumb;\bar{a}](\y) &= \Pi_\x[\lolly;\bar{a}](\y) \xi(\y) - \cdumb^{\bar{a}}(s) \, , \label{eq:dumb} \\
\Pi_\x[\derivativecherry;\bar{a}](\y) &= |\nabla_y \Pi_\x[\lolly;\bar{a}](\y)|^2 - \ccherry^{\bar a}(s) \, , \label{eq:derivativecherry} \\
\Pi_\x[\xnoise](\y) &= (y-x) \xi(\y) \, ,
\end{align}
for $\xi\colon\R\times\T^d\to\R$ from \eqref{eq:spm} and for some counterterms $\cdumb^{\bar a},\ccherry^{\bar a}$ to be specified in Assumption~\ref{ass:noise} below. 
Here and in the following we identify functions on $\T^d$ with $1$-periodic functions on $\R^d$, 
we shall write $\x=(t,x),\,\y=(s,y),\,\z=(r,z)\in\R\times\R^d$ for space-time points, 
and 
\begin{equation}\label{eq:dilatedheatkernel}
\Psi(\bar a,\x)\coloneqq \Phi(\bar at,x) \quad \textnormal{for }\bar a>0,\, \x=(t,x)\in\R\times\R^d
\end{equation}
denotes the dilation by $\bar a$ of the $d$-dimensional heat kernel 
\begin{equation}\label{eq:heatkernel}
\Phi(t,x)\coloneqq (4\pi t)^{-\frac{d}{2}} e^{-\frac{|x|^2}{4t}} 
\1_{(0,\infty)}(t) \, .
\end{equation}
We shall also write $\Pi_\x[\tau;\bar\x]$ as shorthand for $\Pi_\x[\tau;a(u(\bar\x))]$.
The reader wondering why $\Pi[\derivativecherry]$ is required to study \eqref{eq:spm} is referred to Section~\ref{sec:strategy}. 
The main assumption on the noise aside from its regularity $C^{\alpha-2}$, see \eqref{eq:modelbound} below for $\tau=\noise$,
is that the counterterms $\cdumb^{\bar a},\ccherry^{\bar a}$ are chosen such that the model can be estimated uniformly with respect to the qualitative smoothness of $\xi$, 
and that the counterterms are ``close'' to the counterterm $C^{\bar a}$ of \eqref{eq:spm}.
This is made precise in the following assumption.
\begin{assumption}[Noise]\label{ass:noise}
~
\begin{itemize}
\item[(i)] (qualitative).
Assume that $\xi\colon\R\times\T^d\to\R$ is smooth.
\item[(ii)] (quantitative).
Assume that for $\tau\in\mathcal{T}$ and compact $\mathfrak{K}\subseteq\R\times\R^d$ there exists $[\tau]_{|\tau|}<\infty$ such that
\begin{equation}\label{eq:modelbound}
|\big\langle \partial_{\bar{a}}^m \Pi_\x[\tau;\bar{a}] , \varphi_\x^\lambda \big\rangle |
\leq [\tau]_{|\tau|} \, \lambda^{|\tau|} \, \bar{a}^{-\mathfrak{e}(\tau)-m} \, ,
\end{equation}
and 
\begin{equation}\label{eq:lollybound}
| \partial_{\bar{a}}^m \Pi_\x[\lolly;\bar{a}](\y) |
\leq [\lolly]_\alpha \, \|\x-\y\|_\s^\alpha \, \bar{a}^{-\mathfrak{e}(\lollysmall)-m} \, ,
\end{equation}
uniformly for $\x,\y\in\mathfrak{K}$, $\lambda\in(0,1)$, $\bar{a}\in a(\mathrm{supp}\,\sigma)\setminus\{0\}$, $m=0,1$, and $\varphi\in\mathcal{B}$ 
the set of smooth functions supported in the (parabolic) unit ball with 
all its derivatives up to order $2$ bounded by $1$. 
Here $\|\y\|_\s$ denotes the parabolic Carnot-Carath\'eodory metric 
$(|s|+y_1^2+\dots+y_d^2)^{1/2}$, 
and $\varphi_{(t,x)}^\lambda(s,y)=\lambda^{-d-2}\varphi(\tfrac{s-t}{\lambda^2},\tfrac{y-x}{\lambda})$. 
Both the homogeneity $|\tau|$ and $\mathfrak{e}(\tau)$ 
of a symbol $\tau\in\mathcal{T}$ are given in Table~\ref{tab:hom}.

Furthermore, assume that there exists $C>0$ such that for $C^{a(v)}$ from \eqref{eq:spm}
\begin{align}
\int_0^T dt \, \sup_{v\in\R} \, 
\big| \sigma'(v)\sigma(v) \big(C^{a(v)}-\cdumb^{a(v)}(t) \big) \big| 
&\leq C \, , \label{eq:counterterm_time_independent} \\
\int_0^T dt \, \sup_{v\in\R} \, 
\big| \sigma^2(v) / v
\big(\cdumb^{a(v)}(t)-a(v) \ccherry^{a(v)}(t) \big)\big| 
&\leq C \, . \label{eq:counterterm_dumb_cherry}
\end{align}
\end{itemize}
\end{assumption}

Although the analysis carried out in this work is purely deterministic, 
probabilistic arguments are necessary to verify Assumption~\ref{ass:noise}, see Examples~\ref{ex:SWN} and \ref{ex:CN} below. 

\begin{table}[h]
\begin{center}
\begin{tabular}{c|ccccc}
$\tau$ & $\noise$ & $\dumb$ & $\derivativecherry$ & $\xnoise$ & $\lolly$ 
\\ 
\hline
$|\tau|$ & $\alpha-2$ & $2\alpha-2$ & $2\alpha-2$ & $\alpha-1$ & $\alpha$ 
\\ 
\hline 
$\mathfrak{e}(\tau)$ & $0$ & $1+\epsilon$ & $2+2\epsilon$ & $0$ & $1+\epsilon$ 
\end{tabular}
\end{center}
\caption{Homogeneity $|\tau|$ and $\mathfrak{e}(\tau)$ of the symbols $\tau\in\mathcal{T}$ sorted according to their homogeneity from small (left) to large (right), where $\epsilon>0$ is arbitrarily small.}
\label{tab:hom}
\end{table}

Before stating the main result let us finally make more precise the assumption
on the initial condition.

\begin{assumption}[Initial condition]\label{ass:initial}
~
\begin{itemize}
\item[(i)](qualitative). 
Assume that $u_0\colon\T^d\to\R$ is smooth. 
\item[(ii)](quantitative).
Assume that there exist $p>1$ and $C >0$ such that 
\begin{equation}
\|u_0\|_{L^p(\T^d)}\leq C \, .
\end{equation}
\end{itemize}
\end{assumption}


Under these assumptions the main result is the following a priori estimate.

\begin{theorem}[Regularity]\label{thm:main}
Assume that $\alpha\in(2/3,1)$ and 
\begin{equation}\label{eq:restrictionM}
M<1+\frac{3\alpha-2}{\alpha(2-\alpha)} \, ,
\end{equation}
and that $a$, $\sigma$, $\xi$, and $u_0$
satisfy Assumptions~\ref{ass:nonlinearities}, \ref{ass:noise}, and \ref{ass:initial} for this $M$ and $\alpha$. 
Let $u$ be the (unique, classical) solution of \eqref{eq:spm}.

Then there exist $R,T>0$ such that for $\|\y\|_\s\leq R$
\begin{equation}
\int_{D_\y} d\x\,
\big|u(\x+\y)-u(\x)\big| 
\lesssimdata \|\y\|_\s^{\alpha} \, ,
\end{equation}
where $D_\y\coloneqq\{\x\in [0,T]\times\T^d\,|\,\x+\y\in [0,T]\times\T^d\}$. 

Both $R,T$ and the implicit constant in $\lesssimdata$ depend on 
$d$, 
$\alpha$, 
$M$, 
$p$, 
$\{[\tau]_{|\tau|}\}_{\tau\in\mathcal{T}}$, 
and $C$ from Assumptions~\ref{ass:nonlinearities}, \ref{ass:noise}, and \ref{ass:initial},
but are uniform in the qualitative smoothness of $a,\xi,u_0$. 
The best\footnote{here and in the following, best is understood w.r.t.~$\|\y\|_\s\leq R$} constant (implicit in $\lesssimdata$) is denoted by $[u]_{B_{1,\infty}^{\alpha}}$.
\end{theorem}

Actually we shall prove the following higher order modelledness, 
of which Theorem~\ref{thm:main} is a consequence.

\begin{theorem}[Modelledness]\label{thm:modelledness}
Under the assumption of Theorem~\ref{thm:main}
and for $\beta\in(2-\alpha,\frac{2\alpha}{1+(M-1)\alpha})$, 
there exist $R,T>0$ such that for $\|\y\|_\s\leq R$
\begin{equation}
\int_{D_\y} \hspace{-1ex}d\x\, 
\big|u(\x+\y)-u(\x)-\sigma(u(\x))\Pi_\x[\lolly;a(u(\x))](\x+\y)-\nu(\x)\cdot y \big|
\lesssimdata \|\y\|_\s^\beta ,
\end{equation}
where $\Pi_\x[\lolly;\bar{a}]$ is defined in \eqref{eq:lolly} 
as the linear evolution of $\xi$ under the heat flow 
with diffusivity $\bar{a}$ and centered at $\x$, and where
$\nu\colon(0,T)\times\T^d\to\R^d$ is given by 
\begin{equation}\label{eq:gubinelliderivative}
\nu(\x) \coloneqq \nabla_z \big(u(\z) - u(\x) - \sigma(u(\x)) \Pi_\x[\lolly;\x](\z) \big)_{|\z=\x} \, .
\end{equation}
Both $R,T$ and the implicit constant in $\lesssimdata$ depend on $\beta$ and the data as in Theorem~\ref{thm:main}.
The best constant (implicit in $\lesssimdata$) is denoted by $[\mathcal{U}]_{B_{1,\infty}^\beta}$. 
\end{theorem}

These are the first a priori estimates establishing positive space-time regularity respectively modelledness for solutions of stochastic PDEs which are both singular and degenerate.
The regularity $\alpha$ obtained in Theorem~\ref{thm:main} is optimal for noise of regularity $\alpha-2$, 
however the $L^1$-integrability and the restrictions imposed on the nonlinearities $a$ and $\sigma$ are presumably not optimal. 
Let us also mention that the order of modelledness obtained in Theorem~\ref{thm:modelledness} is larger than $2-\alpha$ and thus larger than one, 
and that the interval $(2-\alpha,2\alpha/(1+(M-1)\alpha))$ of Theorem~\ref{thm:modelledness} is not empty by \eqref{eq:restrictionM}. 

\subsection{Strategy of the proof}\label{sec:strategy}

\noindent
\textbf{Kinetic formulation.}
To prove Theorem~\ref{thm:modelledness} we 
follow \cite{Ges21}, which established for the first time differentiability of order higher than one for solutions of the (deterministic) porous medium equation, and
appeal to the \emph{kinetic formulation} of \eqref{eq:spm}, 
i.e.~we consider the kinetic function $\chi\colon[0,T]\times\T^d\times\R\to\R$ given by 
\begin{equation}
\chi(t,x,v) = \1_{(-\infty,u(t,x))}(v) - \1_{(-\infty,0)}(v) \, ,
\end{equation}
made such that for $f\in C^1(\R)$ with $f(0)=0$
\begin{equation}
f(u(t,x)) = \int_\R dv\, f'(v) \chi(t,x,v) \, .
\end{equation}
The kinetic formulation of \eqref{eq:spm} is then
\begin{align}
&(\partial_t-a(v)\Delta)\chi(t,x,v) \\
&= \delta_{u(t,x)}(v) \big( \sigma(v) \xi(t,x)-\sigma'(v)\sigma(v)C^{a(v)}\big)
+ \partial_v \big( \delta_{u(t,x)}(v) a(v) |\nabla u(t,x)|^2 \big) \, , 
\end{align}
which has to be understood in a distributional sense, i.e.~for $\phi\in C_c^\infty((0,T)\times\T^d\times\R)$ 
%
\begin{align}
&\int_{[0,T]\times\T^d\times\R} dt\,dx\,dv\, 
\chi(t,x,v) (-\partial_t - a(v)\Delta)\phi(t,x,v) \\ 
&=\int_{[0,T]\times\T^d} dt\,dx\, \big( \sigma(u(t,x))\xi(t,x) 
- \sigma'(u(t,x))\sigma(u(t,x))C^{a(u(t,x))} \big) \phi(t,x,u(t,x)) \\
&\,- \int_{[0,T]\times\T^d} dt\,dx\, a(u(t,x)) |\nabla u(t,x)|^2 \partial_v\phi(t,x,u(t,x)) \, . \label{eq:kinetic}
\end{align}
An argument for equivalence of \eqref{eq:spm} and its kinetic formulation \eqref{eq:kinetic} is given in Appendix~\ref{app:kinetic}.

\begin{remark}[Terminology]\label{rem:language}
The kinetic formulation was introduced in \cite{LPT94} in the context of scalar conservation laws to obtain regularity results based on \emph{averaging}; 
the latter refers to quantities of the form $\int dv\,f(t,x,v)$. 
More references on subsequent works are given in the corresponding subsection below. 
The auxiliary variable $v$ was called \emph{velocity}, 
and the term $\delta_u(v)a(v)|\nabla u|^2$ is usually referred to as \emph{kinetic measure} in the literature. 
A function $u$ is called \emph{kinetic solution}, if its associated kinetic function $\chi$ satisfies \eqref{eq:kinetic} (with $=$ replaced by $\leq$). 
This notion of solution is more robust than the classical one, 
e.g.~the qualitative assumptions on $a$ and $u_0$ can be dropped. 
If such a solution is constructed by approximation with classical solutions, cf.~\cite[Theorem~3.2]{DGT20}, 
then the a priori estimate obtained in this work immediately transmits to this kinetic solution. 
\end{remark}

\medskip

\noindent
\textbf{Splitting into small and large velocities.}
The kinetic formulation allows by the representation $u=\int_\R dv\,\chi$ 
for a simple splitting of solutions into contributions from small velocities ($a(v)\lesssim\delta$), 
which can be estimated trivially in $L^1$ (cf.~Proposition~\ref{prop:u<}),
and large velocities ($a(v)\gtrsim\delta$),
the contributions of which are effectively not degenerate (cf.~Proposition~\ref{prop:u>}). 
This has been exploited in \cite{TT07} to obtain spatial regularity results for the (deterministic) porous medium equation, 
which was improved to obtain optimal spatial regularity in \cite{Ges21} 
and optimal space-time regularity in \cite{GST20}. 
The main task in these works is to obtain regularity estimates for the large velocity contributions and to trace the precise dependence on the non-degeneracy. 
As the approach is $L^1$-based, an $L^1$-forcing could be incorporated without much effort. 
This was extended to the stochastic porous medium equation for spatially coloured noise with trace-class covariance operator in \cite{BGW22} via It\^o calculus, which is naturally $L^2$-based. 

An extension to rough forcing of regularity $\alpha-2$ as in the present work comes with additional difficulties as we outline now. 
We follow \cite{BGW22} and appeal to the mild formulation of the kinetic formulation to estimate the large velocity contributions.
Neglecting the initial condition, these contributions are given by
\begin{align}
&\int_{(0,\infty)\times\R^d}\! d\z \, \Psi^>\big(a(u(\z)),\x-\z\big) 
\Big(\sigma(u(\z))\xi(\z) - (\sigma'\sigma)(u(\z))C^{a(u(\z))} \Big) \\
&\,- \int_{(0,\infty)\times\R^d}\! d\z\, 
\partial_{v} \big( \Psi^>(a(v),\x-\z) \big)_{|v=u(\z)} 
a(u(\z)) |\nabla u(\z)|^2 \, , \label{eq:ren_intro2}
\end{align}
where $\Psi^>(a,\cdot)$ vanishes for $a\leq\delta$ and coincides with $\Psi(a,\cdot)$ for $a\geq2\delta$, 
see Section~\ref{sec:mild} for details.
This expression
is potentially ill-defined for $\xi\in C^{\alpha-2}$ and $u$ with regularity $\alpha$.
Indeed, the singular forcing $\sigma(u)\xi-\sigma'(u)\sigma(u)C^{a(u)}$ has at most regularity $\alpha-2$, and is multiplied with $\Psi^>(a(u),\cdot)$ which inherits at most regularity $\alpha$ from $u$.
Furthermore, the kinetic measure contains the singular product $|\nabla u|^2$ 
(after making sense of $|\nabla u|^2$ by a renormalization procedure it is expected to have regularity $2\alpha-2$, hence the product with $\partial_v\Psi^>(a(u),\cdot)a(u)$ of regularity $\alpha$ is not problematic), 
which is the reason to include $\Pi[\derivativecherry]$ in \eqref{eq:derivativecherry}. 
However, the two singular products need the same (up to an integrable blowup) 
counterterm with opposite sign to stay under control, 
i.e.~their divergencies ``cancel'', 
as we outline after the following remark. 

\begin{remark}[Comparison to \cite{GH19}]
Of course, for non-constant $a$, convolution with $\Psi(a,\cdot)$ does not correspond to inverting the operator $\partial_t - a \Delta$. 
The mild formulation \eqref{eq:ren_intro2} via the kinetic formulation can be interpreted as considering this convolution anyway, and compensating the resulting error. 
A very similar idea was employed in \cite{GH19} to study a non-divergence form variant of \eqref{eq:spm} in a non-degenerate setting, 
with the subtle difference that integration with 
$\Psi(a(u(\x)),\x-\cdot)$ was considered instead of $\Psi(a(u(\cdot)),\x-\cdot)$. 
Their strategy has the advantage that products with the kernel are well defined (even for rough noise), whereas our strategy creates singular products 
(even for $\sigma\equiv1$ where \eqref{eq:spm} is not singular). 
However, the latter connects to the kinetic formulation and therefore allows for simple estimates of small velocity contributions.
\end{remark}

\medskip

\noindent
\textbf{Cancellations between counterterms.}
To make the necessary counterterms in \eqref{eq:ren_intro2} appear, we rewrite it as
\begin{align}
&\int d\z 
\Big[\Psi^>\big(a(v),\x-\z\big) \sigma(v) \xi(\z) - \partial_v\big( \Psi^>(a(v),\x-\z) \sigma(v) \big)
\sigma(v) \cdumb^{a(v)}(r) \Big]_{|v=u(\z)} \\
&\,- \int d\z \, 
\partial_v \big( \Psi^>(a(v),\x\!-\!\z) \big)_{|v=u(\z)}
a(u(\z)) \big( |\nabla u(\z)|^2 - \sigma^2(u(\z)) \ccherry^{a(u(\z))}(r) \big) \\
&\,+ \int d\z \, 
\partial_v \big( \Psi^>(a(v),\x - \z) \big)_{|v=u(\z)} 
\sigma^2(u(\z)) \big( \cdumb^{a(u(\z))}(r) - a(u(\z)) \ccherry^{a(u(\z))}(r)\big) \\
&\,+ \int d\z \, 
\Psi^>(a(u(\z)),\x - \z) (\sigma'\sigma)(u(\z)) \big( \cdumb^{a(u(\z))}(r)-C^{a(u(\z))} \big) \, .
\label{eq:ren_intro}
\end{align}
A heuristic argument explaining the choice of counterterm can be found in Appendix~\ref{app:heuristics_counterterm}.
The time dependence of the counterterms $\cdumb^a$ and $\ccherry^a$ in \eqref{eq:dumb} and \eqref{eq:derivativecherry} is a result of choosing $\Pi[\lolly]$ in \eqref{eq:lolly} with initial condition zero and avoiding weights in the estimates \eqref{eq:modelbound} and \eqref{eq:lollybound}. 
This is convenient in the proof later, 
but would cause a time dependent counterterm in \eqref{eq:spm} which is not desirable. 
In applications we have in mind (see Examples~\ref{ex:SWN} and \ref{ex:CN} below), 
the difference of the time dependent and the time independent counterterm has only a mild blowup near time zero,
which is integrable and thus the reason for the assumption \eqref{eq:counterterm_time_independent}, taking care of the last line of \eqref{eq:ren_intro}, 
see Proposition~\ref{prop:u3} and its proof in Section~\ref{sec:initial}. 
Similarly, if $\Pi[\lolly]$ were chosen stationary, then we could choose $\cdumb^{a(u)}=a(u)\ccherry^{a(u)}$. 
The zero initial condition causes again a mild blow up at time zero, which is however integrable and thus the reason for the assumption \eqref{eq:counterterm_dumb_cherry}, taking care of the third line of \eqref{eq:ren_intro}.
This relation between the counterterms is not surprising in view of 
\begin{align}
(\partial_t-a\Delta)\tfrac{1}{2}\Pi_\x[\lolly;a]^2
&= \Pi_\x[\lolly;a] (\partial_t-a\Delta )\Pi_\x[\lolly;a]
- a |\nabla\Pi_\x[\lolly;a]|^2 \\
&= \Pi_\x[\lolly;a] \xi
- a |\nabla\Pi_\x[\lolly;a]|^2 \, .
\end{align}
Since the left hand side is well defined as a distribution, 
a cancellation of the counterterms required to define each individual singular product on the right hand side is expected. 

\medskip

\noindent
\textbf{Reconstruction and integration with rough kernels.} 
Due to the limited regularity of the kernel $\Psi^>(a(u),\cdot)$ through composition with the solution $u$, 
a novel combination of what in regularity structures is called reconstruction and integration
is required to estimate \eqref{eq:ren_intro}. 
Section~\ref{sec:roughkernels} contains a number of abstract reconstruction and integration results which are tailored to our application to rough (through composition with the solution) kernels, 
and which trace precisely the dependence on the diffusion coefficient. 
Compensating possible blowups by $\sigma$ allows to estimate the first line of \eqref{eq:ren_intro}, see Proposition~\ref{prop:u1} and its proof in Section~\ref{sec:singularforcing}. 
This can be seen as a first main contribution of the present work. 

\begin{remark}[Restriction of $N$]\label{rem:restriction_N}
In fact, we need to assume that $\sigma,\sigma',\sigma''$ vanish sufficiently fast near zero to compensate certain blowups of $a^{-1},a',a''$ near zero. 
Since the list of necessary conditions is relatively long, 
see equations \eqref{eq:I} -- \eqref{eq:XVI}, 
the slightly stronger but much simpler Assumption~\ref{ass:nonlinearities} is made. 
In particular $N\geq M+1$ is sufficient by the upper bound on $M$ imposed in Theorem~\ref{thm:main}, which implies in particular $M<2$. 
Furthermore, the quantitative assumption on $a$ is only required on the support of $\sigma$.
The exact growth/vanishing conditions required from $a$ and $\sigma$ can be easily read off the proof. 
Let us also mention that the restriction $N\geq M+1$ might not be optimal, 
and a more careful analysis could allow to relax this assumption. 
\end{remark}

\medskip

\noindent
\textbf{Renormalized kinetic measure.}
A second main contribution is to estimate the renormalized kinetic measure appearing in the second line of \eqref{eq:ren_intro}, 
see Proposition~\ref{prop:u2} and its proof in Section~\ref{sec:energy}. 
Here we appeal to a novel energy inequality for singular equations in combination with reconstruction as we heuristically outline now; see Proposition~\ref{prop:energy} for a precise statement.
Naively testing\footnote{In Proposition~\ref{prop:energy} we rather test with $|u|^\epsilon$ for $\epsilon>0$, which only makes use of $\|u_0\|_{L^{1+\epsilon}(\T^d)}$ instead of $\|u_0\|_{L^2(\T^d)}$ as in the informal discussion here. As in \cite{GST20}, this can possibly be further optimized to $\|u_0\|_{L^1(\T^d)}$.} \eqref{eq:spm} with $u$, integrating over $[0,T]\times\T^d$, and integrating by parts, yields
\begin{align}
&\frac12 \int_{\T^d} dx\, \big( u(T,x)^2 - u(0,x)^2 \big) 
+ \int_{[0,T]\times\T^d}dtdx\, a(u)|\nabla u|^2 \\
&= \int_{[0,T]\times\T^d} dtdx\,u\big(\sigma(u)\xi-\sigma'(u)\sigma(u)C^{a(u)}\big) \, .
\end{align}
As mentioned above,
$|\nabla u|^2$ is a singular product, and including its counterterm
\begin{equation}
\int_{[0,T]\times\T^d} dtdx\, a(u)\big(|\nabla u|^2-\sigma(u)^2 \ccherry^{a(u)}\big) \, ,
\end{equation}
we lose the good sign of the term. 
However, plugging in the definition \eqref{eq:gubinelliderivative} of $\nu$ we deduce 
\begin{align}
&\frac{1}{2}\int_{\T^d}dx\, u(T,x)^2 
+ \int_{[0,T]\times\T^d}dtdx\, a(u)|\nu|^2 \\
&= \frac{1}{2}\int_{\T^d}dx\, u(0,x)^2 
+\int_{[0,T]\times\T^d} dtdx \Big( 
u\big(\sigma(u)\xi
- \sigma'(u)\sigma(u) C^{a(u)} \big) \\
&\hphantom{= \frac{1}{2}\int_{\T^d}dx\, u(0,x)^2 
+\int_{[0,T]\times\T^d} dtdx}
-2 a(u)\sigma(u)\nabla\Pi[\lolly]\cdot \nu
- a(u) \sigma(u)^2 |\nabla\Pi[\lolly]|^2 \Big) \, .
\end{align}
The left hand side has again a good sign, 
and the right hand side can be estimated by leveraging once more that $u(\sigma(u)\xi-\sigma'(u)\sigma(u)C^{a(u)})$ requires the same counterterm as $a(u)\sigma(u)^2|\nabla\Pi[\lolly]|^2$.
In conclusion $a(u)|\nu|^2$ can be estimated and consequently also 
\begin{align}
&a(u)\big(|\nabla u|^2-\sigma(u)^2\ccherry^{a(u)}\big) \\
&= a(u)|\nu|^2 + 2 a(u)\sigma(u)\nabla\Pi[\lolly]\cdot \nu + a(u)\sigma(u)^2\big(|\nabla\Pi[\lolly]|^2-\ccherry^{a(u)}\big) \, .
\end{align}

\medskip

\noindent
\textbf{Real interpolation.}
Combining via real interpolation the regularity estimates from both small and large velocity contributions obtained in Section~\ref{sec:splitting}, 
yields estimates of the solution $u$ itself, see Section~\ref{sec:interpolation}. 
However, other than in the existing literature on kinetic solutions/averaging techniques, 
for large velocity contributions we obtain modelledness instead of a plain regularity estimate. 
The method of real interpolation has therefore to be extended to basepoint dependent seminorms arising in the context of regularity structures, which is a third contribution of this work. 

\begin{remark}[Restriction of $M$]\label{rem:restrictionM}
To prove Theorem~\ref{thm:modelledness} we split $u$ into two components $u^<$ and $u^>$ depending on a parameter $\delta>0$. 
The former is estimated in $L^1$ and decays as $\delta^{1/(M-1)}$, 
whereas the latter is modelled to order $2\alpha$ and blows up as $\delta^{-\alpha-\epsilon}$ for any $\epsilon>0$. 
By real interpolation we conclude that $u$ is modelled to order $2\alpha/(1+(M-1)(\alpha+\epsilon))$, 
explaining the upper bound on $\beta$ in Theorem~\ref{thm:modelledness}. 

This is combined with a buckling argument: 
assuming modelledness of $u$, modelledness of $\sigma(u)\xi$ is established, 
which is then ``integrated'' (informally by the application of $(\partial_t-a(u)\Delta)^{-1}$) to obtain again modelledness of $u$. 
To achieve modelledness of $\sigma(u)\xi$, the order of modelledness of $u$ (and hence $\sigma(u)$) has to be sufficiently high so that the regularity $\alpha-2$ of $\xi$ adds up positive, explaining the lower bound on $\beta$ in Theorem~\ref{thm:modelledness}. 

Combining both restrictions enforces \eqref{eq:restrictionM}. 
It is conceivable that this restriction can be relaxed when considering modelledness of $A(u)=\int_0^u dv\, a(v)$ instead of $u$. 
In this case the corresponding $A(u)^<$ would be estimated in $L^1$ and decaying as $\delta^{M/(M-1)}$, and one can hope to model $A(u)^>$ to order $<2$ and blowing up as $\delta^{-\epsilon}$ for any $\epsilon>0$. 
The resulting modelledness of $A(u)$ would be just below $2$, 
which is larger than $2-\alpha$ provided $\alpha>0$.
This strategy requires significantly more effort, and we defer it to future work.

Let us also mention here that $u^<$ can be estimated in $L^\infty$ by $\delta^{1/(M-1)}$  as well. 
This can be used to improve the integrability in Theorem~\ref{thm:modelledness} from $L^1$ to $L^q$ for any $q<1+\alpha(M-1)$.
For simplicity we refrain from doing so, on the one hand to avoid introducing Lorentz spaces, and on the other hand since working in $L^1$ has the pleasant side effect that certain singularities at time $t=0$ are integrable so that weights can be avoided. 
\end{remark}

\subsection{Examples}

\begin{example}[Porous medium equation]\label{ex:pm}
The prime example we have in mind is the porous medium equation corresponding to $a(v)=M|v|^{M-1}$. 
This $a$ satisfies the quantitative Assumption~\ref{ass:nonlinearities}~(ii), 
but it does not satisfy the qualitative part (i) for $M<2$. 
This can be remedied by replacing $a$ by $a_{\epsilon}$ defined as follows.
For $|v|\geq\epsilon$ choose $a_{\epsilon}$ to coincide with $a$. 
For $|v|\leq\epsilon$ choose $a_{\epsilon}$ to be the quartic polynomial $p(v)=p_4 v^4+p_2v^2+p_0$ with 
$p_4=M(M-1)(M-3)\epsilon^{M-5}/8$, 
$p_2=M(M-1)(5-M)\epsilon^{M-3}/4$, and 
$p_0=M(1-(M-1)(7-M)/8)\epsilon^{M-1}$. 
The coefficients are chosen such that $a_{\epsilon}\in C^2(\R)$. 
Furthermore, it is possible to check that 
$a_{\epsilon}\geq p_0>0$, 
$a_{\epsilon}\geq a$, 
$|a'_{\epsilon}|\leq|a'|$, and
$|a''_{\epsilon}|\lesssim|a''|$, 
see Figure~\ref{fig:pm}. 
Thus $a_{\epsilon}$ satisfies Assumption~\ref{ass:nonlinearities}, 
and moreover it converges uniformly to $a$. 
The main result can be applied and the derived a priori estimate is uniform in $\epsilon>0$.
\hfill /\!\!/
\end{example}

\begin{figure}[h]
\centering
\includegraphics{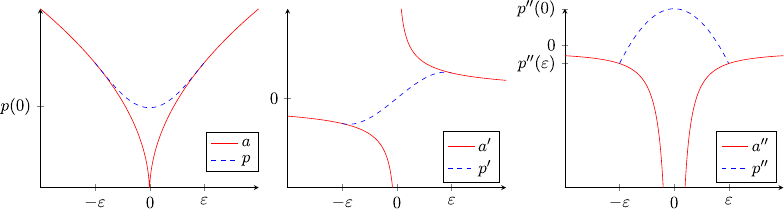}
\caption{Visualization of $a_\epsilon$ of Example~\ref{ex:pm} for $a(v)=M|v|^{M-1}$ with $M=3/2$.}
\label{fig:pm}
\end{figure}

\begin{example}[White noise on $\T^2$]\label{ex:SWN}
Consider space white noise on $\T^2$, which we realize as the random Fourier series 
\begin{equation}
\xi = \sum_{k\in(2\pi\Z)^2} e_k \hat\xi_k \, ,
\end{equation}
where $e_k(x)=e^{i k\cdot x}$ and $\hat\xi_k$ are independent (up to the usual restriction that the complex conjugate of $\hat\xi_k$ coincides with $\hat\xi_{-k}$)
complex-valued centered Gaussians with variance $1$. 
In this example $k$ is always assumed to be an element of $(2\pi\Z)^2$. 
Of course, $\xi$ is not smooth and thus does not satisfy the qualitative Assumption~\ref{ass:noise}~(i). 
This can be fixed by replacing $\xi$ by the smooth $\xi_\epsilon=\sum_{|k|\leq\epsilon^{-1}}e_k\hat\xi_k$. 
It is then easy to check that \eqref{eq:modelbound} holds for $\tau\in\{\noise,\xnoise\}$ for any $\alpha<1$ uniformly in $\epsilon>0$. 
The corresponding $\Pi[\lolly]$ is given by 
\begin{equation}
\Pi_{(t,x)}[\lolly;a](s,y)
= (s-t) \hat\xi_0
+ \sum_{0<|k|\leq\epsilon^{-1}} 
\Big( e_k(y) \frac{1-e^{-a|k|^2s}}{a|k|^2} 
- e_k(x) \frac{1-e^{-a|k|^2t}}{a|k|^2} \Big) \hat\xi_k \, .
\end{equation}
One can verify \eqref{eq:lollybound}, e.g.~for spatial increments and $\bar\epsilon>0$
\begin{align}
\E|\Pi_{(t,x)}[\lolly;a](t,y)|^2 
&= \sum_{0<|k|\leq\epsilon^{-1}} |e_k(x)-e_k(y)|^2 \Big(\frac{1-e^{-a|k|^2t}}{a|k|^2}\Big)^2 \\
&\lesssim \sum_{0<|k|\leq\epsilon^{-1}} \frac{|k|^{2-\bar\epsilon}|x-y|^{2-\bar\epsilon}}{a^2|k|^4} \, ,
\end{align}
where in the inequality we interpolated between $|e_k(x)-e_k(y)|\lesssim1$ and $|e_k(x)-e_k(y)|\lesssim|k||x-y|$. 
One can argue similarly for time increments and derivatives with respect to $a$, 
so that equivalence of moments together with Kolmogorov's continuity theorem yield for compact $\mathfrak{K}\subseteq\R\times\R^d$ and $\bar\epsilon>0$
\begin{equation}
\E\Big| \sup_{\x\neq\y\in\mathfrak{K}} \sup_{\bar a\in a(\mathrm{supp}\,\sigma)} \bar a^{1+\bar\epsilon} \frac{\Pi_{\x}[\lolly;\bar a](\y)}{\|\x-\y\|_\s^{1-\bar\epsilon}}
\Big|^p <\infty \, ,
\end{equation}
and thus almost surely \eqref{eq:lollybound}.

The counterterm $\cdumbeps^{a}(t)$ can be chosen as $t+\sum_{0<|k|\leq\epsilon^{-1}}(1-e^{-a|k|^2t})/(a|k|^2)$
such that $\Pi[\dumb]$ without recentering has vanishing expectation, 
where the time dependence of the counterterm follows from the zero initial condition of $\Pi[\dumb]$. 
With this choice one can show that $\Pi[\dumb]$ satisfies \eqref{eq:modelbound}; 
we refer the reader to \cite[Section~3.3.1]{dLF24} for further details. 
Furthermore, choosing $C_\epsilon^a=\sum_{0<|k|\leq\epsilon^{-1}}(a|k|^2)^{-1}$, 
we have for $v\in\R$ 
\begin{align}
|\sigma'(v)\sigma(v) (C_\epsilon^{a(v)} - \cdumbeps^{a(v)}(t)) |
&\leq |\sigma'(v)\sigma(v)| \Big( t + \sum_{0<|k|\leq\epsilon^{-1}}\frac{e^{-a(v)|k|^2t}}{a(v)|k|^2} \Big) \\
&\leq |\sigma'(v)\sigma(v)| \Big(t+\sum_{0<|k|\leq\epsilon^{-1}}\frac{1}{a(v)^{3/2}|k|^3\sqrt{t}} \Big) \, ,
\end{align}
where in the second inequality we used $e^{-x}\leq x^{-1/2}$.
Using Assumption~\ref{ass:nonlinearities}~(ii) and \eqref{eq:restrictionM} the above is estimated by $t + \sum_{0<|k|\leq\epsilon^{-1}}(|k|^3\sqrt{t})^{-1}$, 
which is integrable in $t$ from $0$ to $T$ and hence \eqref{eq:counterterm_time_independent} is satisfied. 

Similarly $\ccherryeps^a(t)$ can be chosen as $\sum_{0<|k|\leq\epsilon^{-1}} (1-e^{-a|k|^2t})^2/(a^2|k|^2)$ such that $\Pi[\derivativecherry]$ has vanishing expectation.
Then $\Pi[\derivativecherry]$ satisfies \eqref{eq:modelbound}, and
\begin{align}
\cdumbeps^{a(v)}(t)-a(v)\ccherryeps^{a(v)}(t) 
&= t + \sum_{0<|k|\leq\epsilon^{-1}} \Big(
\frac{1-e^{-a(v)|k|^2t}}{a(v)|k|^2}
- a(u) \frac{(1-e^{-a(v)|k|^2t})^2}{a(v)^2|k|^2} \Big) \\
&= t + \sum_{0<|k|\leq\epsilon^{-1}} 
\frac{e^{-a(v)|k|^2t}(1-e^{-a(v)|k|^2t})}{a(v)|k|^2} \\
&\leq t + \sum_{0<|k|\leq\epsilon^{-1}} 
\frac{e^{-a(v)|k|^2t}}{a(v)|k|^2} \, ,
\end{align}
and we conclude as for $C_\epsilon^a - \cdumbeps^a(t)$ that \eqref{eq:counterterm_dumb_cherry} holds. 

Altogether Assumption~\ref{ass:noise} is satisfied for $\xi_\epsilon$, and the main result can be applied to obtain a priori estimates which are uniform in $\epsilon>0$. 
\hfill /\!\!/
\end{example}

Actually, the assumed estimates \eqref{eq:modelbound} and \eqref{eq:lollybound} are a worst case scenario in terms of blowup in $\bar a$. 
This is easiest seen by considering noise which is rough in time and constant in space, e.g.~temporal white noise. In this case, $\Pi_{(t,x)}[\lolly;\bar a](s,y)=\int_t^s dr \, \xi(r)$, 
which is even independent of $\bar a$, i.e.~$\mathfrak{e}(\tau)$ could be replaced by $0$ in \eqref{eq:lollybound}. 
The following example considers a situation somewhat in between temporal noise and spatial noise as considered in the previous example, 
where $\mathfrak{e}(\tau)$ can be replaced by $\mathfrak{e}(\tau)/2$. 
Unfortunately, this improved estimate does not allow to weaken the assumption $N\geq M+1$, 
which is in particular made for equation \eqref{eq:IV} to hold, 
which in turn is independent of $\mathfrak{e(\tau)}$. 
As already mentioned in Remark~\ref{rem:restriction_N}, 
a more careful analysis might allow to take advantage of an improved \eqref{eq:modelbound} to relax the assumption on $N$. 

\begin{example}[Spatially coloured noise]\label{ex:CN}
Consider $\xi$ on $\R\times\T^d$ which is white in time and sufficiently coloured in space realized via
\begin{equation}
\xi = \sum_{k\in(2\pi\Z)^d} \sqrt{c_k} e_k dB_k \, ,
\end{equation}
where in analogy to the previous example $k$ will always be assumed to belong to $(2\pi\Z)^d$ and $e_k(x)=e^{ik\cdot x}$, and where 
\begin{equation}
c_k = \frac{1}{(1+|k|)^{d-2+2\alpha'}}
\end{equation}
for $\alpha'\in(2/3,1)$, and $B_k(t)$ are independent (again up to $\bar{B}_k=B_{-k}$)
complex-valued Brownian motions. 
Here a Fourier cutoff is not enough to obtain a smooth approximation of $\xi$ due to the roughness of $dB_k$ in time. 
Hence we also fix a compactly supported $\rho\colon\R\to[0,1]$ with $\int\rho=1$ which is even. Denoting $\rho_\epsilon=\epsilon^{-1}\rho(\cdot/\epsilon)$ and $B^{(\epsilon)} = \rho_\epsilon*B$ we approximate $\xi$ by the smooth $\xi_\epsilon=\sum_{|k|\leq\epsilon^{-1}} \sqrt{c_k} e_k dB^{(\epsilon)}_k$. 
Then \eqref{eq:modelbound} holds for $\tau\in\{\noise,\xnoise\}$ for any $\alpha<\alpha'$ uniformly in $\epsilon>0$. 
The corresponding $\Pi[\lolly]$ is given by 
\begin{equation}
\begin{split}
&\Pi_{(t,x)}[\lolly;a](s,y) \\
&= \sum_{|k|\leq\epsilon^{-1}}\sqrt{c_k} \Big( 
e_k(y) \int_0^s e^{-a|k|^2(s-r)} dB^{(\epsilon)}_k(r)
- e_k(x) \int_0^t e^{-a|k|^2(t-r)} dB^{(\epsilon)}_k(r) \Big) \, .
\end{split}
\end{equation}
For spatial increments 
\begin{align}
&\E|\Pi_{(t,x)}[\lolly;a](t,y)|^2 \\
&= \sum_{|k|\leq\epsilon^{-1}} c_k |e_k(y)-e_k(x)|^2 
\int_0^t dr \int_0^t dr'\, e^{-a|k|^2(t-r)} e^{-a|k|^2(t-r')} \rho_\epsilon*\rho_\epsilon(r-r') \\
&\leq \sum_{|k|\leq\epsilon^{-1}} c_k |e_k(y)-e_k(x)|^2 
\int_0^t dr \, e^{-2a|k|^2(t-r)}
\int_\R dr \, |\rho_\epsilon*\rho_\epsilon(r)| \\
&= \sum_{0<|k|\leq\epsilon^{-1}} c_k |e_k(y)-e_k(x)|^2 \, 
\frac{1-e^{-2a|k|^2t}}{2a|k|^2} \, ,
\end{align}
where in the inequality we have used Young's inequality for convolutions in the form of $|\int\int f(r)g(r')h(r-r')| \leq \|f\|_{L^2} \|g\|_{L^2} \|h\|_{L^1}$. 
Interpolating between $|e_k(x)-e_k(y)|\lesssim1$ and $|e_k(x)-e_k(y)|\lesssim|k||x-y|$ the above right hand side is bounded by 
\begin{equation}
\sum_{0<|k|\leq\epsilon^{-1}} |x-y|^{2\alpha} a^{-1} |k|^{-d-2\alpha'+2\alpha} \, ,
\end{equation}
which is estimated uniformly in $\epsilon>0$ by $|x-y|^{2\alpha} a^{-1}$ provided $\alpha<\alpha'$. 
One can proceed similarly with time increments and derivatives with respect to $a$, which together with equivalence of moments and Kolmogorov's continuity theorem proves \eqref{eq:lollybound} as in Example~\ref{ex:SWN}, but with $\mathfrak{e}(\lolly)$ replaced by $\mathfrak{e}(\lolly)/2$.

The counterterm $\cdumbeps^a(t)$ can be chosen as $\sum_{|k|\leq\epsilon^{-1}} c_k \int_0^t dr\,e^{-a|k|^2r}\rho_\epsilon*\rho_\epsilon(r)$ 
such that $\Pi[\dumb]$ without recentering has vanishing expectation.
If we ignored the qualitative smoothness assumption on $\xi$ and worked with $dB_k$ without mollification, 
we could construct $\Pi[\dumb]$ as It\^o integral, and the corresponding counterterm would be zero. 
Our choice of counterterm can therefore be seen to be the It\^o-Stratonovich correction.
In any case, $\Pi[\dumb]$ satisfies \eqref{eq:modelbound} for any $\alpha<\alpha'$; 
we refer again to \cite[Section~3.2.2]{dLF24} for details. 
Choosing $C_\epsilon^a= \sum_{|k|\leq\epsilon^{-1}} c_k \int_0^\infty dr\, e^{-a|k|^2r}\rho_\epsilon*\rho_\epsilon(r)$ we have for $v\in\R$
\begin{align}
|\sigma'(v)\sigma(v) (C_\epsilon^{a(v)}-\cdumbeps^{a(v)}(t))|
&= |\sigma'(v)\sigma(v)| \sum_{|k|\leq\epsilon^{-1}} c_k \int_t^\infty dr\, e^{-a(v)|k|^2r}\rho_\epsilon*\rho_\epsilon(r) \\
&\leq |\sigma'(v)\sigma(v)| \sum_{|k|\leq\epsilon^{-1}} c_k e^{-a(v)|k|^2t} \, .
\end{align}
Using $e^{-x}\leq x^{-1/2}$ and the definition of $c_k$ this is bounded by 
$|\sigma'(v)\sigma(v)|(1+\sum_{0<|k|\leq\epsilon^{-1}} |k|^{-d+1-2\alpha'} a(v)^{-1/2} t^{-1/2})$, 
which by Assumption~\ref{ass:nonlinearities}~(ii), \eqref{eq:restrictionM}, and $\alpha'>1/2$ is bounded uniformly in $\epsilon>0$ by $1+t^{-1/2}$. Since the latter is integrable in $t$ from $0$ to $T$ this verifies \eqref{eq:counterterm_time_independent}.

Similarly we choose $\ccherryeps^{a}(t)=\sum_{|k|\leq\epsilon^{-1}} c_k |k|^2 \int_0^tdr\int_0^tdr'\,e^{-a|k|^2r}e^{-a|k|^2r'}\rho_\epsilon*\rho_\epsilon(r-r')$ 
so that $\Pi[\derivativecherry]$ satisfies \eqref{eq:modelbound} for any $\alpha<\alpha'$, and 
\begin{align}
&\cdumbeps^{a(v)}(t) - a(v)\ccherryeps^{a(v)}(t) \\
&= \sum_{|k|\leq\epsilon^{-1}} \! c_k\! \int_0^t \! dr\, e^{-a(v)|k|^2r}\Big(
\rho_\epsilon\!*\!\rho_\epsilon(r)
- a(v)|k|^2\! \int_0^t \! dr'\, e^{-a(v)|k|^2r'}\rho_\epsilon\!*\!\rho_\epsilon(r-r')\Big) \, .
\end{align}
By integration by parts we observe that this equals
\begin{equation}
\sum_{|k|\leq\epsilon^{-1}} \! c_k \! \int_0^t\! dr\, e^{-a(v)|k|^2r}\Big(
e^{-a(v)|k|^2t}\rho_\epsilon*\rho_\epsilon(r-t) - \int_0^t \! dr' e^{-a(v)|k|^2r'}\tfrac{d}{dr'}\rho_\epsilon*\rho_\epsilon(r-r') \Big) .
\end{equation}
The last summand vanishes by symmetry (recall that $\rho$ and hence $\rho_\epsilon*\rho_\epsilon$ is even and its derivative thus odd), so that
\begin{equation}
|\cdumbeps^{a(v)}(t) - a(v)\ccherryeps^{a(v)}(t) |
\leq \sum_{|k|\leq\epsilon^{-1}} c_k e^{-a(v)|k|^2t} \, ,
\end{equation}
and we conclude as for $C_\epsilon^a-\cdumbeps^a(t)$ that \eqref{eq:counterterm_dumb_cherry} holds. 

Altogether Assumption~\ref{ass:noise} is satisfied for $\xi_\epsilon$ and with $\mathfrak{e}(\tau)$ replaced by $\mathfrak{e}(\tau)/2$, 
and the main result can be applied to obtain a priori estimates which are uniform in $\epsilon>0$. 
\hfill /\!\!/
\end{example}

\subsection{Related literature}\label{sec:literature}

On the one hand, \eqref{eq:spm} has been studied extensively in the degenerate case for noise which is white in time and coloured in space with trace-class covariance operator (such that the equation is not singular) using It\^o calculus. 
Early results are based on the monotone operator approach introduced in \cite{Par75,KR79}, 
which covers affine $\sigma$. 
However also additive space-time white noise can be dealt with \cite{BDPR16}, 
in which case the equation is again not singular; 
the monograph \cite{BDPR16} contains a modern account and further references to this approach. 
More recent results 
include \cite{DHV16,GH18} for increasingly more general $a$ and $\sigma$, 
which are based on averaging techniques and the notion of kinetic solution 
as described in Remark~\ref{rem:language}.
A larger class of $a$ and $\sigma$ is considered in \cite{DGG19,DGT20} based on the closely related notion of entropy solution. 
We refer the reader to \cite{DGG19} for a more extensive overview of the literature on this line of research. 
All of these works obtain existence and uniqueness of solutions 
(for the respective notions of solutions)
of relatively low regularity. 
Optimal regularity for kinetic solutions has been established in \cite{BGW22}.
This work and its predecessors \cite{GST20} and \cite{Ges21} on the deterministic porous medium equation, 
which are refinements of \cite{TT07}, 
are a main source of inspiration for the present work.

On the other hand, \eqref{eq:spm} and non-divergence form variants of it have been studied with pathwise techniques in the singular regime (i.e.~$\alpha<1$)
since the advent of regularity structures \cite{Hai14} and paracontrolled distributions \cite{GIP15}, 
however exclusively in the non degenerate setting where $a$ is uniformly bounded away from zero. 
First works \cite{OW19,FG19,BDH19} dealt with $\alpha>2/3$ (as the present work), 
which was then generalized to $\alpha>1/2$ in \cite{GH19,OSSW18}, 
$\alpha>2/5$ in \cite{BM23}, 
$\alpha>1/3$ in \cite{Ger20}, 
and subsequently to the full sub-critical regime $\alpha>0$ in \cite{BHK24}, 
\cite{BGN24,BD24}, and \cite{OSSW25,LOT23,LOTT24,Tem24}; 
for a more concise overview of this line of research we refer the interested reader to \cite{LOTT24}.

Let us finally mention the work \cite{DGG21} which deals with \eqref{eq:spm} in a degenerate \emph{and} singular regime. 
Existence of weak (in the probabilistic and PDE sense) solutions is proven for space-time white noise in $1+1$ dimension under fairly general assumptions on $a$ and $\sigma$.
While we impose certainly more restrictions on $a$ and $\sigma$, see however Remark~\ref{rem:restrictionM}, 
and consider more regular noise 
(space-time white noise in $1+1$ dimension corresponds to $\alpha<1/2$),
the notion of solution we work with and the obtained regularity results are stronger. 
Furthermore, our analysis does not rely on a white-in-time correlation structure of the noise. 

%

\subsection{Notation}\label{sec:notation}

$\mathbb{N}$ denotes the natural numbers including zero,
$d$ is the space dimension. 
For $k\in\N^d$ and $x\in\R^d$ we write $x^k=x_1^{k_1}\cdots x_d^{k_d}$ and $k!=k_1!\cdots k_d!$. 
For $k\in\N^d$ we write $|k|_1$ for the $\ell^1$ norm $k_1+\cdots+k_d$. 
The Euclidean norm in any dimension is denoted by $|\cdot|$. 
Space-time points are always denoted by $\x=(t,x),\,\y=(s,y),\,\z=(r,z)\in\R\times\R^d$.
Distance of space-time points is measured with respect to the parabolic Carnot-Carath\'eodory metric $\|(t,x)\|_\s\coloneqq\sqrt{|t|+x_1^2+\cdots+x_d^2}$, 
and the effective dimension of space-time equipped with this metric is $D\coloneqq d+2$. 
For $\lambda>0$ we write $\S^\lambda$ for the parabolic rescaling 
$\S^\lambda(t,x)\coloneqq(\lambda^2t, \lambda x)$, 
and note that $\|\S^\lambda(t,x)\|_\s = \lambda \|(t,x)\|_\s$. 
For $\bar{a}>0$ we write $\D^{\bar a}$ for the anisotropic dilation $\D^{\bar{a}}(t,x)\coloneqq(\bar{a}t,x)$. 
Recentered and rescaled space-time functions $\varphi$ are denoted by 
$\varphi_\x^\lambda(\y)=\lambda^{-D}\varphi(\S^\lambda(\y-\x))$.
$\mathcal{B}$ is the set of smooth functions supported in the (parabolic) unit ball with 
all its derivatives up to order $2$ bounded by $1$. 
Unless stated otherwise, $\|\cdot\|_{L^p}$ is considered on $[0,T]\times\T^d$, 
and we simply write $\|\cdot\|_\infty$ to denote $\|\cdot\|_{L^\infty(\R)}$. 
For $\alpha\in(0,1)$ 
\begin{equation}
[u]_{B_{1,\infty}^\alpha} 
= \sup_{\|\y\|_\s\leq R} \|\y\|_\s^{-\alpha} \int_{D_\y} d\x\,|u(\x+\y)-u(\x)|\, ,
\end{equation}
where $D_\y=\{\x\in [0,T]\times\T^d\,|\,\x+\y\in [0,T]\times\T^d\}$, 
and $[f(\mathcal{U})]_{B_{1,\infty}^\beta}$ is for $\beta\in(1,2)$ given by the supremum over $\|\y\|_\s\leq R$ of
\begin{equation}
\|\y\|_\s^{-\beta}\! \int_{D_{\y}} d\x \, 
\big| f(u(\x+\y)) - f(u(\x)) - f'(u(\x))\big(\sigma(u(\x))\Pi_{\x}[\lolly;\x](\x+\y) + \nu(\x)\cdot y\big) \big| \, ,
\end{equation}
where $\Pi_\x[\lolly;\bar \x]$ is short for $\Pi_\x[\lolly;a(u(\bar \x))]$ defined in \eqref{eq:lolly} and $\nu$ is defined in \eqref{eq:gubinelliderivative}. 
The $L^2$ inner product on $\R\times\R^d$ is denoted by $\langle\cdot,\cdot\rangle$, 
and for testing on $[0,\infty)\times\R^d$ we also write $F(\varphi)=\int_{[0,\infty)\times\R^d} d\x \,F(\x)\varphi(\x)$.
By $\lesssim$ we understand $\leq C$ for a generic constant $C$ unless stated otherwise, 
and the implicit constant in $\lesssimdata$ depends on 
$d$, 
$\alpha$, 
$\beta$, 
$M$, 
$p$, 
$\{[\tau]_{|\tau|}\}_{\tau\in\mathcal{T}}$, 
and $C$ from Assumptions~\ref{ass:nonlinearities}, \ref{ass:noise}, and \ref{ass:initial}.

\section{Splitting into small and large velocities}\label{sec:splitting}

To prove Theorem~\ref{thm:modelledness} we treat the contributions to $u=\int dv\, \chi$ 
from small and large velocities separately. 
In the proof it will be convenient to avoid sharp cutoffs. 
We thus fix a smooth 
$\varphi^>\colon[0,\infty)\to[0,1]$ which vanishes on $[0,1]$ and equals $1$ on $[2,\infty)$, 
and we denote $\varphi^<\coloneqq 1-\varphi^>$.
Then for $\delta>0$ we decompose
\begin{equation}\label{eq:splitu}
u(t,x)=u^<(t,x)+u^>(t,x) 
\coloneqq 
\int_\R dv\, \varphi^<(\tfrac{a(v)}{\delta}) \chi(t,x,v) 
+ \int_\R dv\, \varphi^>(\tfrac{a(v)}{\delta}) \chi(t,x,v) \, ,
\end{equation}
and estimate $u^<$ and $u^>$ separately. 
We start with the former, which is estimated only in $L^1$ but with good decay as $\delta\to0$. 

\begin{proposition}[Small velocities]\label{prop:u<}
Under the assumption of Theorem~\ref{thm:main} and for $T,\delta>0$ it holds
\begin{equation}
\int_{[0,T]\times\T^d} d\x \, |u^<(\x)| \lesssim T \, \delta^{1/(M-1)} \, ,
\end{equation}
where $\lesssim$ depends on the constant $C$ of Assumption~\ref{ass:nonlinearities}.
\end{proposition}

\begin{proof}
By definition $0\leq\varphi^< \leq\1_{[0,2]}$ and $|\chi|\leq1$, hence
\begin{equation}
\int_{[0,T]\times\T^d} d\x \, | u^<(\x) | 
= \int_{[0,T]\times\T^d} d\x \Big| \int_\R dv \, \varphi^<(\tfrac{a(v)}{\delta}) \chi(\x,v) \Big| 
\leq T \int_\R dv\, \1_{[0,2]}(\tfrac{a(v)}{\delta}) \, .
\end{equation}
We conclude by noting that \eqref{eq:alower} implies $|v|\lesssim a(v)^{1/(M-1)}$.
\end{proof}

To prove Theorem~\ref{thm:modelledness} we shall also make use of the following simple lemma. 

\begin{lemma}\label{lem:Pi<}
Let $\Pi_\x^<[\lolly;\bar a](\y) \coloneqq \varphi^<(\bar a/\delta) \Pi_\x[\lolly;\bar a](\y)$ for $\delta>0$, and let $R,T>0$. 
Then under the assumption of Theorem~\ref{thm:main} it holds for $\|\y\|_\s\leq R$
\begin{equation}
\int_{D_\y} d\x \, |\sigma(u(\x)) \Pi_\x^<[\lolly;\x](\x+\y)|
\lesssim [\lolly]_\alpha \|\y\|_\s^\alpha T\, \delta^{1/(M-1)} \, ,
\end{equation}
where $\lesssim$ depends on the constant $C$ of Assumption~\ref{ass:nonlinearities}.
\end{lemma}

\begin{proof}
The model estimate \eqref{eq:lollybound} yields 
\begin{align}
\int_{D_\y} d\x \, |\sigma(u(\x)) \Pi_\x^<[\lolly;\x](\x+\y)|
&\leq [\lolly]_\alpha \|\y\|_\s^\alpha \int_{D_\y} d\x \, \big|\tfrac{\sigma(u(\x))}{a(u(\x))^{\mathfrak{e}(\lollysmall)}}\big| \, \varphi^<(\tfrac{a(u(\x))}{\delta}) \\
&\lesssim [\lolly]_\alpha \|\y\|_\s^\alpha \int_{D_\y} d\x \, |u(\x)| \, \varphi^<(\tfrac{a(u(\x))}{\delta}) \, ,
\end{align}
where in the second inequality we used that Assumption~\ref{ass:nonlinearities}~(ii) and \eqref{eq:restrictionM} imply 
%
\begin{equation}
\Big|\frac{\sigma(v)}{a(v)^{\mathfrak{e}(\lollysmall)}}\Big|\lesssim |v| \, .
\tag{I}\label{eq:I}
\end{equation}
Using $0\leq\varphi^< \leq\1_{[0,2]}$ and noting that \eqref{eq:alower} implies $|u(\x)|\lesssim a(u(\x))^{1/(M-1)}$, we conclude together with $|D_\y|\leq T$.
\end{proof}


We turn to the estimate of $u^>$.
The main result of this section is modelledness of $u^>$ to order $2\alpha$ which is deteriorating as $\delta\to0$. 
To formulate this modelledness we introduce
\begin{equation}
\Pi_\x^>[\lolly;\bar{a}](\y) 
\coloneqq \varphi^>(\tfrac{\bar{a}}{\delta}) \Pi_\x[\lolly;\bar{a}](\y) \, ,
\end{equation}
and 
\begin{equation}
\nu^>(\x)\coloneqq \nabla_z\big(u^>(\z)-u^>(\x)-\sigma(u(\x))\Pi_\x^>[\lolly;\x](\z)\big)_{|\z=\x} \, .
\end{equation}
With this notation we have the following proposition. 

\begin{proposition}[Large velocities]\label{prop:u>}
Under the assumption of Theorem~\ref{thm:main} and for $\beta\in(2-\alpha,2\alpha)$, 
$\epsilon>0$, 
dyadic\footnote{i.e.~of the form $2^{-n}$ for $n\in\Z$} $\delta\in(0,1)$, and $0<R,T\leq1$ it holds for $\|\y\|_\s\leq R$
\begin{align}\label{eq:modelledness_u>}
&\int_{D_\y} d\x\,
\big|u^>(\x+\y)-u^>(\x)-\sigma(u(\x))\Pi_\x^>[\lolly;a(u(\x))](\x+\y)-\nu^>(\x)\cdot y \big| \\
&\lesssimdata \|\y\|_\s^{2\alpha} \delta^{-\alpha-\epsilon} 
\big(1+(R+T)^{1-\alpha} [\mathcal{U}]_{B_{1,\infty}^\beta} \big) \, .
\end{align}
The implicit constant in $\lesssimdata$ depends on $\beta$, $\epsilon$, and the data as in Theorem~\ref{thm:main}.
The best constant (implicit in $\lesssimdata$) is denoted by $[\mathcal{U}^>]_{B_{1,\infty}^{2\alpha}}$.
\end{proposition}
\noindent
The remainder of this section is dedicated to the proof of Proposition~\ref{prop:u>}, 
which is finally given in Section~\ref{sec:proof_prop:u>}.

\subsection{Mild formulation for large velocities}
\label{sec:mild}

To prove Proposition~\ref{prop:u>} we appeal to the mild formulation of \eqref{eq:kinetic}
to rewrite 
\begin{align}
u^>(t,x) 
&= \int_\R dv\, \varphi^>(\tfrac{a(v)}{\delta}) \chi(t,x,v) \\
&= \int_\R dv\, \varphi^>(\tfrac{a(v)}{\delta}) e^{t a(v)\Delta} \chi(0,x,v) \\
&\,+ \int_0^t ds\, \Big(e^{(t-s)a(v)\Delta} \varphi^>(\tfrac{a(v)}{\delta})
\big(\sigma(v)\xi(s,x) - \sigma'(v)\sigma(v)C^{a(v)}\big) \Big)_{|v=u(s,x)}\\
&\,+ \int_\R dv\int_0^t ds \, e^{(t-s)a(v)\Delta} \varphi^>(\tfrac{a(v)}{\delta}) \, \partial_v \big(\delta_{u(s,x)}(v) a(u(s,x))|\nabla u(s,x)|^2 \big) \, ,
\end{align}
where 
\begin{equation}
\chi(0,x,v) 
= \1_{(-\infty,u_0(x))}(v) - \1_{(-\infty,0)}(v) \, .
\end{equation}
We rewrite this mild formulation in terms of the dilated heat kernel $\Psi$ defined in \eqref{eq:dilatedheatkernel} 
\begin{align}
&u^>(\x) \\
&=\int_{\R} dv \, \varphi^>(\tfrac{a(v)}{\delta}) 
\int_{\R^d} dz\, \Psi\big(a(v),t,x-z\big) \chi(0,z,v) \\
&\,+\int_{(0,\infty)\times\R^d}\! d\z \, \varphi^>(\tfrac{a(u(\z))}{\delta}) \Psi\big(a(u(\z)),\x-\z\big) 
\Big(\sigma(u(\z))\xi(\z) - (\sigma'\sigma)(u(\z))C^{a(u(\z))} \Big) \\
&\,- \int_{(0,\infty)\times\R^d}\! d\z\, 
\partial_{v} \big( \varphi^>(\tfrac{a(v)}{\delta}) \Psi(a(v),\x-\z) \big)_{|v=u(\z)} 
a(u(\z)) |\nabla u(\z)|^2 \, . 
\end{align}
As discussed in Section~\ref{sec:strategy}, this right hand side contains additional singular products. 
Crucially, both singular terms need the same (up to an integrable blowup) counterterm 
with opposite signs which we include now: 
\begin{align}
&u^>(\x) \\
&=\int_\R dv\, \varphi^>(\tfrac{a(v)}{\delta}) 
\int_{\R^d} dz \, \Psi(a(v),t,x-z\big) \chi(0,z,v) \\
&\,+\int_{(0,\infty)\times\R^d} d\z 
\Big[\varphi^>(\tfrac{a(u(\z))}{\delta}) \Psi\big(a(u(\z)),\x-\z\big) \sigma(u(\z)) \xi(\z) \\
&\hphantom{\coloneqq\int_{(0,\infty)\times\R^d}d\z}- \partial_v\Big( \varphi^>(\tfrac{a(v)}{\delta}) \Psi(a(v),\x-\z) \sigma(v) \Big)_{|v=u(\z)}
\sigma(u(\z)) \cdumb^{a(u(\z))}(r) \Big] \\
&\,-\! \int_{(0,\infty)\times\R^d} \! d\z 
\Big[ \partial_v \big( \varphi^>(\tfrac{a(v)}{\delta}) \Psi(a(v),\x\!-\!\z) \big)
a(v) \big( |\nabla u(\z)|^2 \!-\! \sigma^2(v) \ccherry^{a(v)}(r) \big) \Big]_{|v=u(\z)}  \\
&\,+\! \int_{(0,\infty)\times\R^d}\! d\z \Big[\partial_v \big( \varphi^>(\tfrac{a(v)}{\delta}) \Psi(a(v),\x\!-\!\z) \big) \sigma^2(v) \big( \cdumb^{a(v)}(r) \!-\! a(v) \ccherry^{a(v)}(r)\big) \Big]_{|v=u(\z)}\\
&\,+\! \int_{(0,\infty)\times\R^d} \! d\z \, \varphi^>(\tfrac{a(u(\z))}{\delta}) \Psi(a(u(\z)),\x\!-\!\z) (\sigma'\sigma)(u(\z)) \Big( \cdumb^{a(u(\z))}(r)-C^{a(u(\z))}\Big) \, .
\end{align}

In the proof it will be convenient to decompose $\varphi^>$ further into dyadic blocks.
We thus fix a smooth $\varphi:\R\to[0,1]$ with $\mathrm{supp}\,\varphi\subseteq[\tfrac{1}{2},2]$ such that $\sum_{q\in\Z} \varphi(2^q t)=1$ for all $t>0$. 
Choosing $\varphi^>\coloneqq \sum_{q<0} \varphi(2^q \cdot)$ we obtain for dyadic $\delta\in(0,1)$ the decomposition 
\begin{equation}\label{eq:decompose_phi>}
\varphi^>(\tfrac{a(v)}{\delta}) = \sum_{q\colon \delta<2^{-q}} \varphi(2^q a(v)) \, .
\end{equation}
Denoting 
\begin{equation}\label{eq:Psi_q}
\Psi_q(\bar{a},\x)\coloneqq \varphi(2^q\bar{a})\Psi(\bar{a},\x) \, ,
\end{equation}
and 
\begin{align}
u_0^q(\x) &\coloneqq 
\int_\R dv \int_{\R^d} dz \, \Psi_q\big(a(v),t,x-z\big) \chi(0,z,v) \, , \\
u_1^q(\x) &\coloneqq
\int_{(0,\infty)\times\R^d} d\z \Big[ \Psi_q\big(a(u(\z)),\x-\z\big) \sigma(u(\z)) \xi(\z) \\ 
&\hphantom{\coloneqq\int_{(0,\infty)\times\R^d}d\z}- 
\partial_v\Big( \Psi_q\big(a(v),\x-\z\big) \sigma(v) \Big)_{|v=u(\z)} 
\sigma(u(\z)) \cdumb^{a(u(\z))}(r) \Big] \, , \\
u_2^q(\x) &\coloneqq
\int_{(0,\infty)\times\R^d} d\z 
\Big[ \partial_v \Psi_q\big(a(v),\x-\z\big)
a(v)\big( |\nabla u(\z)|^2 - \sigma^2(v) \ccherry^{a(v)}(r) \big) \Big]_{|v=u(\z)} \, , \\
u_3^q(\x) &\coloneqq
\int_{(0,\infty)\times\R^d} d\z \, 
\Big[ \partial_v \Psi_q\big(a(v),\x-\z\big)
\sigma^2(v) \big( \cdumb^{a(v)}(r) - a(v) \ccherry^{a(v)}(r) \big) \Big]_{|v=u(\z)} \\
&\,+\int_{(0,\infty)\times\R^d} d\z \, 
\Psi_q\big(a(u(\z)),\x-\z\big)
(\sigma'\sigma)(u(\z)) \big( \cdumb^{a(u(\z))}(r) - C^{a(u(\z))} \big) \, , 
\end{align}
we obtain for $u^>$ the decomposition 
\begin{equation}
u^> = \sum_{q\colon\delta<2^{-q}} \big(u^q_0 + u^q_1 - u^q_2 + u^q_3\big)\, .
\end{equation}
Analogously we can decompose 
\begin{align}
\Pi^>_\x[\lolly;\bar{a}](\y)
&\,= \sum_{q\colon \delta<2^{-q}} \Pi_\x^q[\lolly;\bar{a}](\y) \\
&\coloneqq \sum_{q\colon \delta<2^{-q}} 
\int_{(0,\infty)\times\R^d} d\z\,\Big(\Psi_q\big(\bar{a},\y-\z\big)-\Psi_q\big(\bar{a},\x-\z\big)\Big)\xi(\z) \, .
\end{align}
To prove Proposition~\ref{prop:u>} we thus have to estimate the $L^1_\x(D_{\y})$-norm of 
\begin{align}
&u_0^q(\x+\y) - u_0^q(\x) - \nabla u_0^q(\x)\cdot y \\
&+u_1^q(\x\!+\!\y) \!-\! u_1^q(\x) \!-\! \sigma(u(\x)) \Pi_\x^q[\lolly;\x](\x\!+\!\y) \!-\! \big(\nabla u_1^q(\x) \!-\! \sigma(u(\x)) \nabla\Pi_\x^q[\lolly;\x](\x) \big) \!\cdot\! y \\
&- \big( u_2^q(\x+\y) - u_2^q(\x) - \nabla u_2^q(\x) \cdot y \big) 
+ u_3^q(\x+\y) - u_3^q(\x) - \nabla u_3^q(\x) \cdot y  \, . 
\end{align}
We break up the contributions from $u_0,\dots,u_3$ by the triangle inequality and estimate them separately 
in the four following propositions.

\begin{proposition}\label{prop:u0}
Let $u_0\in L^1(\T^d)$, $T>0$, $\y\in\R\times\R^d$, and $\delta\in(0,1)$. 
Then 
\begin{equation}
\sum_{q\colon \delta<2^{-q}} \int_{D_{\y}} d\x\,
\Big|u_0^q(\x+\y)-u_0^q(\x)-\nabla u_0^q(\x) \cdot y \Big|
\lesssim \|\y\|_\s^{2\alpha} \, \delta^{-\alpha} T^{1-\alpha} \|u_0\|_{L_1(\T^d)} \, .
\end{equation}
\end{proposition}
\noindent
The proof of Proposition~\ref{prop:u0} is given in Section~\ref{sec:initial}.

In the following two propositions it is convenient to consider a smooth $\Theta\colon\R\to\R$ which is monotone increasing, coincides with the identity on the support of $\sigma$, and is constant outside a neighbourhood of the support of $\sigma$, so that in particular 
\begin{equation}\label{eq:Theta}
f(u)\sigma(u)=f(\Theta(u))\sigma(\Theta(u)) \quad\textnormal{and}\quad 
\sigma(u)=\Theta'(u)\sigma(u) \, .
\end{equation}
This will ensure in the following that $\nu$ is only evaluated where $|u|\lesssim1$, so that its $L^2$ norm can be estimated by $[\mathcal{U}]_{B_{1,\infty}^\beta}$, see Lemma~\ref{lem:absorb}~\ref{it:nu^2_0}.
\begin{proposition}\label{prop:u1}
Let the assumption of Theorem~\ref{thm:main} hold.
Furthermore let $0<R,T\leq1$, $\beta\in(2-\alpha,2\alpha]$, and $\delta\in(0,1)$. 
Then for $\|\y\|_\s\leq R$
\begin{align}
\sum_{q\colon\delta<2^{-q}} \int_{D_{\y}} d\x\,
\Big|u_1^q(\x+\y)&-u_1^q(\x)-\sigma(u(\x))\Pi_\x^q[\lolly;\x](\x+\y) \\
&\!\!\!\!- \big(\nabla u_1^q(\x) - \sigma(u(\x)) \nabla\Pi_\x^q[\lolly;\x](\x) \big) \cdot y \Big|
\lesssim \|\y\|_\s^{2\alpha} \delta^{-\alpha} \, C_1 \, ,
\end{align}
where $\lesssim$ depends on $\beta$ and $C$ of Assumption~\ref{ass:nonlinearities}, and where $C_1$ is given by 
\begin{align}
C_1&=
[\dumb]_{2\alpha-2} T \\
&\,+ R^{1-\alpha} [\xnoise]_{\alpha-1} \|\Theta'(u)\nu\|_{L^1} \\
&\,+ R^{1-\alpha} 
[\xi]_{\alpha-2} 
\Big( [\Theta(\mathcal{U})]_{B^{\beta}_{1,\infty}} 
+ [\lolly]_\alpha^2 T
+ [\lolly]_\alpha \|\Theta'(u)\nu\|_{L^1} 
+ \|\Theta'(u)\nu\|_{L^2}^2 \Big) \\
&\,+ R^{1-\alpha}  
[\dumb]_{2\alpha-2} [\Theta(u)]_{B^\alpha_{1,\infty}} \\
&\,+ R^{1-\alpha} 
[\xnoise]_{\alpha-1} 
\Big( 
[\Theta(u)]_{B^\alpha_{1,\infty}}^{1/2} \|\Theta'(u)\nu\|_{L^2} 
+ [\Theta'(u)\nu]_{B^{\beta-1}_{1,\infty}} \Big)
\, ,
\end{align}
with $[\Theta(\mathcal{U})]_{B_{1,\infty}^\beta}$ defined as the supremum over $\|\y\|_\s\leq R$ of
\begin{equation}
\|\y\|_\s^{-\beta}\! \int_{D_{\y}} d\x \, 
\big| \Theta(u(\x+\y)) - \Theta(u(\x)) - \Theta'(u(\x))\big(\sigma(u(\x))\Pi_{\x}[\lolly;\x](\x+\y) + \nu(\x)\cdot y\big) \big| \, .
\end{equation}
\end{proposition}

\noindent
The proof of Proposition~\ref{prop:u1} is given in Section~\ref{sec:singularforcing}.

\begin{proposition}\label{prop:u2}
Let the assumption of Theorem~\ref{thm:main} hold. 
Furthermore let $0<R,T\leq1$, $\beta\in(2-\alpha,2\alpha]$, $\epsilon>0$, and $\delta\in(0,1)$. 
Then for $\|\y\|_\s\leq R$
\begin{equation}
\sum_{q\colon \delta<2^{-q}} \int_{D_{\y}} d\x\,
\Big|u_2^q(\x+\y)-u_2^q(\x)-\nabla u_2^q(\x) \cdot y \Big|
\lesssim \|\y\|_\s^{2\alpha} \, \delta^{-\alpha-\epsilon} C_2 \, ,
\end{equation}
where $\lesssim$ depends on $\beta$, $\epsilon$ and the constants $C$ of Assumptions~\ref{ass:nonlinearities} and~\ref{ass:noise}, and where $C_2$ is with $g(v)=\int_0^vdw\,|w|^{-1+\epsilon(M-1)}$ given by 
\begin{align}
C_2 
&= [\derivativecherry]_{2\alpha-2} 
\Big( [\Theta(u)]_{B^\alpha_{1,\infty}} 
+ T \Big) \\
&\,+ R^{1-\alpha} [\lolly]_\alpha 
\Big( [\Theta'(u)\nu]_{B_{1,\infty}^{\beta-1}}
+ 
[\Theta(u)]_{B_{1,\infty}^\alpha}^{1/2} \|\Theta'(u)\nu\|_{L^2}
+ \|\Theta'(u)\nu\|_{L^1} \Big) \\
&\,+ T^{1-\alpha} \Big( \|u_0\|_{L^{1+\epsilon(M-1)}(\T^d)}^{1+\epsilon(M-1)} + 1\Big) \\
&\,+ T^{1-\alpha} [\xi]_{\alpha-2} 
\Big( [(g\sigma)(\mathcal{U})]_{B^{\beta}_{1,\infty}} 
+ 1 \Big) \\
&\,+ T^{1-\alpha} ([\dumb]_{2\alpha-2} +[\derivativecherry]_{2\alpha-2} )
\Big( [\Theta(u)]_{B^\alpha_{1,\infty}} 
+ 1 \Big) \\
&\,+ T^{1-\alpha} ([\xnoise]_{\alpha-1} +  [\lolly]_\alpha ) \\
&\quad\times\Big( [\Theta'(u)\nu]_{B_{1,\infty}^{\beta-1}}
+ 
[\Theta(u)]_{B_{1,\infty}^\alpha}^{1/2} \|\Theta'(u)\nu\|_{L^2}
+  \|\Theta'(u)\nu\|_{L^1} (\sqrt{T})^{\alpha-1}\Big) \, .
\end{align}
\end{proposition}
\noindent
The proof of Proposition~\ref{prop:u2} is given in Section~\ref{sec:energy}.

\begin{proposition}\label{prop:u3}
Under the assumption of Theorem~\ref{thm:main}, 
and for $T>0$, $\y\in\R\times\R^d$, and $\delta\in(0,1)$, it holds
\begin{equation}
\sum_{q\colon \delta<2^{-q}} \int_{D_{\y}} d\x\,
\Big|u_3^q(\x+\y)-u_3^q(\x)-\nabla u_3^q(\x) \cdot y \Big|
\lesssim \|\y\|_\s^{2\alpha} \, \delta^{-\alpha} T^{1-\alpha} \, ,
\end{equation}
where $\lesssim$ depends on the constants $C$ of Assumptions~\ref{ass:nonlinearities} and \ref{ass:noise}. 
\end{proposition}
\noindent
The proof of Proposition~\ref{prop:u3} is given in Section~\ref{sec:initial}.

\subsection{Proof of Proposition \texorpdfstring{\ref{prop:u>} (Large velocities)}{large velocities}}\label{sec:proof_prop:u>}

Combining Propositions~\ref{prop:u0},~\ref{prop:u1},~\ref{prop:u2}, and \ref{prop:u3}, 
yields an estimate of $[\mathcal{U}^>]_{B_{1,\infty}^{2\alpha}}$
in terms of $C_1$ and $C_2$. 
The following lemma allows to estimate the constants $C_1$ and $C_2$ in terms of $[\mathcal{U}]_{B_{1,\infty}^\beta}$.

\begin{lemma}[Back to $\mathcal{U}$]\label{lem:absorb}
Let the assumption of Theorem~\ref{thm:main} hold.
Furthermore let $\beta\in(2-\alpha,2\alpha)$ and $0<R,T\leq1$.
Then for $f\in C^2(\R)$ with $f(v)=const$ for $|v|\geq C$ for some $C$
~\begin{enumerate}[label=(\roman*)]
\item \label{it:u_alpha}
$[f(u)]_{B_{1,\infty}^\alpha}
\lesssimdata (R^{\beta-\alpha}+R^{1-\alpha}\sqrt{T}) [\mathcal{U}]_{B_{1,\infty}^{\beta}} + 1$,
\item \label{it:f(U)_2alpha}
$[f(\mathcal{U})]_{B_{1,\infty}^{\beta}} 
\lesssimdata [\mathcal{U}]_{B_{1,\infty}^{\beta}} + 1$,
%
%
\item \label{it:nu^2_0}
$\|f'(u)\nu\|_{L^2}^2 
\lesssimdata [\mathcal{U}]_{B_{1,\infty}^{\beta}} + 1$,
\item\label{it:nu_2alpha-1}
$[f'(u)\nu]_{B_{1,\infty}^{\beta-1}} 
\lesssimdata [\mathcal{U}]_{B_{1,\infty}^{\beta}} + 1$.
\end{enumerate}
The implicit constant in $\lesssimdata$ depends on $\beta,C$, and the data as in Theorem~\ref{thm:main}.
\end{lemma}

\noindent
The proof of Lemma~\ref{lem:absorb} is standard but tedious due to the degeneracy  and is given in Appendix~\ref{app:absorb}.
We have now all ingredients to prove Proposition~\ref{prop:u>}.

\begin{proof}[Proof of Proposition~\ref{prop:u>}]
We begin with plugging Lemma~\ref{lem:absorb} into $C_1$ defined in Proposition~\ref{prop:u1}. 
Absorbing $\{[\tau]_{|\tau|}\}_{\tau\in\mathcal{T}}$ in the implicit constant and using $R,T\leq1$ and $\|\cdot\|_{L^1}\leq\sqrt{T}\|\cdot\|_{L^2}$ yields 
\begin{align}
C_1
\lesssimdata 
1 
&+ R^{1-\alpha} \big( \|\Theta'(u)\nu\|_{L^2} 
+ [\Theta(\mathcal{U})]_{B_{1,\infty}^{\beta}} 
+ [\Theta(u)]_{B_{1,\infty}^\alpha}
+ [\Theta'(u)\nu]_{B_{1,\infty}^{\beta-1}} \big) \, .
\end{align}
Using Lemma~\ref{lem:absorb} together with $R,T\leq1$ and $\beta\geq1$ implies 
\begin{equation}
C_1 
\lesssimdata 1 + R^{1-\alpha} [\mathcal{U}]_{B_{1,\infty}^{\beta}} \, .
\end{equation}
Appealing to the same argumentation for $C_2$ defined in Proposition~\ref{prop:u2}, 
and using that $\|u_0\|_{L^{1+\epsilon(M-1)}(\T^d)}\leq\|u_0\|_{L^p(\T^d)}$ for $\epsilon>0$ small enough yields
\begin{align}
C_2 
&\lesssimdata 
1
+ [\Theta(u)]_{B_{1,\infty}^{\alpha}} 
	+ R^{1-\alpha} \big([\Theta'(u)\nu]_{B_{1,\infty}^{\beta-1}} 
	+ \|\Theta'(u)\nu\|_{L^2}^2 \big) \\
&+ T^{1-\alpha} [(g\sigma)(\mathcal{U})]_{B_{1,\infty}^{\beta}} 
	+ T^{1-\alpha} \big([\Theta'(u)\nu]_{B_{1,\infty}^{\beta-1}} 
	+ \|\Theta'(u)\nu\|_{L^2}^2 \big)\, .
\end{align}
Recall that here $g(v)=\int_0^vdw\, |w|^{-1+\epsilon(M-1)}$ and consequently 
$g\sigma\in C^2(\R)$ and compactly supported (by Assumption~\ref{ass:nonlinearities}~(ii)).
Using Lemma~\ref{lem:absorb} together with $R,T\leq1$ and $\beta\geq1$ implies 
\begin{equation}
C_2
\lesssimdata 1 
+ R^{1-\alpha} [\mathcal{U}]_{B_{1,\infty}^{\beta}} 
+ T^{1-\alpha} [\mathcal{U}]_{B_{1,\infty}^{\beta}}  \, .
\end{equation}
Combining this with Propositions~\ref{prop:u0}~--~\ref{prop:u3} we thus have 
\begin{equation}
[\mathcal{U}^>]_{B_{1,\infty}^{2\alpha}} 
\lesssimdata \delta^{-\alpha-\epsilon} \big( 1 + C_1+C_2 \big)
\lesssimdata \delta^{-\alpha-\epsilon} \big( 1 + (R+T)^{1-\alpha} [\mathcal{U}]_{B_{1,\infty}^{\beta}} \big) \, .
\end{equation}
\end{proof}

\section{Reconstruction and integration with rough singular kernels}\label{sec:roughkernels}

In this section we collect a number of integration and reconstruction arguments used in Section~\ref{sec:averaging} to prove Propositions~\ref{prop:u0} -- \ref{prop:u3}.
Some of these results are well known, however we provide all proofs as we need the dependence on the diffusion coefficient. 

When estimating $u_0,\dots,u_3$ in Propositions~\ref{prop:u0} -- \ref{prop:u3}, increments of the form 
\begin{equation}
\Psi_q\big(a,\x+\y-\z\big)
- \Psi_q\big(a,\x-\z\big) 
- \nabla\Psi_q\big(a,\x-\z\big)\cdot y
\end{equation}
integrated in $\x$ over $D_\y$ and in $\z$ over $(0,\infty)\times\R^d$ will come up. 
It will be convenient to replace $\Psi_q$ defined in \eqref{eq:Psi_q} by a compactly supported $K$; 
we suppress the $q$-dependence of $K$ in the notation as $q$ is fixed in what follows. 
Since the time components of the arguments $\x+\y-\z$ and $\x-\z$ are elements of $[0,T]$ by the definitions of $\Psi_q$ and $D_\y$, 
we only worry about spatial components. 
Furthermore, since we only ever consider convolutions of $\Psi_q$ with $1$-periodic (in space) functions, 
we can replace $\Psi_q$ by any function having the same (spatial) periodization. 
We thus fix a smooth $\zeta\colon\R^d\to[0,1]$ 
with $\mathrm{supp}\,\zeta\subseteq \{x\in\R^d\,|\,|x|<1\}$
such that $\sum_{k\in\Z^d} \zeta(x+k) = 1$ for all $x\in\R^d$. 
An example of such a $\zeta$ is given by $\zeta = \1_{[-1/2,1/2]^d}*\tilde\zeta$
for a smooth approximation $\tilde\zeta\colon\R^d\to\R$ of the Dirac, 
say with compact support in the ball of radius $2^{-10}$ centered at the origin and $\int\tilde\zeta=1$; 
then 
\begin{equation}
\sum_{k\in\Z^d} \zeta(x+k)
= \sum_{k\in\Z^d} \int_{\R^d} dy \, \1_{[-\frac12,\frac12]^d}(x+k-y) \tilde\zeta(y) 
= \int_{\R^d} dy \, \tilde\zeta(y) 
= 1 \, .
\end{equation}
With this $\zeta$ define (remember that we suppress the $q$-dependence in the notation) 
\begin{equation}\label{eq:defK}
K(\bar{a},\x) \coloneqq 
\sum_{k\in\Z^d}
\Psi_q(\bar{a}, (t, x+k)) \zeta(x)
\quad \textnormal{for }\bar{a}\in\R, \, \x=(t,x)\in\R\times\R^d \, .
\end{equation}
The (spatial) periodization of $K$ is given by 
\begin{equation}
\sum_{\ell\in\Z^d} K(\bar a, t, x+\ell) 
= \sum_{k\in\Z^d} \Psi_q(\bar a, (t, x+k)) 
\sum_{\ell\in\Z^d} \zeta(x+\ell) \, , 
\end{equation}
which coincides with the (spatial) periodization of $\Psi_q$. 

To establish estimates for convolutions of $K$ it is convenient to decompose $K$ as follows.
Fix a smooth $\eta$ compactly supported in 
$\{\x\in\R\times\R^d \,\big|\, 2^{-1} <\|\x\|_\s<2 \}$
such that $\{\eta(\S^{2^n}\cdot)\}_{n\in\N}$ is a partition of unity on\footnote{strictly speaking this requires $\sqrt{2T+1}<2$, which will always be satisfied as we are interested in small $T$} $[0,2T]\times\{x\in\R^d\,|\,|x|<1\}$.
Then for $q\in\N$ 
\begin{equation}\label{eq:splittingk}
K(\bar{a},\x) = \sum_{n\in\N} \Kn(\bar{a},\x) 
\quad\textnormal{for }\bar{a}\in\R\textnormal{ and }\x\in[0,T]\times\R^d \, ,
\end{equation}
where $\Kn(\bar{a},\x)\coloneqq\eta(\S^{2^n}\Da\x) K(\bar{a},\x)$ with the dilation $\Da\x\coloneqq(\bar at,x)$, 
and where we used that $K(\bar a,\cdot)$ vanishes unless $\bar a\leq2\cdot2^{-q}$, see \eqref{eq:defK} and \eqref{eq:Psi_q}.

Using this decomposition 
we shall often distinguish the cases $2^{-n}\leq\|\y\|_\s$ and $2^{-n}>\|\y\|_\s$; 
the former is referred to as the near-field regime 
and the latter is referred to as the far-field regime. 
In the near field we shall not make use of cancellations due to the Taylor remainder of $K$ 
and estimate the contributions from $K(\bar{a},\x+\y-\z)$, $K(\bar{a},\x-\z)$, and 
$\nabla K(\bar{a},\x-\z)\cdot y$ separately. 
In the far field we use the following representation of the anisotropic Taylor remainder: 
\begin{align}
&\bar{K}\big(\bar{a},\x,\y,\z\big) \\
&\coloneqq K\big(\bar{a},\x+\y-\z\big) - K\big(\bar{a},\x-\z\big) - \nabla K\big(\bar{a},\x-\z\big) \cdot y \\
&\,=\sum_{\substack{\hphantom{2}l+|k|_1\leq2,\\2l+|k|_1\geq2\hphantom{,}}} 
C_{l,k} 
\int_0^1 d\vartheta \, (1-\vartheta) \vartheta^{2l+|k|_1-2} \partial_t^l\partial_x^k K\big(\bar{a},\x+\S^\vartheta \y-\z\big) s^l y^k \label{eq:anisotropicTaylor}
\end{align}
for some generic constants $C_{l,k}$, 
which can be checked easily by plugging $f(\vartheta)\coloneqq K\big(\bar{a},\x+\S^\vartheta \y-\z\big)$ 
into $f(1)-f(0)-f'(0) = \int_0^1 d\vartheta\, (1-\vartheta) f''(\vartheta)$. 

With this decomposition at hand we give a couple of lemmas which are frequently used. 
The reader may directly jump to Section~\ref{sec:averaging} on first reading, 
and come back to these rather technical results when they are used.
We start with the following $L^1$-based regularity estimate for convolutions with the heat kernel, where we do not yet make use of the decomposition introduced above. 

\begin{lemma}\label{lem:integration1}
Recall $\Psi_q$ defined in \eqref{eq:Psi_q} for $q\in\Z$, and let $m\in\N$, $\y\in\R\times\R^d$, and $f,g\colon(0,T)\times\T^d\to\R$. 
Then 
\begin{align}
&\int_{D_\y}\!\! d\x \Big|\! \int_{(0,\infty)\times\R^d} \!\! d\z\, 
\partial_{\bar a}^m \! \big(\Psi_{\!q}(\bar a,\x\!+\!\y\!-\!\z) \!-\! \Psi_{\!q}(\bar a, \x\!-\!\z) \!-\! \nabla\Psi_{\!q}(\bar a, \x\!-\!\z) \!\cdot\! y \big)_{|\bar a = f(\z)}
g(\z) \Big| \\
&\lesssim T^{1-\alpha} \|\y\|_\s^{2\alpha} 
\int_{(0,T)\times\T^d} d\z \, |g(\z)| |f(\z)|^{-m} (1+|f(\z)|^{-\alpha}) |\tilde\varphi(2^q f(\z))| 
\end{align}
for some smooth $\tilde\varphi:\R\to\R$ supported in $(1/2,2)$. 
The same statement holds if $\int_{(0,\infty)\times\R^d}d\z$ and $\int_{(0,T)\times\T^d}d\z$ are replaced by $\int_{\R^d}dz$ and $\int_{\T^d}dz$, respectively. 
\end{lemma}

\begin{proof}
\textbf{Step 1.} We first prove for $\bar a>0$ and $s,t\in\R$
\begin{align}
&\int_{\R^d} dx \big| \partial_{\bar a}^m \Psi(\bar a, t+s,x+y) - \partial_{\bar a}^m \Psi(\bar a, t,x) - \partial_{\bar a}^m \nabla\Psi(\bar a,t,x)\cdot y \big| \\
&\lesssim \bar a^{-m} (1+\bar a^{-\alpha}) \|(s,y)\|_\s^{2\alpha} \cdot \left\{
\begin{array}{ll}
\min\{t,t+s\}^{-\alpha} & \textnormal{if }t>0,\,t+s>0 \, , \\
t^{-\alpha} & \textnormal{if }t>0,\,t+s<0 \, , \\
(t+s)^{-\alpha} & \textnormal{if }t<0,\,t+s>0 \, , \\
0 & \textnormal{if }t<0,\,t+s<0 \, .
\end{array}\right. 
\end{align}
To see this, recall that the heat kernel $\Phi$ defined in \eqref{eq:heatkernel} satisfies for $l\in\N$ and $k\in\N^d$ (we will only need it for $2l+|k|_1\leq2$) 
\begin{equation}
\int_{\R^d}dx \big| \partial_t^l \partial_x^k \Phi(t,x)\big| 
\lesssim t^{-(2l+|k|_1)/2} \1_{t>0} \, ,
\end{equation}
which by $\partial_{\bar a}^m \partial_t^l\partial_x^k\Psi(\bar a,t,x)=\partial_t^{m+l}\partial_x^k\Phi(\bar at,x) a^l t^m$ implies 
\begin{equation}
\int_{\R^d}dx \big| \partial_{\bar a}^m \partial_t^l \partial_x^k \Psi(\bar a,t,x)\big| 
\lesssim a^{-(2m+|k|_1)/2} t^{-(2l+|k|_1)/2} \1_{t>0} \, .
\end{equation}
Thus for time increments where $t,t+s>0$ we obtain by the fundamental theorem of calculus 
\begin{align}
&\int_{\R^d} dx \big| \Psi(\bar a,t+s,x+y) - \Psi(\bar a,t,x+y) \big| \\
&=\int_{\R^d} dx \Big| \int_0^1 d\vartheta \, \partial_t \Psi(\bar a,t+\vartheta s,x+y) s \Big| 
\lesssim \int_0^1d\vartheta\, (t+\vartheta s)^{-1} |s|
\lesssim \frac{|s|}{\min\{t,t+s\}} \, ,
\end{align}
and by the triangle inequality also 
\begin{equation}\label{eq:mt01}
\int_{\R^d} dx \big| \Psi(\bar a,t+s,x+y) - \Psi(\bar a,t,x+y) \big| 
\lesssim 1 \, ,
\end{equation}
and hence by interpolation 
\begin{equation}
\int_{\R^d} dx \big| \Psi(\bar a,t+s,x+y) - \Psi(\bar a,t,x+y) \big| 
\lesssim \frac{|s|^\alpha}{\min\{t,t+s\}^\alpha} 
\leq \frac{\|(s,y)\|_\s^{2\alpha}}{\min\{t,t+s\}^\alpha} \, .
\end{equation}
For time increments with $t+s<0<t$ we use \eqref{eq:mt01} (which holds for all $s,t\in\R$) and that in this case $1\leq |s|/t$ to deduce
\begin{equation}
\int_{\R^d} dx \big| \Psi(\bar a,t+s,x+y) - \Psi(\bar a,t,x+y) \big| 
\lesssim \frac{|s|^\alpha}{t^\alpha} 
\leq \frac{\|(s,y)\|_\s^{2\alpha}}{t^\alpha} \, ,
\end{equation}
and if $t<0<t+s$ we use \eqref{eq:mt01} and that in this case $1\leq s/(t+s)$ to obtain 
\begin{equation}
\int_{\R^d} dx \big| \Psi(\bar a,t+s,x+y) - \Psi(\bar a,t,x+y) \big| 
\lesssim \frac{|s|^\alpha}{(t+s)^\alpha} 
\leq \frac{\|(s,y)\|_\s^{2\alpha}}{(t+s)^\alpha} \, .
\end{equation}
For spatial increments and $t\in\R$ we obtain by the fundamental theorem of calculus and the triangle inequality 
\begin{align}
&\int_{\R^d} dx \big| \Psi(\bar a,t,x+y) - \Psi(\bar a,t,x) - \nabla\Psi(\bar a,t,x)\cdot y \big| \\
&=\int_{\R^d} dx \Big| \int_0^1d\vartheta \, \nabla \Psi(\bar a,t,x+\vartheta y)\cdot y - \nabla\Psi(\bar a,t,x)\cdot y \Big| 
\lesssim (\bar a t)^{-1/2} |y| \1_{t>0} \, .
\end{align}
Instead of applying the triangle inequality, we can also apply the fundamental theorem of calculus a second time to obtain 
\begin{align}
&\int_{\R^d} dx \big| \Psi(\bar a,t,x+y) - \Psi(\bar a,t,x) - \nabla\Psi(\bar a,t,x)\cdot y \big| \\
&=\int_{\R^d} dx \Big| \int_0^1d\vartheta \int_0^1d\bar\vartheta \sum_{|k|_1=2} \partial_x^k \Psi(\bar a,t,x+\vartheta\bar\vartheta y) y^k \Big| 
\lesssim (\bar a t)^{-1} |y|^2 \1_{t>0} \, ,
\end{align}
so that by interpolation we obtain
\begin{equation}
\int_{\R^d} dx \big| \Psi(\bar a,t,x+y) - \Psi(\bar a,t,x) - \nabla\Psi(\bar a,t,x)\cdot y \big| 
\lesssim \frac{|y|^{2\alpha}}{(\bar a t)^\alpha} \1_{t>0}
\leq \frac{\|(s,y)\|_\s^{2\alpha}}{(\bar a t)^\alpha} \1_{t>0} \, .
\end{equation}
The argument for $m>0$ is analogous, which concludes Step~1.

\textbf{Step 2.}
We turn to the proper estimate.
Since $\Psi_q(\bar{a},\x+\vartheta\y-\z)$ for $\vartheta\in\{0,1\}$ vanishes unless the time component $t+\vartheta s-r$ of $\x+\vartheta\y-\z$ is positive, 
and since $\x\in D_\y$ implies $t+s\leq T$, 
we have effectively $r<t+\vartheta s\leq T$. 
Hence the integral $\int_{(0,\infty)\times\R^d}d\z$ effectively reduces to $\int_{(0,T)\times\R^d}d\z$. 
By periodicity of $f,g$, we further rewrite the left hand side of the assertion as
\begin{align}
\int_{D_\y} d\x \Big| \int_{(0,T)\times\T^d} d\z \sum_{j\in\Z^d}
\partial_{\bar a}^m 
\big(\Psi_{\!q}(\bar a,\x\!+\!\y\!-\!\z+(0,j)) 
- \Psi_{\!q}(\bar a, \x\!-\!\z+(0,j)) & \\
- \nabla\Psi_{\!q}(\bar a, \x\!-\!\z+(0,j)) \cdot y \big)_{|\bar a = f(\z)}
&g(\z) \Big| \, .
\end{align}
By the triangle inequality and Fubini this expression is bounded by (w.l.o.g.~let $s>0$, and recall $\x=(t,x),\y=(s,y),\z=(r,z)$)
\begin{align}
\int_{(0,T)\times\T^d} \!\! d\z \, |g(\z)| \int_{(0,T-s)\times\R^d} \!\! d\x 
\big| \partial_{\bar a}^m \big( \Psi_{\! q}(\bar a, t+s&-r,x+y) 
- \Psi_{\! q}(\bar a,t-r,x) \\
&- \nabla\Psi_{\! q}(\bar a,t-r,x)\cdot y \big)_{|\bar a=f(\z)} \big| \, .
\end{align}
The claim follows from $\Psi_q(\bar a,\cdot)=\varphi(2^q\bar a)\Psi(\bar a,\cdot)$ (see \eqref{eq:Psi_q}) and the product rule, together with Step~1 and $\int_0^Tdt \,t^{-\alpha}\lesssim T^{1-\alpha}$. 
\end{proof}

The next lemma is similar in flavour to the previous one, 
but is $L^\infty$-based and considers basepoint dependent integrands.

\begin{lemma}\label{lem:integration2}
Let $\bar K$ be the anisotropic second order Taylor remainder of $K$ from \eqref{eq:defK} as in \eqref{eq:anisotropicTaylor}, $m,q\in\N$, $\y\in\R\times\R^d$, $\x\in D_\y$, $T\leq1$, $\bar{a}\in\R$, 
and $f_\x\colon(0,\infty)\times\R^d\to\R$. 
Furthermore, let $\theta_1,\theta_2\in\R$, $\theta_1>-2$, $\theta_2\geq0$, $-1<\theta_1+\theta_2<0$, and
\begin{equation}
[f_\x]_{\theta_1,\theta_2}
\coloneqq
\sup_{\substack{\varphi\in\mathcal{B},\, \lambda\leq1 \\ \|\bar\x\|_\s<\|\y\|_\s+1}}
\lambda^{-\theta_1}(\lambda+\|\bar\x\|_\s)^{-\theta_2}
\Big|\int_{(0,\infty)\times\R^d}d\z f_\x(\z)\varphi^\lambda_{\x+\bar\x}(\z)\Big| \, .
\end{equation}
Then 
\begin{equation}
\Big|\int_{(0,\infty)\times\R^d} d\z\, 
\partial_{\bar a}^m \bar K\big(\bar{a},\x,\y,\z\big) f_\x(\z) \Big| 
\lesssim \|\y\|_\s^{\theta_1+\theta_2+2} [f_\x]_{\theta_1,\theta_2} \, 
\frac{|\tilde\varphi(2^q\bar{a})|}{\bar{a}^{m+1+\theta_2/2}} 
\end{equation}
for some smooth $\tilde\varphi\colon\R\to\R$ supported in $(1/2,2)$, 
with the understanding that the right hand side vanishes for $\bar a=0$.
\end{lemma}
\begin{proof}
If $\bar{a}\leq0$ or $\bar{a}>2$ the left hand side vanishes by the definition of $K$ in \eqref{eq:defK} via $\Psi_q$ (defined in \eqref{eq:Psi_q}) for $q\in\N$, 
and so does the right hand side by the support of $\tilde\varphi$. 
It remains to consider $0<\bar{a}\leq2$.

\textbf{Step 1.}
We first prove for $\vartheta\in[0,1]$ and $l,n\in\N$, $k\in\N^d$ that 
\begin{align}
&\Big|\int_{(0,\infty)\times\R^d} d\z\, 
\partial_{\bar a}^m \partial_t^l\partial_x^k K^{(n)} \big(\bar{a},\x+\S^\vartheta\y-\z\big) f_\x(\z) \Big| \\
&\lesssim 2^{-n(\theta_1+2-2l-|k|_1)} (2^{-n} + \vartheta\|\y\|_\s)^{\theta_2} [f_\x]_{\theta_1,\theta_2} \, 
\frac{|\tilde\varphi(2^q\bar{a})|}{\bar{a}^{m+1+\theta_2/2}} \, .
\end{align}
According to Lemma~\ref{lem:multtest}~\textit{(ii)} we have that 
\begin{equation}
\partial_{\bar a}^m \partial_t^l\partial_x^k K^{(n)} (\bar{a},\yy-\z) 
= 2^{-n(2-2l-|k|_1)} \bar{a}^{l-m} \rho_{\Da\yy}^{2^{-n}}(\Da\z) \tilde\varphi(2^q\bar{a})
\end{equation}
for some $\rho\in C\mathcal{B}$ supported on non-positive time components
and a smooth $\tilde\varphi\colon\R\to\R$ supported in $(1/2,2)$. 
Since $\rho(\Da\z)=\rho(\bar{a}r,z)$ is not a test function uniform in $\bar{a}$ we decompose it. 
We fix $r_i\in(0,\bar{a}^{-1})$ for $i=1,\dots,\bar{a}^{-1}$ and a\footnote{$\tilde\rho$ could depend on $i$, however such that its $C^n$-norms are uniformly bounded, which is why we drop the dependence in the notation for brevity} test function $\tilde\rho\in C\mathcal{B}$ supported on non-positive time components such that 
$\rho(\bar{a}r,z) = \sum_{i=1}^{\bar{a}^{-1}} \tilde\rho(r+r_i,z)$.
Then $\rho_{\Da\yy}^{2^{-n}}(\Da\z)
= \sum_{i=1}^{\bar{a}^{-1}} \tilde\rho^{2^{-n}}_{\yy-(2^{-2n}r_i,0)}(\z)$ 
and thus 
\begin{align}
&\Big|\int_{(0,\infty)\times\R^d} d\z\, 
\partial_{\bar a}^m \partial_t^l\partial_x^k K^{(n)} \big(\bar{a},\x+\S^\vartheta\y-\z\big) f_\x(\z) \Big|\\
&\leq 2^{-n(2-2l-|k|_1)} \bar{a}^{l-m} |\tilde\varphi(2^q\bar{a})| \sum_{i=1}^{\bar{a}^{-1}}
\Big|\int_{(0,\infty)\times\R^d}d\z\, \tilde\rho^{2^{-n}}_{\x+\S^\vartheta\y-(2^{-2n}r_i,0)}(\z) f_\x(\z) \Big| \, .
\end{align}
Since $\tilde\rho$ is supported in the past, 
the time component $r-t-\vartheta^2s+2^{-2n}r_i$ of $\z-\x-\S^\vartheta\y+(2^{-2n}r_i,0)$ is effectively non-positive; 
hence $2^{-2n}r_i\leq t+\vartheta^2s-r$, which by assumption and $r\geq0$ implies $2^{-2n}r_i\leq T\leq1$.
Thus the above right hand side is bounded by 
\begin{align}
2^{-n(2-2l-|k|_1)} \bar{a}^{l-m} |\tilde\varphi(2^q\bar{a})| \sum_{i=1}^{\bar{a}^{-1}} 
2^{-n\theta_1} (2^{-n}+2^{-n}\sqrt{r_i}+\vartheta\|\y\|_\s)^{\theta_2} 
[f_\x]_{\theta_1,\theta_2} \, ,
\end{align}
which is bounded as desired since $1,r_i\lesssim\bar{a}^{-1}$. 

\textbf{Step 2.}
Decomposing $K$ as in \eqref{eq:splittingk}, 
splitting into near-field and far-field, 
and using in the far-field the anisotropic Taylor remainder \eqref{eq:anisotropicTaylor} yields
\begin{align}
&\int_{D_\y} d\x \Big| \int_{(0,\infty)\times\R^d} d\z\, 
\partial_{\bar a}^m \bar K\big(\bar a,\x,\y,\z\big) f_\x(\z) \Big| \\
&\lesssim \sum_{n\colon 2^{-n}\leq \|\y\|_\s} 
\int_{D_\y} d\x \Big| \int_{(0,\infty)\times\R^d} d\z\, 
\partial_{\bar a}^m K^{(n)}\big(\bar a,\x+\y-\z\big) f_\x(\z) \Big| \\
&\,+ \sum_{n\colon 2^{-n}\leq \|\y\|_\s} 
\int_{D_\y} d\x \Big| \int_{(0,\infty)\times\R^d} d\z\, 
\partial_{\bar a}^m K^{(n)}\big(\bar a,\x-\z\big) f_\x(\z) \Big| \\
&\,+ \sum_{n\colon 2^{-n}\leq \|\y\|_\s} 
\int_{D_\y} d\x \Big| \int_{(0,\infty)\times\R^d} d\z\, 
\partial_{\bar a}^m \nabla K^{(n)}\big(\bar a,\x-\z\big)\cdot y f_\x(\z) \Big| \\
&\,+ \sum_{n\colon 2^{-n}> \|\y\|_\s} 
\int_{D_\y} d\x \Big| \int_{(0,\infty)\times\R^d} d\z \\
&\hphantom{\,+ \sum_{n\colon 2^{-n}> \|\y\|_\s} \int_{D_\y} d\x}\times\!
\sum_{\substack{\hphantom{2}l+|k|_1\leq2,\\2l+|k|_1\geq2\hphantom{,}}}
\int_0^1 d\vartheta\, \partial_{\bar a}^m \partial_t^l\partial_x^k K^{(n)}\big(\bar a,\x+\S^\vartheta\y-\z\big) s^l y^kf_\x(\z) \Big| \, .
\end{align}
Using Step 1 it remains to argue that 
\begin{align}
&\sum_{n\colon 2^{-n}\leq \|\y\|_\s} 
\big( 2^{-n(\theta_1+2)} (2^{-n} + \|\y\|_\s)^{\theta_2} 
+2^{-n(\theta_1+\theta_2+2)} 
+ 2^{-n(\theta_1+\theta_2+1)} \|\y\|_\s \big) \\
&+ \sum_{n\colon 2^{-n}> \|\y\|_\s} 
\sum_{\substack{\hphantom{2}l+|k|_1\leq2,\\2l+|k|_1\geq2\hphantom{,}}} 
2^{-n(\theta_1+2-2l-|k|_1)} (2^{-n} + \|\y\|_\s)^{\theta_2}
 \|\y\|_\s^{2l+|k|_1}
\lesssim \|\y\|_\s^{\theta_1+\theta_2+2} \, .
\end{align}
Indeed, the geometric series are bounded by their largest summands as the exponents on $2^{-n}$ are positive, 
where in the first near-field term we use additionally $2^{-n}+\|\y\|_\s\leq2\|\y\|_\s$.
For the geometric sum we use $2^{-n}+\|\y\|_\s\leq2\cdot2^{-n}$ and $\theta_1+\theta_2+2-2l-|k|_1<0$, 
hence it is also estimated by its largest summand. 
\end{proof}

The following lemma is a small upgrade of the modelbound \eqref{eq:modelbound} when 
the model is tested with the product of a smooth kernel and $K^{(n)}$. 

\begin{lemma}\label{lem:integration3}
Let $l,\bar m,n,q\in\N$, $k\in\N^d$, $\tilde m=0,1$, $\x,\y\in\R\times\R^d$, $\varphi\in\mathcal{B}$, 
$0<\lambda\leq2^{-n}$, $\bar{a}\in\R$, $\tilde{a}\in a(\mathrm{supp}\,\sigma)\setminus\{0\}$, 
$\tau\in\mathcal{T}$, and $\Pi_\x[\tau;\tilde a]$ such that \eqref{eq:modelbound} holds.
Then 
\begin{align}
&\Big|\int_{(0,\infty)\times\R^d} d\z\,\varphi_\x^\lambda(\z) \partial_{\bar a}^{\bar m} \partial_t^l\partial_x^k K^{(n)} \big(\bar{a},\y-\z\big) 
\partial_{\tilde a}^{\tilde m} \Pi_\x[\tau;\tilde a](\z) \Big| \\
&\lesssim 2^{n(D-2+2l+|k|_1)} \lambda^{|\tau|} [\tau]_{|\tau|} \, 
\frac{\tilde\varphi(2^q\bar{a})}{\tilde{a}^{\mathfrak{e}(\tau)+\tilde m} \, \bar{a}^{\bar m}} 
\end{align}
for some smooth $\tilde\varphi\colon\R\to\R$ supported in $(1/2,2)$, 
with the understanding that the right hand side vanishes for $\bar a=0$. 
\end{lemma}
\begin{proof}
If $\bar{a}\leq0$ or $\bar{a}>2$ the left hand side vanishes by the definition of $K$ in \eqref{eq:defK} via $\Psi_q$ (defined in \eqref{eq:Psi_q}) for $q\in\N$, 
and so does the right hand side by the support of $\tilde\varphi$. 
It remains to consider $0<\bar{a}\leq2$.
In this case, 
according to Lemma~\ref{lem:multtest} and the assumption $\lambda\leq2^{-n}$ we have that 
$\varphi_\x^\lambda(\z) \partial_{\bar a}^{\bar m}\partial_t^l\partial_x^k K^{(n)} (\bar{a},\y-\z) 
= 2^{n(D-2+2l+|k|_1)} \bar{a}^{l-\bar m} \psi_\x^\lambda(\z) \tilde\varphi(2^q\bar{a})$ 
for some $\psi\in C\mathcal{B}$ and some a $\tilde\varphi\colon\R\to\R$ supported in $(1/2,2)$; 
we emphasize that Lemma~\ref{lem:multtest}~\textit{(i)} still holds if the test function on the larger scale is dilated with $\bar{a}\leq2$. 
The assumed model bound \eqref{eq:modelbound} then yields the claim. 
\end{proof}

Since the composition of the kernel $K$ with the solution $u$,
which appears e.g.~when estimating $u_1$ in Proposition~\ref{prop:u1}, 
has only limited regularity, its convolution with distributions may be in need of renormalization. 
The next lemma is a combination of reconstruction and integration to account for this. 

\begin{lemma}\label{lem:reconstruction0}
Let $f\colon[0,T]\times\T^d\to\R$ and $g,h\colon([0,T]\times\R^d)^2\to\R$ be jointly $1$-periodic in its spatial arguments, i.e.~$g_\x(\y)=g_{\x+(0,k)}(\y+(0,k))$ for all $k\in\Z^d$. 
For $\x,\y,\z\in[0,T]\times\R^d$ and $q\in\N$ (recall that $K$ in \eqref{eq:defK} depends on $q$) define 
\begin{equation}
F_\x(\z;\y)\coloneqq K(f(\x),\y-\z)g_\x(\z) + \partial_{\bar a} K(\bar a,\y-\z)_{|\bar a=f(\x)} h_\x(\z) \, ,
\end{equation}
and denote by $F^{(n),l,k}$ the analogous expression where $K$ is replaced by $\partial_t^l \partial_x^k K^{(n)}$.

Assume that there exist a smooth $\varphi$ compactly supported in the past with $\int\varphi\neq0$ and 
$\theta_1,\theta_2\in\R$ with $\theta_1+\theta_2>0$, $\theta_1>-2$, $\theta_2\geq0$
such that 
\begin{align}
&\int_{D_{\y'}\cap D_{\y''}} d\x \Big| \int_{(0,\infty)\times\R^d} d\z \, 
\big(F^{(n),l,k}_{\x+\y'}(\z;\x+\y'') - F^{(n),l,k}_{\x}(\z;\x+\y'') \big) \varphi_\x^\lambda (\z) \Big| \\
&\leq C 2^{n(D-2+2l+|k|_1)} \lambda^{\theta_1} (\lambda+\|\y'\|_\s)^{\theta_2} \, ,
\end{align}
for $n,l\in\N$, $k\in\N^d$, $\y'\in\R\times\R^d$, $\y''\in[0,\infty)\times\R^d$,
and dyadic $\lambda\leq4\cdot2^{-n}$. 

Then $\bar F_{\x'}(\z;\x,\y)\coloneqq F_{\x'}(\z;\x+\y) - F_{\x'}(\z;\x) - \nabla_x F_{\x'}(\z;\x)\cdot y$ satisfies 
\begin{align}
\int_{D_\y} d\x \Big| \int_{(0,\infty)\times\R^d} d\z \, 
\big( \bar F_{\z}(\z;\x,\y) - \bar F_{\x}(\z;\x,\y) \big) \Big| 
\lesssim C 2^{q(\theta_2+2)/2} \|\y\|_\s^{2-} \, ,
\end{align}
for $\|\y\|_\s\leq1$.
By $2-$ we understand $2-\epsilon$ for arbitrary $\epsilon>0$, and the implicit constant in $\lesssim$ depends on this $\epsilon$.
The assumption is only used for $l+|k|_1\leq2$ and $\|\y'\|_\s,\|\y''\|_\s\lesssim1$. 
\end{lemma}

\begin{proof}
\textbf{Step 1 \textnormal{(splitting near-field and far field)}.}
Decomposing $K$ as in \eqref{eq:splittingk}, splitting into near-field and far-field, 
and using in the far-field the anisotropic Taylor remainder \eqref{eq:anisotropicTaylor} yields
\begin{align}
&\int_{D_\y} d\x \Big| \int_{(0,\infty)\times\R^d} d\z \, 
\big( \bar F_{\z}(\z;\x,\y) - \bar F_{\x}(\z;\x,\y) \big) \Big| \\
&\lesssim\sum_{n\colon 2^{-n}\leq \|\y\|_\s} 
\int_{D_\y} d\x \Big| \int_{(0,\infty)\times\R^d} d\z \, 
\big( F^{(n),0,0}_{\z}(\z;\x+\y) - F^{(n),0,0}_{\x}(\z;\x+\y) \big) \Big| \\
&\,+\sum_{n\colon 2^{-n}\leq \|\y\|_\s} 
\int_{D_\y} d\x \Big| \int_{(0,\infty)\times\R^d} d\z \, 
\big( F^{(n),0,0}_{\z}(\z;\x) - F^{(n),0,0}_{\x}(\z;\x) \big) \Big| \\
&\,+\sum_{n\colon 2^{-n}\leq \|\y\|_\s} 
\int_{D_\y} d\x \Big| \int_{(0,\infty)\times\R^d} d\z \sum_{|k|_1=1}
\big( F^{(n),0,k}_{\z}(\z;\x) - F^{(n),0,k}_{\x}(\z;\x) \big) y^k \Big| \\
&\,+\sum_{n\colon 2^{-n}> \|\y\|_\s} 
\int_{D_\y} d\x \Big| \int_{(0,\infty)\times\R^d} d\z 
\sum_{\substack{\hphantom{2}l+|k|_1\leq2,\\2l+|k|_1\geq2\hphantom{,}}} \\
&\hphantom{\,+\sum_{n\colon 2^{-n}> \|\y\|_\s} \int_{D_\y} d\x}\times\!
\int_0^1 d\vartheta\, 
\big( F^{(n),l,k}_{\z}(\z;\x+\S^\vartheta\y) - F^{(n),l,k}_{\x}(\z;\x+\S^\vartheta\y) \big) s^l y^k \Big| \, .
\end{align}

\textbf{Step 2 \textnormal{(smuggling in test functions)}.}
To estimate these expressions by reconstruction it is convenient to smuggle in a suitable test function. 
It is furthermore convenient to remember in the test function that 
$F_\x^{(n),l,k}(\z;\bar\y)$ vanishes unless the time component 
of $\bar\y-\z$ is positive. 
Unfortunately $\1_{(0,\infty)}$ is not smooth, 
which we fix by decomposing the indicator as in \cite[Section~6.1]{Hai14} as follows.
Recall
$\varphi\colon\R\to[0,1]$ with $\textrm{supp}\,\varphi\subseteq[\frac12,2]$ 
such that $\sum_{m\in\Z} \varphi(2^m s) = 1$ for all $s>0$ from \eqref{eq:decompose_phi>},
and 
$\zeta\colon\R^d\to[0,1]$ with $\textrm{supp}\,\zeta\subseteq \{x\in\R^d\,|\,|x|<1\}$ 
such that $\sum_{k\in\Z^d} \zeta(x+k) = 1$ for all $x\in\R^d$ from \eqref{eq:defK}. 
Finally, for $m\in\Z$ define the set 
\begin{equation}
S^m \coloneqq \{ \bar\x=(\bar{t},\bar{x})\in\R\times\R^d \, |\, 
\bar{t}=2^{-2(m+1)} \textnormal{ and } \bar{x}\in2^{-m}\Z^d \} \, .
\end{equation}
One can then check that $\varphi_{\bar\x,m}$ defined by 
\begin{equation}
\varphi_{\bar\x,m}(s,y) \coloneqq \varphi(2^m \sqrt{s}) \zeta(2^m(y-\bar{x})) 
\end{equation}
for $s>0$ and extended by $0$ for $s\leq0$, 
is a smooth partition of $(0,\infty)\times\R^d$, i.e.
\begin{equation}
\sum_{m\in\Z} \sum_{\bar\x\in S^m} \varphi_{\bar\x,m}(s,y) = \1_{(0,\infty)}(s) \, .
\end{equation}
Replacing $(s,y)$ by $\Dd(\yy-\z)$ we thus obtain
\begin{equation}
F^{(n),l,k}_\x(\z;\bar\y) 
= \sum_{m\in\Z} \sum_{\bar\x\in S^m} \varphi_{\bar\x,m}(\Dd(\bar\y-\z)) F^{(n),l,k}_\x(\z;\bar\y) \, .
\end{equation}
Since $F_\x^{(n),l,k}(\z;\yy)$ is for $\yy=(\bar{s},\bar{y})$ and $\z=(r,z)$ supported on 
$$2^{-n-1}<\|(f(\x)(\bar{s}-r),\bar{y}-z)\|_\s<2^{-n+1} \textnormal{ and }
2^{-q-1}<f(\x)<2^{-q+1} \, ,$$ 
and since $\varphi_{\bar\x,m}(\Dd(\yy-\z))$ is supported on $2^{-m-1}<\sqrt{2^{-q}(\bar{s}-r)}<2^{-m+1}$ 
we deduce that their product $\varphi_{\bar\x,m} F_\x^{(n),l,k}$ is only non-vanishing if 
$2^{-m-1}<2^{-n+3/2}$, i.e.~$m\geq n-2$. 
Similarly, $\varphi_{\bar\x,m}(\Dd(\yy-\z))$ is supported on $|\bar{y}-z-\bar{x}|<2^{-m}$, 
hence effectively 
$|\bar{x}|\leq|\bar{y}-z-\bar{x}|+|\bar{y}-z|<2^{-m}+2^{-n+1}<10\cdot2^{-n}$.
Furthermore there is a $\psi\in C \mathcal{B}$ (for some $C>0$) such that 
$\varphi_{\bar\x,m}(\Dd(\yy-\z)) = 2^{-mD} \psi_{\Dd(\yy-\bar\x)}^{2^{-m}}(\Dd\z)$.
Altogether this yields
\begin{equation}\label{eq:smuggleTestfunction}
F^{(n),l,k}_\x(\z;\bar\y) 
= \sum_{m\geq n-2} 2^{-mD} \sum_{\substack{\bar\x\in S^m \\ |\bar{x}|<10\cdot2^{-n}}} 
\psi_{\Dd(\yy-\bar\x)}^{2^{-m}}(\Dd\z) F^{(n),l,k}_\x(\z;\bar\y) \, .
\end{equation}
Let us point out that the kernel $\psi$ is supported on negative times: 
by definition 
\begin{align}
\psi(s,y) 
&= \varphi_{\bar{\x},m}\big((2^{-q}\bar{t},\bar{x})-\S^{2^{-m}}(s,y)\big) \\
&= \varphi\big(2^m \sqrt{2^{-q}\bar{t}-2^{-2m}s}\,\big) \,
\zeta\big(2^m (\bar{x}-2^{-m}y-\bar{x})\big) \\
&= \varphi\big(\sqrt{2^{-q}2^{-2}-s}\,\big) \, \zeta(-y) \, ,
\end{align}
where in the last equality we used that\footnote{Note that it is here where the choice of $\bar{t}$ in the definition of $S^m$ matters, 
whereas $\varphi_{\bar\x,m}$ is independent of this choice of $\bar{t}$.} 
$\bar\x=(\bar{t},\bar{x})\in S^m$, 
hence the claim follows from $\varphi$ being supported in $(1/2,2)$ and $q\in\N$.

\textbf{Step 3 \textnormal{(reconstruction)}.}
We prove for $n,l\in\N$, $k\in\N^d$, $\|\y\|_\s\lesssim1$, and $\vartheta\in[0,1]$ that 
\begin{align}
&\int_{D_\y} d\x \Big| \int_{(0,\infty)\times\R^d} d\z \, 
\big( F_{\z}^{(n),l,k}(\z;\x+\S^\vartheta\y) - F_{\x}^{(n),l,k}(\z;\x+\S^\vartheta\y) \big) \Big| \\
&\lesssim C 2^{q(\theta_2+2)/2} 2^{-n(\theta_1+2-2l-|k|_1)} (2^{-n}+\vartheta\|\y\|_\s)^{\theta_2} \, .
\end{align}
By \eqref{eq:smuggleTestfunction} and the triangle inequality we have
\begin{align}
&\int_{D_\y} d\x \Big| \int_{(0,\infty)\times\R^d} d\z \, 
\big( F_{\z}^{(n),l,k}(\z;\x+\S^\vartheta\y) - F_{\x}^{(n),l,k}(\z;\x+\S^\vartheta\y) \big) \Big| \\
&\leq 
\sum_{m\geq n-2} 2^{-mD} 
\sum_{\bar\x\in S^m\colon\, |\bar{x}|<10\cdot2^{-n}} 
\int_{D_\y} d\x\\
&\times\! \Big|\! \int_{(0,\infty)\times\R^d}\! d\z 
\big( F_{\z}^{(n),l,k}(\z;\!\x\!+\!\S^\vartheta\y) \!-\! F_{\x}^{(n),l,k}(\z;\!\x\!+\!\S^\vartheta\y) \big) 
\psi_{\Dd(\x+\S^\vartheta\y-\bar\x)}^{2^{-m}}(\Dd\!\z) \Big|
\end{align}
for some test function $\psi$ supported in the past. 
Decomposing $\psi$ as in the proof of Lemma~\ref{lem:integration2} 
to eliminate its $q$-dependence this is bounded by 
\begin{align}
&\sum_{m\geq n-2} 2^{-mD} 
\sum_{\bar\x\in S^m\colon\, |\bar{x}|<10\cdot2^{-n}} 
\sum_{i=1}^{2^q} 
\int_{D_\y} d\x \\
&\times\! \Big|\! \int_{(0,\infty)\times\R^d} \!d\z  
\big( F_{\z}^{(n),l,k}(\z;\!\x\!+\!\S^\vartheta\y) \!-\! F_{\x}^{(n),l,k}(\z;\!\x\!+\!\S^\vartheta\y) \big) 
\tilde\psi_{\x+\S^\vartheta\y-\bar\x-(2^{-2m}r_i,0)}^{2^{-m}}(\z) \Big|
\end{align}
for some $r_i\in(0,2^q)$ and $\tilde\psi\in C\mathcal{B}$ supported in the past.
As a consequence of the assumption and 
general reconstruction Theorem~\ref{thm:reconstruction}, 
this is estimated by 
\begin{align}
\sum_{m\geq n-2} 2^{-mD} 
\hspace{-3ex}
\sum_{\substack{\bar\x\in S^m \\ |\bar{x}|<10\cdot2^{-n}}} 
\hspace{-3ex}
\sum_{i=1}^{2^q} 
C 2^{n(D-2+2l+|k|_1)} 
2^{-m\theta_1} (2^{-m}+\|\S^\vartheta\y-\bar\x-(2^{-2m}r_i,0)\|_\s)^{\theta_2} \, .
\end{align}
For $\bar\x\in S^m$ with $|\bar x|\lesssim2^{-n}$ we have $\|\bar\x\|_\s\lesssim 2^{-m}+2^{-n}$, 
and since $m\geq n-2$ and $1,r_i\leq2^q$ we have by the triangle inequality
$2^{-m}+\|\S^\vartheta\y-\bar\x-(2^{-2m}r_i,0)\|_\s\lesssim 2^{q/2}(2^{-n}+\vartheta\|\y\|_\s)$. 
Furthermore, $\sum_{\bar\x\in S^m, \, |\bar{x}|\lesssim 2^{-n}}1$ is of the order $(2^{-n}/2^{-m})^d$, 
hence the above expression is further estimated by 
\begin{align}
\sum_{m\geq n-2} 
C 2^{q(\theta_2+2)/2} 2^{n(2l+|k|_1)} 2^{-m(\theta_1+2)} (2^{-n}+\vartheta\|\y\|_\s)^{\theta_2} \, .
\end{align}
Since the exponent on $2^{-m}$ is positive this geometric series is summable 
and bounded by its largest summand, yielding the desired estimate.

\textbf{Step 4 \textnormal{(conclusion)}.}
Combining Steps~1 and 3 it remains to argue that 
\begin{align}
&\sum_{n\colon 2^{-n}\leq \|\y\|_\s} \big(
2^{-n(\theta_1+2)} (2^{-n}+\|\y\|_\s)^{\theta_2}
+2^{-n(\theta_1+\theta_2+2)} 
+2^{-n(\theta_1+\theta_2+1)} \|\y\|_\s \big) \\
&+ \sum_{n\colon 2^{-n}> \|\y\|_\s} 
\sum_{\substack{\hphantom{2}l+|k|_1\leq2,\\2l+|k|_1\geq2\hphantom{,}}}
2^{-n(\theta_1+2-2l-|k|_1)} (2^{-n}+\|\y\|_\s)^{\theta_2} \|\y\|_\s^{2l+|k|_1} 
\lesssim \|\y\|_\s^{2-} \, .
\end{align}
The three geometric series are each bounded by their largest summand as the exponents on $2^{-n}$ are all positive, where for the first term we use additionally $2^{-n}+\|\y\|_\s\leq 2\|\y\|_\s$. 
In the geometric sum we first use $2^{-n}+\|\y\|_\s\leq 2\cdot2^{-n}$. 
We only know that $2l+|k|_1\geq2$, hence the exponent $\theta_1+\theta_2+2-2l-|k|_1$ on $2^{-n}$ does not have a sign. 
We remedy this by giving up $2^{-n(\theta_1+\theta_2+\epsilon)}\leq1$, 
so that the remaining $2^{-n(2-2l-|k|_1-\epsilon)}$ has an overall positive exponent and the  corresponding geometric sum is bounded by its largest summand. 
\end{proof}

Finally we provide a version of reconstruction which allows to ``test with indicators''.

\begin{lemma}\label{lem:reconstruction4}
Let $F_\x(\z)$ for 
$\x\in[0,T]\times\R^d$ and $\z\in[0,\infty)\times\R^d$
be jointly $1$-periodic in its spatial arguments, 
i.e.~$F_{\x}(\z) = F_{\x+(0,k)}(\z+(0,k))$ for all $k\in\Z^d$. 

Assume that there exist 
a smooth $\varphi$ compactly supported in the past with $\int\varphi\neq0$ and 
$\theta_1,\theta_2\in\R$ with $\theta_1+\theta_2>0$, $\theta_1>-2$, $\theta_2\geq0$ 
such that 
\begin{align}
\int_{D_{\y}} d\x \Big| \int_{(0,\infty)\times\R^d} d\z \, 
\big(F_{\x+\y}(\z) - F_{\x}(\z) \big) \varphi_\x^\lambda (\z) \Big|
\leq C_1 \lambda^{\theta_1} (\lambda+\|\y\|_\s)^{\theta_2} \, ,
\end{align}
for $\y$ in compacts and $0<\lambda\leq\sqrt{T}$.

Assume furthermore that there exist $\theta_3,\theta_4\in\R$ with $\theta_3>-2$, $\theta_4\geq0$ such that
\begin{align}
\int_{(0,T)\times\T^d} d\x \Big| \int_{(0,\infty)\times\R^d} d\z \, 
F_{\x}(\z) \psi_{\x-(\bar t,0)}^\lambda(\z) \Big|
\leq C_2 \lambda^{\theta_3} (\lambda+\sqrt{\bar t\,}\,)^{\theta_4} \, ,
\end{align}
for smooth $\psi$ compactly supported in the past, 
$0<\lambda\leq\sqrt{T}$, and $0\leq\bar t\leq T$.

Then 
\begin{equation}
\Big| \int_{(0,T)\times\T^d} d\z \, F_{\z}(\z) \Big|
\lesssim C_1 (\sqrt{T})^{\theta_1+\theta_2} + C_2 (\sqrt{T})^{\theta_3+\theta_4} \, .
\end{equation}
\end{lemma}

\begin{proof}
\textbf{Step 1 \textnormal{(smuggling in test functions)}.}
In this first step we smuggle in another integral $\int_{(0,T)\times\T^d} d\x$ and a suitable test function $\psi_\x(\z)$. 
To do so recall from the definition \eqref{eq:defK} of $K$ the smooth $\zeta\colon\R^d\to[0,1]$ 
with $\mathrm{supp}\,\zeta\subseteq \{x\in\R^d\,|\,|x|<1\}$
such that $\sum_{k\in\Z^d} \zeta(x+k) = 1$ for all $x\in\R^d$. 
Then for $1$-periodic $f$ we have 
\begin{align}
\int_{(0,T)\times\T^d} d\z \, f(\z)
&= \frac{1}{T} \int_{(0,T)\times\T^d} d\x \int_{(0,T)\times\T^d} d\z \, f(\z) \\
&= \frac{1}{T} \int_{(0,T)\times\T^d} d\x \sum_{k\in\Z^d} \int_{(0,T)\times\T^d} d\z\, f(\z) \, 
\zeta(x-z+k) \\
&= \frac{1}{T} \int_{(0,T)\times\T^d} d\x \int_{(0,T)\times\R^d} d\z\, f(\z) \, 
\zeta(x-z) \, .
\end{align}
Noting that $t\in(0,T)$ and $t-r\in(0,T)$ imply $r<t<T$, 
yields
\begin{equation}
\int_{(0,T)\times\T^d} d\z \, f(\z) 
= \frac{1}{T} \int_{(0,T)\times\T^d} d\x 
\int_{(0,\infty)\times\R^d} d\z\, f(\z) \zeta(x-z)\1_{(0,T)}(t-r) \, .
\end{equation}
Unfortunately the indicator $\1_{(0,T)}$ is not a suitable test function, which we remedy by a decomposition similar to the one in Step~2 of the proof of Lemma~\ref{lem:reconstruction0}. 
Similar to before, we fix a smooth $\varphi\colon\R\to[0,1]$ with 
$\textrm{supp}\,\varphi\subseteq[\frac14,1]$ 
such that $\sum_{m\in\Z} \varphi(2^m s) = 1$ for all $s>0$. 
Then define $\varphi_m(t)\coloneqq\varphi(2^{|m|} t)$ for $m\geq1$, 
$\varphi_m(t)\coloneqq\varphi_{-m}(1-t)$ for $m\leq-1$ 
(made such that $\varphi_m$ is the around $1/2$ mirrored version of $\varphi_{-m}$), 
and $\varphi_0\coloneqq \1_{(0,1)}-\sum_{m\in\Z\setminus\{0\}}\varphi_m$.
This yields a smooth decomposition $\sum_{m\in\Z} \varphi_m=\1_{(0,1)}$. 
Finally, for $m\in\Z$ define the set 
\begin{align}
S^m \coloneqq \{ \bar\x=(\bar{t},\bar{x})\in\R\times\R^d \, |\, 
&\bar{t}=2^{-2(|m|+2)} \textnormal{ if } m\geq1, \, 
\bar{t}=(1\!-\!2^{-|m|})^2 \textnormal{ if } m\leq-1, \\
&\bar{t}=(\tfrac14)^2 \textnormal{ if } m=0,
\textnormal{ and } \bar{x}\in2^{-|m|}\Z^d \} \, ,
\end{align}
and note that $\bar t$ is the square of the left end point of the support of $\varphi_m$, see Figure~\ref{fig:dyadic}. 
%
%
\begin{figure}[h]
\centering
\includegraphics{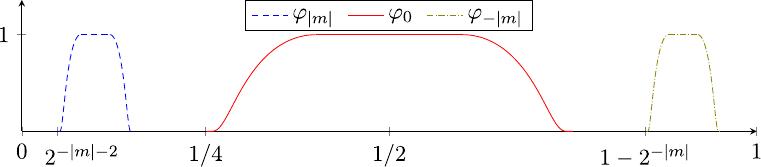}
\caption{Visualization of the dyadic decomposition $\sum_{m\in\Z}\varphi_m=\1_{(0,1)}$ of Step~1 of the proof of Lemma~\ref{lem:reconstruction4}.}
\label{fig:dyadic}
\end{figure}
One can then check that $\varphi_{\bar\x,m}$ defined by 
\begin{equation}
\varphi_{\bar\x,m}(s,y) \coloneqq \varphi_m(\sqrt{s}) \zeta(2^{|m|}(y-\bar{x})) 
\end{equation}
for $s>0$ and extended by $0$ for $s\leq0$, 
is a smooth partition of $(0,1)\times\R^d$, i.e.
\begin{equation}
\sum_{m\in\Z} \sum_{\bar\x\in S^m} \varphi_{\bar\x,m}(s,y) = \1_{(0,1)}(s) \, .
\end{equation}
In particular 
\begin{equation}
\1_{(0,T)}(s) 
= \1_{(0,1)}(\tfrac{s}{T}) 
= \sum_{m\in\Z} \sum_{\bar\x\in S^m} \varphi_{\bar\x,m}(\tfrac{s}{T},y) 
= \sum_{m\in\Z} \sum_{\bar\x\in S^m} \varphi_m(\sqrt{\tfrac{s}{T}}) \zeta(2^{|m|}(y-\bar{x})) \, , 
\end{equation}
and plugging this into the above expression for $\int d\z f(\z)$ yields
\begin{align}
&\int_{(0,T)\times\T^d} d\z \, f(\z) \\
&= \frac{1}{T}\! \int_{(0,T)\times\T^d} \! d\x 
\int_{(0,\infty)\times\R^d} \! d\z\, f(\z) \zeta(x\!-\!z) 
\sum_{m\in\Z} \sum_{\bar\x\in S^m} 
\varphi_m(\sqrt{\tfrac{t-r}{T}}) \zeta(2^{|m|}(x\!-\!z\!-\!\bar{x})) \, .
\end{align}
Note that since $\zeta$ 
is supported in a ball of radius $1$, the sum over $\bar\x\in S^m$ effectively restricts to $|\bar{x}|\leq2$.
Since $2^{-|m|}\leq1$ Lemma~\ref{lem:multtest}~\textit{(i)} yields 
$\zeta(x-z) \zeta(2^{|m|}(x-z-\bar{x}))
= \hat\psi(2^{|m|}(x-z-\bar{x}))$
for some test function $\hat\psi$. 
Furthermore, the set $S^m$ (in particular the values of $\bar{t}$) is chosen such that 
$\varphi_m(\sqrt{\tfrac{t-r}{T}}) = \check\psi(\tfrac{t-r-T\bar{t}}{T 2^{-2|m|}})$ 
for $(\bar{t},\bar{x})\in S^m$ and for some\footnote{$\check\psi$ could depend on the sign of $m$, however such that its $C^n$-norms are uniformly bounded, which is why we drop the dependence in the notation for brevity} test function $\check\psi$ supported on positive times. 
Since $\check\psi(2^{2|m|}(t-r-T\bar{t})/T)$ and $\hat\psi(2^{|m|}(x-z-\bar{x}))$ 
live on different scales (w.l.o.g.~$T\leq1$ hence $\hat\psi$ lives on larger scales)
we decompose $\hat\psi$ further. 
We fix $x_k\in\R^d$ with $|x_k|<2^{-|m|}$ and a\footnote{also here $\tilde\psi$ could depend on $k$, however such that its $C^n$-norms are uniformly bounded, which is why we drop the dependence in the notation for brevity} test function $\tilde\psi$ such that 
$\hat\psi(2^{|m|}(x-z-\bar{x}))
= \sum_{k=1}^{T^{-d/2}} \tilde\psi(2^{|m|} (x-z-\bar{x}-x_k)/\sqrt{T})$.
We thus have
\begin{equation}
\check\psi(2^{2|m|}(t-r-T\bar{t})/T)
\hat\psi(2^{|m|}(x-z-\bar{x})) 
= \sum_{k=1}^{T^{-d/2}} \psi_{\x-(T\bar t,\bar x+x_k)}^{\sqrt{T}2^{-|m|}} (\z) 
2^{-|m|D} T^{D/2}
\end{equation}
for some $\psi\in C\mathcal{B}$ supported in the past. 
Altogether we obtain (for $f(\z)=F_\z(\z)$)
\begin{align}
&\int_{(0,T)\times\T^d} d\z \, F_\z(\z) \\
&= \sum_{m\in\Z} \sum_{\substack{\bar\x\in S^m\\ |\bar{x}|\leq2}} \sum_{k=1}^{T^{-d/2}} 
2^{-|m|D} T^{d/2} 
\int_{(0,T)\times\T^d} d\x 
\int_{(0,\infty)\times\R^d} d\z\, 
F_\z(\z) \psi_{\x-(T\bar t,\bar x+x_k)}^{\sqrt{T}2^{-|m|}}(\z) \, .
\end{align}

\textbf{Step 2 \textnormal{(reconstruction)}.}
Using a change of variables together with periodicity, 
and using the triangle inequality, yield
\begin{align}
&\int_{(0,T)\times\T^d} d\x \Big| \int_{(0,\infty)\times\R^d} d\z \, 
F_\z(\z) \psi_{\x-(T\bar t,\bar x+x_k)}^{\sqrt{T}2^{-|m|}}(\z) \Big| \\
&=\int_{(0,T)\times\T^d} d\x \Big| \int_{(0,\infty)\times\R^d} d\z \, 
F_\z(\z) \psi_{\x-(T\bar{t},0)}^{\sqrt{T}2^{-|m|}}(\z) \Big| \\
&\leq
\int_{(0,T)\times\T^d} d\x \Big| \int_{(0,\infty)\times\R^d} d\z \, 
\big( F_\z(\z) - F_\x(\z) \big) \psi_{\x-(T\bar{t},0)}^{\sqrt{T}2^{-|m|}}(\z) \Big| \\
&\,+
\int_{(0,T)\times\T^d} d\x \Big| \int_{(0,\infty)\times\R^d} d\z \, 
F_\x(\z) \psi_{\x-(T\bar{t},0)}^{\sqrt{T}2^{-|m|}}(\z) \Big| \, .
\end{align}
The first right hand side term can be estimated by general reconstruction (Theorem~\ref{thm:reconstruction}) as a consequence of the assumption; 
the second right hand side term can be estimated directly using the assumption,
so that 
\begin{align}
&\int_{(0,T)\times\T^d} d\x \Big| \int_{(0,\infty)\times\R^d} d\z \, 
F_\z(\z) \psi_{\x-(T\bar t,\bar x+x_k)}^{\sqrt{T}2^{-|m|}}(\z) \Big| \\
&\lesssim C_1 (\sqrt{T}2^{-|m|})^{\theta_1} (\sqrt{T}2^{-|m|} + \sqrt{T\bar t}\,)^{\theta_2} 
+ C_2 (\sqrt{T}2^{-|m|})^{\theta_3} (\sqrt{T}2^{-|m|} + \sqrt{T\bar t}\,)^{\theta_4} \, .
\end{align}
Since $2^{-|m|}+\sqrt{\bar t\,}\lesssim1$ this right hand side is bounded by 
\begin{equation}
C_1(\sqrt{T})^{\theta_1+\theta_2} 2^{-|m|\theta_1} 
+ C_2 (\sqrt{T})^{\theta_3+\theta_4} 2^{-|m|\theta_3} \, .
\end{equation}

\textbf{Step 3 \textnormal{(conclusion)}.}
Plugging the estimate of Step~2 into the expression for $\int d\z\, F_\z(\z)$ from Step~1 yields
\begin{align}
&\Big| \int_{(0,T)\times\T^d} d\z \, F_\z(\z) \Big| \\
&\lesssim 
\sum_{m\in\Z} \sum_{\substack{\bar\x\in S^m\\ |\bar{x}|\leq2}} \sum_{k=1}^{T^{-d/2}} 
2^{-|m|D} T^{d/2} 
\big(C_1(\sqrt{T})^{\theta_1+\theta_2} 2^{-|m|\theta_1} 
+ C_2 (\sqrt{T})^{\theta_3+\theta_4} 2^{-|m|\theta_3} \big) \\
&=\sum_{m\in\Z} \sum_{\substack{\bar\x\in S^m\\ |\bar{x}|\leq2}} 2^{-|m|D} 
\big(C_1(\sqrt{T})^{\theta_1+\theta_2} 2^{-|m|\theta_1} 
+ C_2 (\sqrt{T})^{\theta_3+\theta_4} 2^{-|m|\theta_3} \big) \, .
\end{align}
The number of $\bar\x\in S^m$ with $|\bar{x}|\leq2$ is of the order $2^{|m|d}$, 
hence this right hand side is further estimated by 
\begin{equation}
\sum_{m\in\Z} 
2^{-2|m|} 
\big(C_1(\sqrt{T})^{\theta_1+\theta_2} 2^{-|m|\theta_1} 
+ C_2 (\sqrt{T})^{\theta_3+\theta_4} 2^{-|m|\theta_3} \big) \, .
\end{equation}
We conclude the proof by noting that the exponent on $2^{-|m|}$ is positive by assumption and the geometric series thus bounded by its largest summand. 
\end{proof}

\section{Large velocity contributions}\label{sec:averaging}

In the following we frequently use that 
as a consequence of its definition, $\Pi_\x$ is transforming according to 
\begin{align}
\Pi_\x[\noise] &= \Pi_\y[\noise] \, , \\
\Pi_\x[\lolly;\bar{a}] &= \Pi_\y[\lolly;\bar{a}] + \Pi_\x[\lolly;\bar{a}](\y) \, , \label{eq:trafololly}\\
\Pi_\x[\dumb;\bar{a}] &= \Pi_\y[\dumb;\bar{a}] + \Pi_\x[\lolly;\bar{a}](\y) \xi \, , \label{eq:trafodumb} \\
\Pi_\x[\derivativecherry;\bar{a}] &= \Pi_\y[\derivativecherry;\bar{a}] \, , \label{eq:trafoderivativecherry} \\
\Pi_\x[\xnoise] &= \Pi_\y[\xnoise] + (y-x) \xi \label{eq:trafoxnoise}  \, , 
\end{align}
which combined with \eqref{eq:modelbound} yields for $\tau\in\mathcal{T}$
\begin{equation}\label{eq:offdiagonal}
|\big\langle \partial_{\bar{a}}^m \Pi_\x[\tau;\bar{a}] , \varphi_\y^\lambda \big\rangle |
\lesssim [\tau]_{|\tau|} \lambda^{\alpha-2} (\lambda+\|\x-\y\|_\s)^{|\tau|+2-\alpha} \, \bar{a}^{-\mathfrak{e}(\tau)-m} 
\end{equation}
uniformly for $\x,\y$ in compacts, $\lambda\in(0,1)$, $\bar a\in a(\mathrm{supp}\,\sigma)\setminus\{0\}$, $m=0,1$, and $\varphi\in\mathcal{B}$. 
Here we made a small abuse of notation: the constant $[\dumb]_{2\alpha-2}$ in \eqref{eq:offdiagonal} in case of $\tau=\dumb$ should actually be $[\dumb]_{2\alpha-2}+[\lolly]_{\alpha}[\noise]_{\alpha-2}$ in order to be consistent with \eqref{eq:trafodumb} and \eqref{eq:modelbound}, 
which however we ignore for notational simplicity.\footnote{The reason why we keep $[\tau]_{|\tau|}$ at all throughout most estimates and do not hide it in $\lesssim$ is that we hope it helps the reader to match in long estimates, e.g.~right before equation \eqref{eq:XIII}, which terms arise from which contributions.}

\subsection{Contributions of the initial condition and counterterms}\label{sec:initial} 

Both Propositions~\ref{prop:u0} and \ref{prop:u3} are a simple consequence of Lemma~\ref{lem:integration1}.

\begin{proof}[Proof of Proposition~\ref{prop:u0}]
By the definition of $u_0^q$ in Section~\ref{sec:mild} we have
\begin{align}
&u_0^q(\x+\y)-u_0^q(\x)-\nabla u_0^q(\x)\cdot y \\
&= \int_\R dv \int_{\R^d} dz
\Big( \Psi_q\big(a(v),t+s,x+y-z\big) 
- \Psi_q\big(a(v),t,x-z\big) \\
&\hphantom{= \int_\R dv \int_{\R^d} dz\Big( \Psi_q\big(a(v),t+s,x+y-z\big)\ }
- \nabla \Psi_q\big(a(v),t,x-z\big) \cdot y \Big)
\chi(0,z,v) \, .
\end{align}
Applying Lemma~\ref{lem:integration1} 
(with $m=0$, $f(z)=a(v)$, i.e.~$f$ is constant in $z$, and $g(z)=\chi(0,z,v)$) yields 
\begin{align}
&\int_{D_\y}d\x \big| u_0^q(\x+\y)-u_0^q(\x)-\nabla u_0^q(\x)\cdot y \big| \\
&\lesssim T^{1-\alpha} \|\y\|_\s^{2\alpha} \int_\R dv \int_{\T^d} dz \, \big|
\chi(0,z,v)\,(1+a(v)^{-\alpha}) \tilde\varphi(2^q a(v)) \big| \, .
\end{align}
Here $\tilde\varphi$ is smooth and supported in $(1/2,2)$, 
hence $(1+ a^{-\alpha})|\tilde\varphi(2^q a)| \lesssim (1+2^{q\alpha})|\tilde\varphi(2^q a)|$. 
Summing over $q\in\Z$ with $\delta<2^{-q}$ we observe that for such $q$ we have $(1+2^{q\alpha})\lesssim\delta^{-\alpha}$, and furthermore $\sum_{q} |\tilde\varphi(2^q a)|\lesssim 1$. 
Thus the claim follows from noting that 
$\int_\R dv\,|\chi(0,z,v)|\leq |u_0(z)|$. 
\end{proof}


\begin{proof}[Proof of Proposition~\ref{prop:u3}]
By the definition of $u_3^q$ in Section~\ref{sec:mild} we have
\begin{align}
&u_3^q(\x+\y) - u_3^q(\x) - \nabla u_3^q(\x) \cdot y \\
&=\!\int_{(0,\infty)\times\R^d} \hspace{-3ex}d\z \, 
\partial_v \Big(\!\Psi_q\big(a(v),\x\!+\!\y\!-\!\z\big)
\!-\! \Psi_q\big(a(v),\x\!-\!\z\big) 
\!-\! \nabla\Psi_q\big(a(v),\x\!-\!\z\big)\!\cdot\! y \Big)_{|v=u(\z)} \\
&\qquad\times 
\sigma^2(u(\z)) \big( \cdumb^{a(u(\z))}(r) - a(u(\z)) \ccherry^{a(u(\z))}(r) \big) \\
&\,+\!\int_{(0,\infty)\times\R^d}\hspace{-3ex} d\z \, 
\Big(\!\Psi_q\big(a(u(\z)),\x\!+\!\y\!-\!\z\big)
\!-\! \Psi_q\big(a(u(\z)),\x\!-\!\z\big)
\!-\! \nabla\Psi_q\big(a(u(\z)),\x\!-\!\z\big)\!\cdot\! y\Big) \\
&\qquad\times
\sigma'(u(\z))\sigma(u(\z)) \big( \cdumb^{a(u(\z))}(r) - C^{a(u(\z))} \big) \, .
\end{align}
Applying Lemma~\ref{lem:integration1} 
(once with $m=1$ and $g=a'(u)\sigma^2(u)(\cdumb^{a(u)}-a(u)\ccherry^{a(u)})$,
once with $m=0$ and $g=\sigma'\sigma(\cdumb^{a(u)}-C^{a(u)})$,
and both times with $f=a(u)$)
yields 
\begin{align}
&\int_{D_{\y}} d\x\,
\Big|u_3^q(\x+\y)-u_3^q(\x)-\nabla u_3^q(\x) \cdot y \Big| 
\lesssim T^{1-\alpha} \|\y\|_\s^{2\alpha} \\
&\times\!\Big( 
\int_{(0,T)\times\T^d}\! d\z \big| \big(a'\!\sigma^2\!/a (1\!+\!a^{-\alpha}) \tilde\varphi(2^qa)\big)(u(\z)) \big(\cdumb^{a(u(\z))}(r)-a(u(\z))\ccherry^{a(u(\z))}(r)\big) \big| \\
&\ \ + \int_{(0,T)\times\T^d}\! d\z \big| \big(\sigma'\!\sigma(1+a^{-\alpha}) \tilde\varphi(2^qa)\big)(u(\z)) \big(\cdumb^{a(u(\z))}(r)-C^{a(u(\z))}\big) \big| \Big) \, .
\end{align}
Summing over $q\in\Z$ with $\delta<2^{-q}$ we argue as in the proof of Proposition~\ref{prop:u0} to estimate $\sum_{\delta<2^{-q}} (1+a^{-\alpha})|\tilde\varphi(2^qa)|\lesssim \delta^{-\alpha}$. 
Since Assumption~\ref{ass:nonlinearities} implies $|a'(v)/a(v)|\lesssim|v|^{-1}$ the claim follows from Assumption~\ref{ass:noise}. 
\end{proof}

\subsection{Contributions of the singular forcing}\label{sec:singularforcing}

\begin{proof}[Proof of Proposition~\ref{prop:u1}]
Appealing to the definitions of $u_1^q$ and $\Pi_\x^q$ in Section~\ref{sec:mild}, we observe
\begin{align}
&u_1^q(\x\!+\!\y)\!-\!u_1^q(\x)\!-\!\sigma(u(\x))\Pi_\x^q[\lolly;\x](\x\!+\!\y) 
\!-\! \big(\nabla u_1^q(\x) \!-\! \sigma(u(\x)) \nabla\Pi_\x^q[\lolly;\x](\x) \big) \!\cdot\! y \\
&=\int_{(0,\infty)\times\R^d} d\z \, 
\bar\Psi_q\big(a(u(\z)),\x,\y,\z\big) \sigma(u(\z))\xi(\z) \\ 
&\,-\int_{(0,\infty)\times\R^d} d\z \, 
\partial_v\Big( \bar\Psi_q\big(a(v),\x,\y,\z\big) \sigma(v) \Big)_{|v=u(\z)} 
\sigma(u(\z)) \cdumb^{a(u(\z))}(r) \\
&\,-\int_{(0,\infty)\times\R^d} d\z \, 
\bar\Psi_q \big(a(u(\x)),\x,\y,\z\big) \sigma(u(\x)) \xi(\z) \, , 
\end{align}
where $\bar\Psi_q\big(\bar a,\x,\y,\z\big)$ denotes the anisotropic second order Taylor remainder 
\begin{equation}
\Psi_q \big(\bar a,\x+\y-\z\big) 
- \Psi_q \big(\bar a,\x-\z\big)
- \nabla \Psi_q \big(\bar a,\x-\z\big) \cdot y \, .
\end{equation}
Since all terms contain an instance of $\sigma(u)$, we effectively have to deal with bounded $u$, hence we may assume w.l.o.g.~that $a\leq1$ so that we only need to consider $q\in\N$. 
As in Section~\ref{sec:roughkernels} we replace $\Psi_q$ by $K$ defined in \eqref{eq:defK}. 
In the following we smuggle in suitable terms. 
To shorten the expressions a bit we 
write $\K(\bar{a},\x,\y,\z)$ as in \eqref{eq:anisotropicTaylor} to denote the anisotropic Taylor remainder. 
Furthermore we use the shorthand notation 
$\fnoise=f(u)$ and $\flolly=f'(u)$.

We start by rewriting 
\begin{align}
&\K\big(\anoise(\z),\x,\y,\z\big) \snoise(\z) \xi(\z)
- \partial_v \big( \K\big(a(v),\x,\y,\z\big) \sigma(v) \big)_{|v=u(\z)} \snoise(\z) \cdumb^{\anoise(\z)}(r) \\
&\,- \K\big(\anoise(\x),\x,\y,\z\big) \snoise(\x) \xi(\z) \\
&= \partial_v \big( \K \big(a(v),\x,\y,\z\big) \sigma(v) \big)_{|v=u(\x)}
\snoise(\x) \Pi_\x[\dumb;\x](\z) \label{eq:nf1} \\
&\,+ \partial_v \big( \K\big(a(v),\x,\y,\z\big) \sigma(v) \big)_{|v=u(\x)}
\nu(\x) \cdot \Pi_\x[\xnoise](\z) \label{eq:nf2} \\
&\,+ \K\big(\anoise(\z),\x,\y,\z\big) \snoise(\z) \xi(\z)
- \partial_v\big( \K\big(a(v),\x,\y,\z\big) \sigma(v) \big)_{|v=u(\z)} 
\snoise(\z) \cdumb^{\anoise(\z)}(r) \\ 
&\quad- \K\big(\anoise(\x),\x,\y,\z\big) \snoise(\x) \xi(\z) \\
&\quad- \partial_v \big( \K\big(a(v),\x,\y,\z\big) \sigma(v) \big)_{|v=u(\x)} 
\big( \snoise(\x) \Pi_\x[\dumb;\x](\z) 
+ \nu(\x)\cdot\Pi_\x[\xnoise](\z) \big) \, . \label{eq:nf3}
\end{align}
We estimate \eqref{eq:nf1} and \eqref{eq:nf2} in Step~1, and \eqref{eq:nf3} in Step~2 below. 

\textbf{Step 1 \textnormal{(estimate of \eqref{eq:nf1} and \eqref{eq:nf2})}.}
Note that 
\begin{equation}
\partial_v\big(\K\big(a(v),\cdot\big)\sigma(v)\big)_{|v=u(\x)}
= \partial_{\bar a}\K\big(\bar a,\cdot\big)_{|\bar a=\anoise(\x)} \alolly(\x) \snoise(\x) 
+ \K\big(\anoise(\x),\cdot\big) \slolly(\x) \, ,
\end{equation}
so that Lemma~\ref{lem:integration2} can be applied: 
once with 
$m=1$, $f_\x=\alolly(\x)\ssnoise(\x)\Pi_\x[\dumb;\x]$, $\theta_1=\alpha-2$, $\theta_2=\alpha$,
once with 
$m=0$, $f_\x=\slolly(\x)\snoise(\x)\Pi_\x[\dumb;\x]$, $\theta_1=\alpha-2$, $\theta_2=\alpha$,
once with
$m=1$, $f_\x=\alolly(\x)\snoise(\x)\nu(\x)\cdot\Pi_\x[\xnoise]$, $\theta_1=\alpha-2$, $\theta_2=1$,
and once with
$m=0$, $f_\x=\slolly(\x)\nu(\x)\cdot\Pi_\x[\xnoise]$, $\theta_1=\alpha-2$, $\theta_2=1$,
to deduce 
\begin{align}
&\int_{D_{\y}} d\x \Big| \int_{(0,\infty)\times\R^d} d\z \, 
\big( \eqref{eq:nf1} + \eqref{eq:nf2} \big) \Big| \\
&\lesssim \|\y\|_\s^{2\alpha} [\dumb]_{2\alpha-2} 
\int_{D_\y} d\x \, \Big( \frac{|\alolly(\x)\ssnoise(\x)\tilde\varphi(2^q\anoise(\x))|}{\anoise(\x)^{1+1+\alpha/2+\mathfrak{e}(\dumbsmall)}}
+ \frac{|\slolly(\x)\snoise(\x)\tilde\varphi(2^q\anoise(\x))|}{\anoise(\x)^{1+\alpha/2+\mathfrak{e}(\dumbsmall)}}\Big) \\
&\,+ \|\y\|_\s^{\alpha+1} [\xnoise]_{\alpha-1} \!
\int_{D_\y} \! d\x \, \Big( \frac{|\alolly(\x)\snoise(\x)\nu(\x)\tilde\varphi(2^q\anoise(\x))|}{\anoise(\x)^{1+1+1/2}}
+ \frac{|\slolly(\x)\nu(\x)\tilde\varphi(2^q\anoise(\x))|}{\anoise(\x)^{1+1/2}}\Big) \, ,
\end{align}
where we directly used \eqref{eq:offdiagonal} 
(and $\mathfrak{e}(\xnoise)=0$, see Table~\ref{tab:hom}).
Using \eqref{eq:Theta} and 
that Assumption~\ref{ass:nonlinearities}~(ii) and \eqref{eq:restrictionM} imply 
%
\begin{equation}
\max\Big\{
\Big|\frac{a'\sigma^2}{a^{\alpha/2+2+\mathfrak{e}(\dumbsmall)}}\Big| \, , \ 
\Big|\frac{\sigma'\sigma}{a^{\alpha/2+1+\mathfrak{e}(\dumbsmall)}}\Big|\, , \ 
\Big|\frac{a'\sigma}{a^{5/2}}\Big| \, , \ 
\Big|\frac{\sigma'}{a^{3/2}}\Big| \Big\} \lesssim \frac{1}{a^{\alpha}} 
\tag{II}
\end{equation}
we obtain 
\begin{align}
&\int_{D_{\y}} d\x \Big| \int_{(0,\infty)\times\R^d} d\z \, 
\big( \eqref{eq:nf1} + \eqref{eq:nf2} \big) \Big| \\
&\lesssim \|\y\|_\s^{2\alpha} 2^{q\alpha} [\dumb]_{2\alpha-2} T
+ \|\y\|_\s^{\alpha+1} 2^{q\alpha} [\xnoise]_{\alpha-1} \|\Theta'(u)\nu\|_{L^1} \, . \label{eq:nf5}
\end{align}
Summing over $q\in\N$ with $\delta<2^{-q}$ this is estimated as desired since $\sum_{\delta<2^{-q}} 2^{q\alpha}\lesssim\delta^{-\alpha}$. 

\textbf{Step 2 \textnormal{(estimate of \eqref{eq:nf3})}.}
We aim to apply Lemma~\ref{lem:reconstruction0} to 
\begin{align}
F_\x(\z;\yy) 
&\coloneqq K\big(\anoise(\x),\yy-\z\big) \snoise(\x) \xi(\z) \label{eq:F0}\\
&\ + \partial_v \big( K\big(a(v),\yy-\z\big) \sigma(v) \big)_{|v=u(\x)}
\big( \snoise(\x) \Pi_\x[\dumb;\x](\z) 
+ \nu(\x) \cdot\Pi_\x[\xnoise](\z) \big) \, .
\end{align}
Note that 
\begin{align}
F_\z(\z;\yy) 
&= K\big(\anoise(\z),\yy-\z\big) \snoise(\z) \xi(\z) \\
&\,- \partial_v \big( K\big(a(v),\yy-\z\big) \sigma(v) \big)_{v=u(\z)}  
\snoise(\z) \cdumb^{\anoise(\z)}(r) \, ,
\end{align}
so that (recall from Lemma~\ref{lem:reconstruction0} that $\bar F$ denotes the Taylor remainder of $F$) 
\begin{equation}
\eqref{eq:nf3} 
= \bar F_\z(\z;\x,\y) - \bar F_\x(\z;\x,\y) \, .
\end{equation}
The following lemma verifies the assumption of Lemma~\ref{lem:reconstruction0}
(recall that $F^{(n),l,k}$ is defined analogously to $F$ with $K$ replaced by $\partial_t^l\partial_x^k \Kn$).

\begin{lemma}\label{lem:reconstruction1}
Let $n,l,q\in\N$, $k\in\N^d$, $\y,\y'\in\mathfrak{K}\subseteq\R\times\R^d$ for compact $\mathfrak{K}$, $0<\lambda\leq2^{-n}$, $\varphi\in\mathcal{B}$, $\beta\in(1,2)$, and $\bar\alpha\coloneqq\min\{1+\alpha/2,\beta\}$. 
Then $F$ defined in \eqref{eq:F0} satisfies 
\begin{align}
&\int_{D_{\y}\cap D_{\y'}} d\x \Big| \int_{(0,\infty)\times\R^d} d\z \, 
\big(F^{(n),l,k}_{\x+\y}(\z;\x+\y') - F^{(n),l,k}_{\x}(\z;\x+\y') \big) \varphi_\x^\lambda (\z) \Big| \\
&\lesssim 
2^{n(D-2+2l+|k|_1)} 
\lambda^{\alpha-2} (\lambda+ \|\y\|_\s)^{\bar\alpha} C_{\eqref{eq:nf3}} 2^{-q} \, ,
\end{align}
where $\lesssim$ depends on $\mathfrak{K}$ and the constant $C$ of Assumption~\ref{ass:nonlinearities}, and where 
\begin{align}
C_{\eqref{eq:nf3}} 
&= [\xi]_{\alpha-2} 
\Big( [\Theta(\mathcal{U})]_{B^{\beta}_{1,\infty}} 
+ [\lolly]_\alpha^2 T 
+ [\lolly]_\alpha \|\Theta'(u)\nu\|_{L^1} 
+ \|\Theta'(u)\nu\|_{L^2}^2 \Big) \\
&\,+ [\dumb]_{2\alpha-2} 
[\Theta(u)]_{B_{1,\infty}^\alpha} \\
&\,+ [\xnoise]_{\alpha-1} 
\Big( [\Theta(u)]_{B^\alpha_{1,\infty}}^{1/2} \|\Theta'(u)\nu\|_{L^2}  
+ [\Theta'(u)\nu]_{B^{\beta-1}_{1,\infty}} \Big) \, ,
\end{align}
with $\Theta$ from \eqref{eq:Theta}.
\end{lemma}

Before giving the proof of Lemma~\ref{lem:reconstruction1} we continue Step~2 in the proof of Proposition~\ref{prop:u1}. 
Lemma~\ref{lem:reconstruction1} verifies the assumption of Lemma~\ref{lem:reconstruction0} with $\theta_1=\alpha-2$ and $\theta_2=\min\{1+\alpha/2,\beta\}$. 
Since by assumption $\alpha-2+\beta>0$ and $3\alpha-2>0$, Lemma~\ref{lem:reconstruction0} yields
\begin{align}
\int_{D_\y} d\x \Big| \int_{(0,\infty)\times\R^d} d\z\, \eqref{eq:nf3} \Big| 
\lesssim C_{\eqref{eq:nf3}} 2^{-q} 2^{q(2\alpha+2)/2} \|\y\|_\s^{2-} \, ,
\end{align}
where we have directly used that $\min\{1+\alpha/2,\beta\}\leq 2\alpha$. 
From $\|\y\|_\s\leq R\leq1$ follows $\|\y\|_\s^{2-}\leq\|\y\|_\s^{2\alpha}R^{1-\alpha}$ and hence the desired bound from summing over $q\in\N$ with $\delta<2^{-q}$. 
\end{proof}

\begin{proof}[Proof of Lemma~\ref{lem:reconstruction1}]
We start with having a closer look on $F_{\x+\y} - F_{\x}$,
\begin{align}
&F_{\x+\y}(\z;\yy) - F_{\x}(\z;\yy) \\
&= K\big(\anoise(\x+\y),\yy-\z\big) \snoise(\x+\y) \xi(\z) \\
&\,+ \partial_v \big( K\big(a(v),\yy-\z\big) \sigma(v) \big)_{|v=u(\x+\y)} \\[-1ex]
&\quad\times
\big( \snoise(\x+\y) \Pi_{\x+\y}[\dumb;\x+\y](\z) 
+ \nu(\x+\y) \cdot\Pi_{\x+\y}[\xnoise](\z) \big) \\
&\,- K\big(\anoise(\x),\yy-\z\big) \snoise(\x) \xi(\z) \\
&\,- \partial_v \big( K\big(a(v),\yy-\z\big) \sigma(v) \big)_{|v=u(\x)} 
\big( \snoise(\x) \Pi_\x[\dumb;\x](\z) 
+ \nu(\x) \cdot\Pi_\x[\xnoise](\z) \big) \, .
\end{align}
Using \eqref{eq:trafodumb} in form of 
$\Pi_{\x}[\dumb;\x] = \Pi_{\x+\y}[\dumb;\x] + \Pi_{\x}[\lolly;\x](\x+\y) \xi$
and \eqref{eq:trafoxnoise} in form of 
$\Pi_{\x}[\xnoise] = \Pi_{\x+\y}[\xnoise] + y \xi$,
we obtain 
\begin{align}
&F_{\x+\y}(\z;\yy) - F_{\x}(\z;\yy) \\
&= K\big(\anoise(\x+\y),\yy-\z\big) \snoise(\x+\y) \xi(\z) \\
&\,- K\big(\anoise(\x),\yy-\z\big) \snoise(\x) \xi(\z) \\
&\,- \partial_v \big( K\big(a(v),\yy-\z\big) \sigma(v) \big)_{v=u(\x)} 
\big( \snoise(\x) \Pi_{\x}[\lolly;\x](\x+\y) 
+ \nu(\x) \cdot y \big) \xi(\z) \\
&\,+ \partial_v \big( K\big(a(v),\yy-\z\big) \sigma(v) \big)_{v=u(\x+\y)} \\[-1ex]
&\quad\times
\big( \snoise(\x+\y) \Pi_{\x+\y}[\dumb;\x+\y](\z) 
+ \nu(\x+\y) \cdot\Pi_{\x+\y}[\xnoise](\z) \big) \\
&\,- \partial_v \big( K\big(a(v),\yy-\z\big) \sigma(v) \big)_{|v=u(\x)}
\big( \snoise(\x) \Pi_{\x+\y}[\dumb;\x](\z) 
+ \nu(\x) \cdot\Pi_{\x+\y}[\xnoise](\z) \big) \, .
\end{align}
Note that because of the presence of $\sigma$, all instances of $u$ 
(also those implicit in $\anoise=a(u)$ and $\snoise=\sigma(u)$) 
can be replaced by $\bar u\coloneqq \Theta(u)$ with $\Theta$ from \eqref{eq:Theta}, 
and all instances of $\nu$ can be replaced by $\bar\nu\coloneqq\Theta'(u)\nu$. 
Denoting $G(v)\coloneqq K(a(v),\yy-\z)\sigma(v)$ 
and $G'(v)=\partial_v G(v)$, we rewrite the first three lines of the right hand side as 
\begin{align}
&G\big( \bar u(\x+\y)\big)\xi(\z) 
- G\big( \bar u(\x) + \snoise(\x)\Pi_\x[\lolly;\x](\x+\y) + \bar\nu(\x) \cdot y \big) \xi(\z) \label{eq:mt05} \qquad \\
&+ G\big( \bar u(\x) + \snoise(\x)\Pi_\x[\lolly;\x](\x+\y) + \bar\nu(\x) \cdot y) \xi(\z) \\
&\quad- G\big( \bar u(\x) + \bar\nu(\x)\cdot y\big) \xi(\z)
- G'\big(\bar u(\x)\big) \snoise(\x) \Pi_\x[\lolly;\x](\x+\y) \xi(\z) \label{eq:mt06} \\
&+ G\big(\bar u(\x) + \bar\nu(\x) \cdot y \big) \xi(\z) 
- G\big(\bar u(\x)\big) \xi(\z) 
- G'\big(\bar u(\x)\big) \bar\nu(\x) \cdot y \xi(\z) \, , \label{eq:mt07}
\end{align}
and the last two lines as 
\begin{align}
&G'\big(\bar u(\x+\y)\big) \snoise(\x+\y) \Pi_{\x+\y}[\dumb;\x+\y](\z)
- G'\big(\bar u(\x)\big) \snoise(\x) \Pi_{\x+\y}[\dumb;\x](\z) \qquad \label{eq:mt08} \\
&+ \big( G'\big(\bar u(\x+\y)\big) - G'\big(\bar u(\x)\big) \big) \bar\nu(\x+\y) \cdot \Pi_{\x+\y}[\xnoise](\z) \label{eq:mt09} \\
&+ G'\big(\bar u(\x)\big) \big( \bar\nu(\x+\y) - \bar\nu(\x) \big) \cdot \Pi_{\x+\y}[\xnoise](\z) \label{eq:mt10} \, .
\end{align}
Using the fundamental theorem of calculus we have 
\begin{align}
\eqref{eq:mt05}
= \int_0^1 d\vartheta \, &G'\big( \vartheta \bar u(\x+\y) +(1-\vartheta) \big(\bar u(\x)+\snoise(\x)\Pi_\x[\lolly;\x](\x+\y)+\bar\nu(\x)\cdot y\big)\big) \\
&\times \big( \bar u(\x+\y) - \bar u(\x) - \snoise(\x)\Pi_\x[\lolly;\x](\x+\y) - \bar\nu(\x)\cdot y\big) \xi(\z) \, ,
\end{align}
which by $\snoise=\sigma(u)=\Theta'(u)\sigma(u)$, 
\begin{equation}
G'(v) 
= \partial_{\bar a}K(\bar a,\yy-\z)_{|\bar a=a(v)} a'(v)\sigma(v)
+ K(a(v),\yy-\z)\sigma'(v) \, ,
\end{equation}
and Lemma~\ref{lem:integration3} with $\yy=\x+\y'$ yields for $\lambda\leq2^{-n}$ 
\begin{align}
&\int_{D_{\y}\cap D_{\y'}} d\x \Big| \int_{(0,\infty)\times\R^d} d\z \, 
\varphi_\x^\lambda(\z) \eqref{eq:mt05}^{(n),l,k} \Big| \\
&\lesssim 
2^{n(D-2+2l+|k|_1)} \lambda^{\alpha-2} [\xi]_{\alpha-2} 
\big( \|a'\!\sigma/a\tilde\varphi(2^qa)\|_\infty + \|\sigma'\tilde\varphi(2^qa)\|_\infty \big)
[\Theta(\mathcal{U})]_{B_{1,\infty}^{\beta}} \|\y\|_\s^{\beta} \, .
\end{align}
For \eqref{eq:mt06} we use the fundamental theorem of calculus twice to see
\begin{align}
\eqref{eq:mt06}
&=\! \int_0^1\! d\vartheta\, G'\big( \bar u(\x) \!+\! \vartheta \snoise(\x) \Pi_\x[\lolly;\x](\x\!+\!\y) \!+\! \bar\nu(\x) \!\cdot\! y \big) 
\snoise(\x) \Pi_\x[\lolly;\x](\x\!+\!\y) \xi(\z) \\
&\,- G'\big( \bar u(\x)\big) \snoise(\x) \Pi_\x[\lolly;\x](\x+\y) \xi(\z) \\ 
&= \int_0^1 d\vartheta \int_0^1 d\bar{\vartheta} \, G''\big( \bar u(\x) + \vartheta\bar{\vartheta} \snoise(\x) \Pi_\x[\lolly;\x](\x+\y) + \bar{\vartheta} \bar \nu(\x) \cdot y \big) \\
&\quad\times \snoise(\x) \Pi_\x[\lolly;\x](\x+\y) 
\big( \vartheta \snoise(\x) \Pi_\x[\lolly;\x](\x+\y) + \bar \nu(\x)\cdot y\big) \xi(\z) \, .
\end{align}
Using 
\begin{align}
G''(v) 
&= 
\partial_{\bar a}^2K(\bar a,\yy-\z)_{|\bar a=a(v)} (a'(v))^2 \sigma(v) 
+ \partial_{\bar a}K(\bar a,\yy-\z)_{|\bar a=a(v)} a''(v)\sigma(v) \\
&\,+ 2 \partial_{\bar a}K(\bar a,\yy-\z)_{|\bar a=a(v)} a'(v)\sigma'(v)
+ K(a(v),\yy-\z)\sigma''(v)
\end{align}
and Lemma~\ref{lem:integration3} with $\yy=\x+\y'$ yields for $\lambda\leq2^{-n}$ 
\begin{align}
&\int_{D_{\y}\cap D_{\y'}} d\x \Big| \int_{(0,\infty)\times\R^d} d\z \, 
\varphi_\x^\lambda(\z) \eqref{eq:mt06}^{(n),l,k} \Big| \\
&\lesssim 
2^{n(D-2+2l+|k|_1)} \lambda^{\alpha-2} [\xi]_{\alpha-2} \\
&\,\times \!
\big( \|a^{-2}(a')^2\sigma\tilde\varphi(2^qa)\|_\infty 
\!+\! \|a^{-1}\!a''\!\sigma\tilde\varphi(2^qa)\|_\infty
\!+\! \|a^{-1}\!a'\!\sigma'\tilde\varphi(2^qa)\|_\infty
\!+\! \|\sigma''\tilde\varphi(2^qa)\|_\infty \big) \\
&\,\times \!
\big( \|a^{-\mathfrak{e}(\lollysmall)}\sigma \|^2_\infty [\lolly]_\alpha^2 \|\y\|_\s^{2\alpha} T 
+ \|a^{-\mathfrak{e}(\lollysmall)}\sigma\|_\infty [\lolly]_\alpha \|\y\|_\s^\alpha \|\Theta'(u)\nu\|_{L^1} \|\y\|_\s \big) \, .
\end{align}
Also for \eqref{eq:mt07} we use the fundamental theorem of calculus twice to obtain
\begin{align}
\eqref{eq:mt07}
&= \int_0^1 d\vartheta \, G'\big( \bar u(\x) + \vartheta\bar\nu(\x)\cdot y\big) \bar\nu(\x)\cdot y \xi(\z)
- G'\big( \bar u(\x)\big) \bar\nu(\x)\cdot y \xi(\z) \\ 
&= \int_0^1 d\vartheta \int_0^1 d\bar{\vartheta} \, G''\big( \bar u(\x) + \vartheta\bar{\vartheta}\bar\nu(\x)\cdot y\big) (\bar\nu(\x)\cdot y)^2 \xi(\z) \, .
\end{align}
As before we deduce from Lemma~\ref{lem:integration3} with $\yy=\x+\y'$ that for $\lambda\leq2^{-n}$ 
\begin{align}
&\int_{D_{\y}\cap D_{\y'}} d\x \Big| \int_{(0,\infty)\times\R^d} d\z \, 
\varphi_\x^\lambda(\z) \eqref{eq:mt07}^{(n),l,k} \Big| \\
&\lesssim 
2^{n(D-2+2l+|k|_1)} \lambda^{\alpha-2} [\xi]_{\alpha-2} \|\Theta'(u)\nu\|_{L^2}^2 \|\y\|_\s^2 \\
&\times \!
\big( \|a^{-2}(a')^2\sigma\tilde\varphi(2^qa)\|_\infty 
\!+\! \|a^{-1}\!a''\!\sigma\tilde\varphi(2^qa)\|_\infty
\!+\! \|a^{-1}\!a'\!\sigma'\tilde\varphi(2^qa)\|_\infty
\!+\! \|\sigma''\tilde\varphi(2^qa)\|_\infty \big) .
\end{align}
For \eqref{eq:mt08} we use the fundamental theorem of calculus to see
\begin{align}
&\eqref{eq:mt08} \\
&= \int_0^1 d\vartheta \, \partial_v \big( G'(v) \sigma(v) \Pi_{\x+\y}[\dumb;a(v)](\z) \big)_{|v=\vartheta \bar u(\x+\y)+(1-\vartheta)\bar u(\x)} \big( \bar u(\x+\y) - \bar u(\x)\big) \, ,
\end{align}
and deduce from Lemma~\ref{lem:integration3} with $\yy=\x+\y'$ that for $\lambda\leq2^{-n}$ 
\begin{align}
&\int_{D_{\y}\cap D_{\y'}} d\x \Big| \int_{(0,\infty)\times\R^d} d\z \, 
\varphi_\x^\lambda(\z) \eqref{eq:mt08}^{(n),l,k} \Big| \\
&\lesssim 
2^{n(D-2+2l+|k|_1)} \lambda^{2\alpha-2} [\dumb]_{2\alpha-2} \\
&\,\times \!
\big( \|a^{-2-\mathfrak{e}(\dumbsmall)}(a')^2\sigma^2\tilde\varphi(2^qa)\|_\infty 
\!+\! \|a^{-1-\mathfrak{e}(\dumbsmall)}a''\sigma^2\tilde\varphi(2^qa)\|_\infty
\!+\! \|a^{-1-\mathfrak{e}(\dumbsmall)}a'\sigma'\sigma\tilde\varphi(2^qa)\|_\infty \\
&\quad+ \|a^{-\mathfrak{e}(\dumbsmall)}\sigma''\sigma\tilde\varphi(2^qa)\|_\infty 
+ \|a^{-1-\mathfrak{e}(\dumbsmall)}a'\sigma\sigma'\tilde\varphi(2^qa)\|_\infty 
+ \|a^{-\mathfrak{e}(\dumbsmall)}(\sigma')^2\tilde\varphi(2^qa)\|_\infty \\
&\quad+ \|a^{-2-\mathfrak{e}(\dumbsmall)}a'\sigma^2\tilde\varphi(2^qa)\|_\infty 
+ \|a^{-1-\mathfrak{e}(\dumbsmall)}\sigma\sigma'\tilde\varphi(2^qa)\|_\infty \big) 
[\Theta(u)]_{B_{1,\infty}^\alpha} \|\y\|_\s^\alpha \, .
\end{align}
Similarly, we have by the fundamental theorem of calculus
\begin{align}
&\eqref{eq:mt09} \\
&= \!\int_0^1\! d\vartheta \, G''\big(\vartheta \bar u(\x\!+\!\y))+(1-\vartheta)\bar u(\x)\big) \big(\bar u(\x\!+\!\y)-\bar u(\x)\big) \bar\nu(\x\!+\!\y) \cdot \Pi_{\x+\y}[\xnoise](\z) \, ,
\end{align}
so that again by Lemma~\ref{lem:integration3} with $\yy=\x+\y'$ we obtain for $\lambda\leq2^{-n}$ 
\begin{align}
&\int_{D_{\y}\cap D_{\y'}} d\x \Big| \int_{(0,\infty)\times\R^d} d\z \, 
\varphi_\x^\lambda(\z) \eqref{eq:mt09}^{(n),l,k} \Big| \\
&\lesssim 
2^{n(D-2+2l+|k|_1)} \lambda^{\alpha-1} [\xnoise]_{\alpha-1} 
[\Theta(u)]_{B_{1,\infty}^\alpha}^{1/2} \|\y\|_\s^{\alpha/2} \|\Theta'(u)\nu\|_{L^2} \\
&\times \!
\big( \|a^{-2}(a')^2\sigma\tilde\varphi(2^qa)\|_\infty 
\!+\! \|a^{-1}\!a''\!\sigma\tilde\varphi(2^qa)\|_\infty
\!+\! \|a^{-1}\!a'\!\sigma'\tilde\varphi(2^qa)\|_\infty
\!+\! \|\sigma''\tilde\varphi(2^qa)\|_\infty \big) ,
\end{align}
where we used Cauchy-Schwarz in form of $\int d\x |\bar u(\x+\y)-\bar u(\x)| |\bar\nu(\x+\y)| \lesssim \|\bar u\|_{L^\infty}^{1/2} [\bar u]_{B_{1,\infty}^\alpha}^{1/2} \|\y\|_\s^{\alpha/2} \|\bar\nu\|_{L^2}$ 
and that $\|\bar u\|_{L^\infty}\lesssim1$ by construction.
For \eqref{eq:mt10} we have by Lemma~\ref{lem:integration3} with $\yy=\x+\y'$ that for $\lambda\leq2^{-n}$ 
\begin{align}
&\int_{D_{\y}\cap D_{\y'}} d\x \Big| \int_{(0,\infty)\times\R^d} d\z \, 
\varphi_\x^\lambda(\z) \eqref{eq:mt10}^{(n),l,k} \Big| \\
&\lesssim 
2^{n(D-2+2l+|k|_1)} \lambda^{\alpha-1} [\xnoise]_{\alpha-1} \\
&\,\times 
\big( \|a^{-1}a'\sigma\tilde\varphi(2^qa)\|_\infty 
+ \|\sigma'\tilde\varphi(2^qa)\|_\infty \big) [\Theta'(u)\nu]_{B_{1,\infty}^{\beta-1}} 
\|\y\|_\s^{\beta-1} \, .
\end{align}

We conclude the proof by balancing $a$ and $\sigma$.
For \eqref{eq:mt05} and \eqref{eq:mt10} we use
that Assumption~\ref{ass:nonlinearities}~(ii) and \eqref{eq:restrictionM} imply 
%
\begin{equation}
\max\Big\{
\Big|\frac{a'\sigma}{a}\Big| \, , \
\big|\sigma'\big| \Big\} \lesssim a \, ,
\tag{III}\label{eq:III}
\end{equation}
so that $\|a^{-1}a'\sigma\tilde\varphi(2^qa)\|_\infty\lesssim 2^{-q}$ 
and $\|\sigma'\tilde\varphi(2^qa)\|_\infty\lesssim 2^{-q}$. 
Analogously, for \eqref{eq:mt06}, \eqref{eq:mt07}, and \eqref{eq:mt09} we use 
%
\begin{equation}
\max\Big\{
\Big|\frac{(a')^2\sigma}{a^2}\Big| \, , \
\Big|\frac{a''\sigma}{a}\Big| \, , \
\Big|\frac{a'\sigma'}{a}\Big| \, , \
\big|\sigma''\big| \Big\} \lesssim a \, ,
\tag{IV}\label{eq:IV}
\end{equation}
and for \eqref{eq:mt06} additionally 
%
\begin{equation}
\Big|\frac{\sigma}{a^{\mathfrak{e}(\lollysmall)}}\Big| \lesssim a \, .
\tag{V}\label{eq:V}
\end{equation}
For \eqref{eq:mt08} we use  
%
\begin{align}
\max\Big\{
\Big|\frac{(a')^2\sigma^2}{a^{2+\mathfrak{e}(\dumbsmall)}}\Big| , 
\Big|\frac{a''\sigma^2}{a^{1+\mathfrak{e}(\dumbsmall)}}\Big| , 
\Big|\frac{a'\sigma\sigma'}{a^{1+\mathfrak{e}(\dumbsmall)}}\Big| , 
\Big|\frac{\sigma\sigma''}{a^{\mathfrak{e}(\dumbsmall)}}\Big| , 
\Big|\frac{(\sigma')^2}{a^{\mathfrak{e}(\dumbsmall)}}\Big| , 
\Big|\frac{a'\sigma^2}{a^{2+\mathfrak{e}(\dumbsmall)}}\Big| , 
\Big|\frac{\sigma\sigma'}{a^{1+\mathfrak{e}(\dumbsmall)}}\Big| \Big\} \lesssim a \, ,
\tag{VI}
\end{align}
yielding the desired estimate. 
\end{proof}

\subsection{Contributions of the renormalized kinetic measure}\label{sec:energy}

\begin{proof}[Proof of Proposition~\ref{prop:u2}]
The definition of $u_2^q$ in Section~\ref{sec:mild} yields
\begin{align}
&u_2^q(\x+\y)-u_2^q(\x)-\nabla u_2^q(\x)\cdot y\\
&=\!\int_{(0,\infty)\times\R^d}\! d\z \, \partial_v \bar\Psi_q\big(a(v),\x,\y,\z\big)_{|v=u(\z)} 
a(u(\z)) \big( |\nabla u(\z)|^2 \!-\! \sigma^2(u(\z)) \ccherry^{a(u(\z))}(r)\big) \, ,
\end{align}
where $\bar\Psi_q(a(v),\x,\y,\z)=\Psi_q(a(v),\x+\y-\z) - \Psi_q(a(v),\x-\z)
- \nabla \Psi_q(a(v),\x-\z)\cdot y$.
As in Section~\ref{sec:roughkernels} we replace $\Psi_q$ by $K$ defined in \eqref{eq:defK}, and we write $\bar K$ as in \eqref{eq:anisotropicTaylor} to denote the anisotropic Taylor remainder. 
Furthermore we momentarily write $K_1(\bar a,\x)\coloneqq\partial_{\bar a} K(\bar a,\x)$, and correspondingly $\bar K_1$ to denote its Taylor remainder. 
With this notation
\begin{align}
&u_2^q(\x+\y)-u_2^q(\x)-\nabla u_2^q(\x)\cdot y\\
&=\int_{(0,\infty)\times\R^{d}} \! d\z\, \K_1\big(a(u(\z)),\x,\y,\z\big) 
(a'\!a)(u(\z)) \Big( |\nabla u(\z)|^2 - \sigma^2(u(\z)) \ccherry^{a(u(\z))}(r) \Big) \, .
\end{align}
By the definition \eqref{eq:gubinelliderivative} of $\nu$, 
and using the same shorthand notation $\fnoise=f(u)$ and $\flolly=f'(u)$ as in the proof of Proposition~\ref{prop:u1}, 
we have 
\begin{align}
&\K_1(a(u(\z)),\x,\y,\z) a'(u(\z)) a(u(\z)) \Big( |\nabla u(\z)|^2 - \sigma^2(u(\z)) \ccherry^{a(u(\z))}(r) \Big) \\
&= \K_1(\anoise(\z),\x,\y,\z) \alolly(\z) \anoise(\z) \Big( 
\ssnoise(\z) |\nabla \Pi_\z[\lolly;\z](\z) |^2 - \ssnoise(\z) \ccherry^{\anoise(\z)}(r) \\
&\hphantom{=\K_1(\anoise(\z),\x,\y,\z) \alolly(\z) \anoise(\z)}
\,+ 2 \snoise(\z) \nabla \Pi_\z[\lolly;\z](\z) \cdot \nu(\z) 
+ |\nu(\z)|^2 \Big) \, ,
\end{align}
which by the definition \eqref{eq:derivativecherry} of $\Pi[\derivativecherry]$ equals
\begin{align}
&\K_1(\anoise(\z),\x,\y,\z) \alolly(\z)
\anoise(\z)\ssnoise(\z) \Pi_\z[\derivativecherry,\z](\z) \label{eq:u2part1} \\
&+ 2 \K_1(\anoise(\z),\x,\y,\z) \alolly(\z) \anoise(\z) \snoise(\z) \nabla \Pi_\z[\lolly;\z](\z) \cdot \nu(\z) \label{eq:u2part2} \\
&+ \K_1(\anoise(\z),\x,\y,\z) \alolly(\z) \anoise(\z) |\nu(\z)|^2 \label{eq:u2part3} \, .
\end{align}
We estimate these three contributions separately in the following three steps. 

\textbf{Step 1 \textnormal{(estimate of \eqref{eq:u2part1})}.}
To estimate \eqref{eq:u2part1} we define 
\begin{equation}\label{eq:F1}
F_\x(\z;\yy) \coloneqq 
K_1\big(\anoise(\x),\yy-\z\big) \alolly(\x) \anoise(\x) \ssnoise(\x) \Pi_\x[\derivativecherry;\x](\z) \, ,
\end{equation}
and note that 
\begin{equation}
F_\z(\z;\yy) 
= K_1\big(\anoise(\z),\yy-\z\big) \alolly(\z) \anoise(\z) \ssnoise(\z) \Pi_\z[\derivativecherry;\z](\z) \, .
\end{equation}
Hence (recall from Lemma~\ref{lem:reconstruction0} that $\bar F$ denotes the Taylor remainder of $F$)
\begin{align}
\eqref{eq:u2part1} 
&= \bar F_\z(\z;\x,\y) \\
&= \bar F_\z(\z;\x,\y) - \bar F_\x(\z;\x,\y) + \bar K_1(\anoise(\x),\x,\y,\z) \alolly(\x)\anoise(\x)\ssnoise(\x)\Pi_\x[\derivativecherry;\x](\z) \, .
\end{align}
The following lemma verifies the assumption of Lemma~\ref{lem:reconstruction0} 
(recall that $F^{(n),l,k}$ is defined analogously to $F$ with $K$ replaced by $\partial_t^l\partial_x^k K^{(n)}$). 
Note that as in the proof of Proposition~\ref{prop:u1}, 
due to the presence of $\sigma$ and its compact support we may assume w.l.o.g.~that $a\leq1$ and consider only $q\in\N$. 

\begin{lemma}\label{lem:reconstruction2}
Let $n,l,q\in\N$, $k\in\N^d$, $\y,\y'\in\mathfrak{K}\subseteq\R\times\R^d$ for compact $\mathfrak{K}$, $0<\lambda\leq2^{-n}$,
and $\varphi\in\mathcal{B}$.
Then $F$ defined in \eqref{eq:F1} satisfies
\begin{align}
&\int_{D_\y\cap D_{\y'}} d\x \Big| \int_{(0,\infty)\times\R^d} d\z\, 
\big( F_{\x+\y}^{(n),l,k}(\z;\x+\y') - F_\x^{(n),l,k}(\z;\x+\y') \big) 
\varphi_\x^\lambda(\z) \Big| \\
&\lesssim 2^{n(D-2+2l+|k|_1)} \lambda^{2\alpha-2} \|\y\|_\s^\alpha [\derivativecherry]_{2\alpha-2} [\Theta(u)]_{B_{1,\infty}^\alpha} 2^{-q(1-\alpha/2)} \, ,
\end{align}
where $\Theta$ is as in \eqref{eq:Theta}, and $\lesssim$ depends on $\mathfrak{K}$ and the constant $C$ of Assumption~\ref{ass:nonlinearities}. 
\end{lemma}

Before giving the proof of Lemma~\ref{lem:reconstruction2} at the end of the section 
we continue Step 1 in the proof of Proposition~\ref{prop:u2}.
We have by the triangle inequality 
\begin{align}
&\int_{D_\y}d\x \Big| \int_{(0,\infty)\times\R^d} d\z \, \eqref{eq:u2part1} \Big| \\
&\leq \int_{D_\y}d\x \Big| \int_{(0,\infty)\times\R^d} d\z \, \big(
\bar F_\z(\z;\x,\y) - \bar F_\x(\z;\x,\y)\big) \Big| \\
&\,+ \int_{D_\y}d\x \Big| \int_{(0,\infty)\times\R^d} d\z \, 
\bar K_1(\anoise(\x),\x,\y,\z) \alolly(\x)\anoise(\x)\ssnoise(\x)\Pi_\x[\derivativecherry;\x](\z) \Big|  \, .
\end{align}
Lemma \ref{lem:reconstruction2} verifies the assumption of Lemma~\ref{lem:reconstruction0} with 
$(\theta_1,\theta_2)=(2\alpha-2,\alpha)$, 
where we note that by assumption $3\alpha-2>0$. 
Thus by Lemma~\ref{lem:reconstruction0} the first right hand side term is bounded by
\begin{equation}
[\derivativecherry]_{2\alpha-2} [\Theta(u)]_{B_{1,\infty}^\alpha} 2^{-q(1-\alpha/2)} 2^{q(\alpha+2)/2} \|\y\|_\s^{2-} \, .
\end{equation}
To estimate the second right hand side term we apply Lemma~\ref{lem:integration2} 
with $m=1$, $f_\x=\alolly(\x)\anoise(\x)\ssnoise(\x)\Pi_\x[\derivativecherry;\x]$, $\theta_1=2\alpha-2$, and $\theta_2=0$, hence it is estimated by 
\begin{equation}
\|\y\|_\s^{2\alpha} [\derivativecherry]_{2\alpha-2} \int_{D_\y} d\x \, 
\Big| \frac{\alolly(\x)\anoise(\x)\ssnoise(\x)\tilde\varphi(2^q\anoise(\x))}{\anoise(\x)^{1+1+\mathfrak{e}(\cherrysmall)}}\Big| \, ,
\end{equation}
where we have directly used \eqref{eq:modelbound} along with $\Pi_\x[\derivativecherry]=\Pi_\y[\derivativecherry]$ (see \eqref{eq:trafoderivativecherry}). 
Using that Assumption~\ref{ass:nonlinearities}~(ii) and \eqref{eq:restrictionM} imply 
%
\begin{equation}
\Big|\frac{a'\sigma^2}{a^{1+\mathfrak{e}(\cherrysmall)}}\Big| \lesssim \frac{1}{a^\alpha} \, ,
\tag{VII}
\end{equation}
we deduce by $\|\y\|_\s\leq R\leq1$
\begin{equation}
\int_{D_\y}d\x \Big| \int_{(0,\infty)\times\R^d} d\z \, \eqref{eq:u2part1} \Big| 
\lesssim 2^{q\alpha} \|\y\|_\s^{2\alpha} [\derivativecherry]_{2\alpha-2} \big([\Theta(u)]_{B_{1,\infty}^\alpha} + T\big) \, .
\end{equation}
Summation over $q\in\N$ with $\delta<2^{-q}$ yields the desired bound.

\textbf{Step 2 \textnormal{(estimate of \eqref{eq:u2part2})}.}
To estimate \eqref{eq:u2part2} we proceed analogously to Step~1 and define 
\begin{equation}\label{eq:F2}
F_\x(\z;\yy) 
\coloneqq K_1(\anoise(\x),\yy-\z) \alolly(\x) \anoise(\x) \snoise(\x) \nabla \Pi_\x[\lolly;\x](\z) \cdot \nu(\x) \, .
\end{equation}
We then have the following analogue of Lemma~\ref{lem:reconstruction2}, again only for $q\in\N$. 
\begin{lemma}\label{lem:reconstruction3}
Let $n,l,q\in\N$, $k\in\N^d$, $\y,\y'\in\mathfrak{K}\subseteq\R\times\R^d$ for compact $\mathfrak{K}$, $0<\lambda\leq2^{-n}$, $\varphi\in\mathcal{B}$, $\beta\in(1,2)$, and $\bar\alpha\coloneqq\min\{\alpha/2,\beta-1\}$.
Then $F$ defined in \eqref{eq:F2} satisfies
\begin{align}
&\int_{D_\y\cap D_{\y'}} d\x \Big| \int_{(0,\infty)\times\R^d} d\z \, 
\big( F_{\x+\y}^{(n),l,k}(\z;\x+\y') - F_\x^{(n),l,k}(\z;\x+\y') \big) 
\varphi_\x^\lambda(\z) \Big| \\
&\lesssim 2^{n(D-2+2l+|k|_1)} \lambda^{\alpha-1} \|\y\|_\s^{\bar\alpha} 
[\lolly]_{\alpha} 
\Big( [\Theta'(u)\nu]_{B_{1,\infty}^{\beta-1}} 
+ [\Theta(u)]_{B_{1,\infty}^\alpha}^{1/2} \|\Theta'(u)\nu\|_{L^2} \Big) 2^{-q/2} \, ,
\end{align}
where $\Theta$ is as in \eqref{eq:Theta}, and $\lesssim$ depends on $\mathfrak{K}$ and the constant $C$ of Assumption~\ref{ass:nonlinearities}.
\end{lemma}

Before giving the proof of Lemma~\ref{lem:reconstruction3} at the end of the section 
we continue Step~2 in the proof of Proposition~\ref{prop:u2}. 
By the triangle inequality 
\begin{align}
&\int_{D_\y}d\x \Big| \int_{(0,\infty)\times\R^d} d\z \, \eqref{eq:u2part2} \Big| \\
&\leq \int_{D_\y}d\x \Big| \int_{(0,\infty)\times\R^d} d\z \, \big(
\bar F_\z(\z;\x,\y) - \bar F_\x(\z;\x,\y)\big) \Big| \\
&\,+ \int_{D_\y}d\x \Big| \int_{(0,\infty)\times\R^d} d\z \, 
\bar K_1(\anoise(\x),\x,\y,\z) \alolly(\x)\anoise(\x)\snoise(\x)\nabla\Pi_\x[\lolly;\x](\z) \cdot\nu(\x)\Big| \, .
\end{align}
Lemma~\ref{lem:reconstruction3} verifies the assumption of Lemma~\ref{lem:reconstruction0} with $\theta_1=\alpha-1$ and $\theta_2=\min\{\alpha/2,\beta-1\}$. 
Since by assumption $\alpha-2+\beta>0$ and $3\alpha-2>0$, 
the first right hand side term is estimated by 
\begin{equation}
[\lolly]_\alpha \Big( [\Theta'(u)\nu]_{B_{1,\infty}^{\beta-1}} 
+ [\Theta(u)]_{B_{1,\infty}^\alpha}^{1/2} \|\Theta'(u)\nu\|_{L^2} \Big) 
2^{-q/2}2^{q(2\alpha+1)/2} \|\y\|_\s^{2-} \, ,
\end{equation}
where we have directly used that $\min\{\alpha/2,\beta-1\}\leq 2\alpha-1$.
To estimate the second right hand side term we apply Lemma~\ref{lem:integration2} 
with $m=1$, $f_\x=\alolly(\x)\anoise(\x)\snoise(\x)\nu(\x)\cdot\nabla\Pi_\x[\lolly;\x]$, $\theta_1=\alpha-1$, and $\theta_2=0$, 
hence it is estimated by 
\begin{equation}
\|\y\|_\s^{\alpha+1} [\lolly]_{\alpha} \int_{D_\y} d\x \, 
\Big| \frac{\alolly(\x)\anoise(\x)\snoise(\x)\nu(\x)\tilde\varphi(2^q\anoise(\x))}{\anoise(\x)^{1+1+\mathfrak{e}(\lollysmall)}}\Big| \, ,
\end{equation}
where we have directly used \eqref{eq:modelbound} along with $\nabla\Pi_\x[\lolly]=\nabla\Pi_\y[\lolly]$ (see \eqref{eq:trafololly}). 
Using \eqref{eq:Theta} and that Assumption~\ref{ass:nonlinearities}~(ii) and \eqref{eq:restrictionM} imply 
%
\begin{equation}
\Big|\frac{a'\sigma}{a^{1+\mathfrak{e}(\lollysmall)}}\Big| \lesssim \frac{1}{a^\alpha} 
\tag{VIII}
\end{equation}
we deduce by $\|\y\|_\s\leq R\leq1$
\begin{align}
&\int_{D_\y}d\x \Big| \int_{(0,\infty)\times\R^d} d\z \, \eqref{eq:u2part2} \Big| \\
&\lesssim 2^{q\alpha} \|\y\|_\s^{\alpha+1} [\lolly]_\alpha 
\Big( [\Theta'(u)\nu]_{B_{1,\infty}^{\beta-1}} 
+ [\Theta(u)]_{B_{1,\infty}^\alpha}^{1/2} \|\Theta'(u)\nu\|_{L^2}
+ \|\Theta'(u)\nu\|_{L^1} \Big) \, .
\end{align}
Summation over $q\in\N$ with $\delta<2^{-q}$ yields the desired bound.

\textbf{Step 3 \textnormal{(estimate of \eqref{eq:u2part3})}.}
Applying Lemma~\ref{lem:integration1} (with $m=1$, $f=a(u)$, $g=a'(u)a(u)|\nu|^2$) yields\footnote{It is here where the blowup of $[\mathcal{U}^>]_{B_{1,\infty}^{2\alpha}}$ in Proposition~\ref{prop:u>} in terms of $2^q$ and ultimately $\delta$ comes from, as there is no $\sigma$ to compensate.}
\begin{align}
&\int_{D_\y}d\x \Big| \int_{(0,\infty)\times\R^d} d\z \, \eqref{eq:u2part3} \Big| \\
&\lesssim T^{1-\alpha} \|\y\|_\s^{2\alpha} 
\int_{(0,T)\times\T^d} d\z \, |a'(u(\z))| (1+a(u(\z))^{-\alpha}) |\tilde\varphi(2^qa(u(\z)))| |\nu(\z)|^2 \, .
\end{align}
Summing over $q\in\Z$ with $\delta<2^{-q}$ we argue as in the proof of Proposition~\ref{prop:u0} to estimate $\sum_{\delta<2^{-q}} a^{-\epsilon}(1+a^{-\alpha})|\tilde\varphi(2^qa)|\lesssim \delta^{-\alpha-\epsilon}$ for any $\epsilon>0$, 
so it is left to estimate 
\begin{equation}
\int_{(0,T)\times\T^d} d\z \, |a'(u(\z))| a(u(\z))^\epsilon |\nu(\z)|^2 \, .
\end{equation}
This is achieved in the following proposition.  

\begin{proposition}\label{prop:energy}
Let the assumption of Theorem~\ref{thm:main} hold,
in particular let u be the solution of \eqref{eq:spm} and $\nu$ be given in \eqref{eq:gubinelliderivative}. 
Furthermore let $g(v)=\int_0^vdw\,|w|^{p-2}$ for $p>1$.
Then for $T\leq1$ and $\beta\in(2-\alpha,2\alpha]$
\begin{equation}
\int_{(0,T)\times\T^d} d\z \, |g'(u(\z))|\,a(u(\z)) |\nu(\z)|^2 
\lesssim \|u_0\|_{L^p(\T^d)}^p + C_{\eqref{eq:u2part3}} + 1 \, , 
\end{equation}
where $\lesssim$ depends on $\beta$ and the constants $C$ of Assumptions~\ref{ass:nonlinearities} and~\ref{ass:noise}, 
and 
\begin{align}
&C_{\eqref{eq:u2part3}} \\
&=
[\xi]_{\alpha-2} 
\Big( [(g\sigma)(\mathcal{U})]_{B^{\beta}_{1,\infty}} 
+ (\sqrt{T})^{\alpha} \Big) \\
&+\! \big( [\dumb]_{2\alpha-2} + [\derivativecherry]_{2\alpha-2} \big)
\Big( [\Theta(u)]_{B^\alpha_{1,\infty}} 
+ (\sqrt{T})^{2\alpha} \Big) \\
&+\! \big( [\xnoise]_{\alpha-1} \!+\! [\lolly]_\alpha \big) \!
\Big( [\Theta'(u)\nu]_{B_{1,\infty}^{\beta-1}}
\!+\! [\Theta(u)]_{B_{1,\infty}^\alpha}^{1/2} \|\Theta'(u)\nu\|_{L^2} 
\!+\! \|\Theta'(u)\nu\|_{L^1} (\sqrt{T})^{\alpha-1} \Big) ,
\end{align}
with $\Theta$ from \eqref{eq:Theta}.
\end{proposition}
The proof of Proposition~\ref{prop:energy} is given at the end of this section. 
To continue with Step~3 in the proof of Proposition~\ref{prop:u2} we use Proposition~\ref{prop:energy} with $p=1+\epsilon(M-1)$, 
so that by Assumption~\ref{ass:nonlinearities} we have 
$|a'(u)|a(u)^{\epsilon-1}\lesssim |u|^{-1+\epsilon(M-1)}=|u|^{p-2}$.
It follows 
\begin{equation}
\int_{(0,T)\times\T^d}d\z \, |a'(u(\z))|a(u(\z))^\epsilon|\nu(\z)|^2
\lesssim \|u_0\|_{L^{1+\epsilon(M-1)}(\T^d)}^{1+\epsilon(M-1)} + C_{\eqref{eq:u2part3}} + 1 \, .
\end{equation}
Combining Steps~1 -- 3 finishes the proof of Proposition~\ref{prop:u2}.
\end{proof}


\begin{proof}[Proof of Lemma~\ref{lem:reconstruction2}]
We have by the definition of $F$ in \eqref{eq:F1}
\begin{align}
&F_{\x+\y}(\z;\yy) - F_\x(\z;\yy) \\
&= K_1\big(\anoise(\x+\y),\yy-\z\big) \alolly(\x+\y) \anoise(\x+\y) \ssnoise(\x+\y) \Pi_{\x+\y}[\derivativecherry;\x+\y](\z) \\
&\, - K_1\big(\anoise(\x),\yy-\z\big) \alolly(\x) \anoise(\x) \ssnoise(\x) \Pi_\x[\derivativecherry;\x](\z) \\
&= K_1\big(\anoise(\x+\y),\yy-\z\big) \alolly(\x+\y) \anoise(\x+\y) \ssnoise(\x+\y) \Pi_{\x}[\derivativecherry;\x+\y](\z) \\
&\, - K_1\big(\anoise(\x),\yy-\z\big) \alolly(\x) \anoise(\x) \ssnoise(\x) \Pi_{\x}[\derivativecherry;\x](\z) \, ,
\end{align}
where in the second equality we have used \eqref{eq:trafoderivativecherry} 
in form of $\Pi_\x[\derivativecherry;\x] = \Pi_{\x+\y}[\derivativecherry;\x]$. 
Denoting $G(v)\coloneqq K_1(a(v),\yy-\z)a'(v)a(v)\sigma^2(v)\Pi_{\x}[\derivativecherry;a(v)](\z)$ and observing that $G(u)=G(\bar u)$ for $\bar u\coloneqq\Theta(u)$ with $\Theta$ from \eqref{eq:Theta}, 
we obtain from the fundamental theorem of calculus 
\begin{align}
&F_{\x+\y}(\z;\yy) - F_\x(\z;\yy) = \int_0^1 d\vartheta \, G'\big(\vartheta \bar u(\x+\y)+(1-\vartheta) \bar u(\x)\big) \big(\bar u(\x+\y)-\bar u(\x)\big) \, .
\end{align}
Denoting $K_{11}(\bar a,\x)\coloneqq\partial_{\bar a}^2 K(\bar a,\x)$ note that 
\begin{align}
G'(v)
&= K_{11}(a(v),\yy-\z)(a'(v))^2a(v)\sigma^2(v)\Pi_{\x}[\derivativecherry;a(v)](\z) \\
&\,+ K_1(a(v),\yy-\z)\big(a'a\sigma^2)'(v)\Pi_{\x}[\derivativecherry;a(v)](\z) \\
&\,+ K_1(a(v),\yy-\z)a'(v)a(v)\sigma^2(v)\partial_{a}\Pi_{\x}[\derivativecherry;a(v)](\z)a'(v) \, ,
\end{align}
so that we can apply Lemma~\ref{lem:integration3} to obtain for $\lambda\leq2^{-n}$
\begin{align}
&\int_{D_\y\cap D_{\y'}} d\x \Big| \int_{(0,\infty)\times\R^d} d\z\, 
\big( F_{\x+\y}^{(n),l,k}(\z;\x+\y') - F_\x^{(n),l,k}(\z;\x+\y') \big) \varphi_\x^\lambda(\z) \Big| \\
&\lesssim 	2^{n(D-2+2l+|k|_1)} \lambda^{2\alpha-2} [\derivativecherry]_{2\alpha-2} 
[\Theta(u)]_{B_{1,\infty}^\alpha} \|\y\|_\s^\alpha \\
&\,\times \big(
\|a^{-2-\mathfrak{e}(\cherrysmall)} (a')^2 a \sigma^2 \tilde\varphi(2^q a)\|_\infty 
+ \|a^{-1-\mathfrak{e}(\cherrysmall)} (a' a \sigma^2)' \tilde\varphi(2^q a)\|_\infty \big) \, .
\end{align}
The claim follows by using that Assumption~\ref{ass:nonlinearities}~(ii) and \eqref{eq:restrictionM} imply 
%
\begin{equation}
\max\Big\{
\Big|\frac{(a')^2\sigma^2}{a^{1+\mathfrak{e}(\cherrysmall)}}\Big| \, , \ 
\Big|\frac{(a'a\sigma^2)'}{a^{1+\mathfrak{e}(\cherrysmall)}}\Big| \Big\} \lesssim a^{1-\alpha/2} \, , 
\tag{IX}
\end{equation}
so that $\|a^{-1-\mathfrak{e}(\cherrysmall)} (a')^2 \sigma^2 \tilde\varphi(2^q a)\|_\infty\lesssim 2^{-q(1-\alpha/2)}$, 
and similarly for the remaining $\|\cdot\|_\infty$ term. 
\end{proof}


\begin{proof}[Proof of Lemma~\ref{lem:reconstruction3}]
We have by the definition of $F$ in \eqref{eq:F2}
\begin{align}
&F_{\x+\y}(\z;\yy) - F_\x(\z;\yy) \\
&= K_1\big(\anoise(\x+\y),\yy-\z\big) \alolly(\x\!+\!\y) \anoise(\x\!+\!\y) \snoise(\x\!+\!\y) \nabla\Pi_{\x+\y}[\lolly;\x+\y](\z) \cdot \nu(\x\!+\!\y) \\
&\, - K_1\big(\anoise(\x),\yy-\z\big) \alolly(\x) \anoise(\x) \snoise(\x) \nabla\Pi_\x[\lolly;\x](\z) \cdot \nu(\x) \\
&= K_1\big(\anoise(\x+\y),\yy-\z\big) \alolly(\x+\y) \anoise(\x+\y) \snoise(\x+\y) \nabla\Pi_{\x}[\lolly;\x+\y](\z) \cdot \nu(\x\!+\!\y) \\
&\,- K_1\big(\anoise(\x+\y),\yy-\z\big) \alolly(\x+\y) \anoise(\x+\y) \snoise(\x+\y) \nabla\Pi_{\x}[\lolly;\x+\y](\z) \cdot \nu(\x) \\
&\,+ K_1\big(\anoise(\x+\y),\yy-\z\big) \alolly(\x+\y) \anoise(\x+\y) \snoise(\x+\y) \nabla\Pi_{\x}[\lolly;\x+\y](\z) \cdot \nu(\x) \\
&\, - K_1\big(\anoise(\x),\yy-\z\big) \alolly(\x) \anoise(\x) \snoise(\x) \nabla\Pi_{\x}[\lolly;\x](\z) \cdot \nu(\x) \, ,
\end{align}
where in the second equality we have used \eqref{eq:trafololly} 
in form of $\nabla\Pi_\x[\lolly;\x] = \nabla\Pi_{\x+\y}[\lolly;\x]$. 
Again we replace all instances of $u$ by $\bar u\coloneqq\Theta(u)$ with $\Theta$ from \eqref{eq:Theta}, 
and we replace $\nu$ by $\bar\nu\coloneqq\Theta'(u)\nu$ (for the first and last term this is justified by the presence of $\sigma$, while the two middle terms cancel anyway).
Denoting $G(v)\coloneqq K_1(a(v),\yy-\z) a'(v) a(v)\sigma(v) \nabla\Pi_\x[\lolly;a(v)](\z)$, 
we obtain from the fundamental theorem of calculus
\begin{align}
&F_{\x+\y}(\z;\yy) - F_\x(\z;\yy) \\
&=G\big(\bar u(\x+\y)\big) \cdot \big(\bar\nu(\x+\y)-\bar\nu(\x)\big) \\
&\,+ \int_0^1 d\vartheta \, G'\big(\vartheta \bar u(\x+\y)+(1-\vartheta)\bar u(\x)\big) \big(\bar u(\x+\y)-\bar u(\x)\big) \cdot \bar\nu(\x) \, .
\end{align}
Note that (using $K_{11}(\bar a,\x)=\partial_{\bar a}^2K(\bar a,\x)$ as in the proof of Lemma~\ref{lem:reconstruction2})
\begin{align}
G'(v)
&= K_{11}(a(v),\yy-\z) (a'(v))^2 a(v) \sigma(v) \nabla\Pi_\x[\lolly;a(v)](\z) \\
&\,+ K_1(a(v),\yy-\z) (a' a \sigma)'(v) \nabla\Pi_\x[\lolly;a(v)](\z) \\
&\,+ K_1(a(v),\yy-\z) a'(v) a(v)\sigma(v) \partial_a\nabla\Pi_\x[\lolly;a(v)](\z) a'(v) \, ,
\end{align}
so that we can apply Lemma~\ref{lem:integration3} to obtain for $\lambda\leq2^{-n}$
\begin{align}
&\int_{D_\y\cap D_{\y'}} d\x \Big| \int_{(0,\infty)\times\R^d} d\z\, 
\big( F_{\x+\y}^{(n),l,k}(\z;\x+\y') - F_\x^{(n),l,k}(\z;\x+\y') \big) \varphi_\x^\lambda(\z) \Big| \\
&\lesssim 	2^{n(D-2+2l+|k|_1)} \lambda^{\alpha-1} [\lolly]_{\alpha} [\Theta'(u)\nu]_{B_{1,\infty}^{\beta-1}} \|\y\|_\s^{\beta-1} \|a^{-1-\mathfrak{e}(\lollysmall)}a'a\sigma\tilde\varphi(2^q a)\|_\infty \\
&\,+ 	2^{n(D-2+2l+|k|_1)} \lambda^{\alpha-1} [\lolly]_{\alpha} 
[\Theta(u)]_{B_{1,\infty}^\alpha}^{1/2} \|\y\|_\s^{\alpha/2} \|\Theta'(u)\nu\|_{L^2} \\
&\,\times \big(
\|a^{-2-\mathfrak{e}(\lollysmall)} (a')^2 a \sigma \tilde\varphi(2^q a)\|_\infty
+ \|a^{-1-\mathfrak{e}(\lollysmall)} (a' a \sigma)' \tilde\varphi(2^q a)\|_\infty \big) \, .
\end{align}
As in the estimate of \eqref{eq:mt09} in the proof of Lemma~\ref{lem:reconstruction1} we have used Cauchy-Schwarz and $\|\bar u\|_{L^\infty}\lesssim1$.
The claim follows by using that Assumption~\ref{ass:nonlinearities}~(ii) and \eqref{eq:restrictionM} imply
%
\begin{equation}
\max\Big\{
\Big|\frac{a'\sigma}{a^{\mathfrak{e}(\lollysmall)}}\Big| \, , \ 
\Big|\frac{(a')^2\sigma}{a^{1+\mathfrak{e}(\lollysmall)}}\Big| \, , \ 
\Big|\frac{(a'a\sigma)'}{a^{1+\mathfrak{e}(\lollysmall)}}\Big| \Big\} \lesssim a^{1/2} \, , 
\tag{X}
\end{equation}
so that $\|a^{-1-\mathfrak{e}(\lollysmall)}a'a\sigma\tilde\varphi(2^q a)\|_\infty\lesssim 2^{-q/2}$, 
and similarly for the remaining $\|\cdot\|_\infty$ terms. 
\end{proof}


\begin{proof}[Proof of Proposition~\ref{prop:energy}]
Let 
\begin{equation}
g'(v)\coloneqq |v|^{p-2} \, ,\quad 
g(v)\coloneqq \int_0^v dw \, g'(w) \, ,\quad\textnormal{and}\quad
G(v)\coloneqq \int_0^v dw \, g(w) \, .
\end{equation}
Then $G\in C^1(\R)\cap C^2(\R\setminus\{0\})$ (with derivative $g$, the derivative of which is $g'$), 
$G\geq0$,
and by Assumption~\ref{ass:nonlinearities} $g'\sigma\in C^1(\R)$. 

By the chain rule and \eqref{eq:spm} 
\begin{equation}
\partial_t G(u)
= g(u)\partial_t u 
= g(u) \Big( \nabla\cdot \big(a(u)\nabla u\big) 
+ \sigma(u)\xi 
- \sigma'(u)\sigma(u) C^{a(u)} \Big) \, .
\end{equation}
Integration over $(0,T)\times\T^d$ and integrating by parts\footnote{which can be made rigorous by an approximation argument, e.g.~replace $g'(v)$ by $(|v|+\epsilon)^{p-2}$ and appeal to the monotone convergence theorem} yields
\begin{align}
&\int_{\T^d} dz\, \big(G(u(T,z))-G(u(0,z))\big)
+ \int_{(0,T)\times\T^d} d\z\, g'(u(\z)) a(u(\z)) |\nabla u(\z)|^2 \\
&= \int_{(0,T)\times\T^d} d\z\, g(u(\z)) \Big( \sigma(u(\z)) \xi(\z) 
- \sigma'(u(\z))\sigma(u(\z)) C^{a(u(\z))} \Big) \, .
\end{align}
Again, the left as well as the right hand side of this equality are problematic as 
$|\nabla u|^2$ is a singular product and similarly $g(u)$ multiplied with the renormalized product of $\sigma(u)$ and $\xi$ is singular. 
Luckily it turns out that both need the same counterterm 
(up to an integrable blowup), 
which we now subtract from the equality to arrive at 
\begin{align}
&\int_{\T^d} dz\, \big(G(u(T,z))-G(u(0,z))\big) \\
&\,+ \int_{(0,T)\times\T^d} d\z\, g'(u(\z)) a(u(\z)) \Big( |\nabla u(\z)|^2 - \sigma^2(u(\z)) \ccherry^{a(u(\z))}(r) \Big) \\
&= \int_{(0,T)\times\T^d} d\z\, \Big( g(u(\z)) \sigma(u(\z)) \xi(\z) 
- (g\sigma)'(u(\z))\sigma(u(\z)) \cdumb^{a(u(\z))}(r) \Big) \\
&\,+ \int_{(0,T)\times\T^d} d\z\, g'(u(\z)) \sigma^2(u(\z)) 
\Big( \cdumb^{a(u(\z))}(r) - a(u(\z)) \ccherry^{a(u(\z))}(r) \Big) \\
&\,+ \int_{(0,T)\times\T^d} d\z\, g(u(\z)) \sigma'(u(\z)) \sigma(u(\z)) 
\Big( \cdumb^{a(u(\z))}(r) - C^{a(u(\z))} \Big) \, .
\end{align}
Along the lines that led to \eqref{eq:u2part1} -- \eqref{eq:u2part3} 
we use the definitions \eqref{eq:gubinelliderivative} of $\nu$ and \eqref{eq:derivativecherry} of  $\Pi[\derivativecherry]$ to rewrite 
\begin{align}
&|\nabla u(\z)|^2 - \sigma^2(u(\z)) \ccherry^{a(u(\z))}(r) \\
&=\sigma^2(u(\z)) \Pi_\z[\derivativecherry,\z](\z) 
+ 2 \sigma(u(\z)) \nabla \Pi_\z[\lolly;\z](\z) \cdot \nu(\z) 
+ |\nu(\z)|^2
\end{align}
We thus obtain 
\begin{align}
&\int_{\T^d} dz\, G(u(T,z))
+ \int_{(0,T)\times\T^d} d\z \, g'(u(\z)) a(u(\z)) |\nu(\z)|^2 \\
&= \int_{\T^d} dz\, G(u(0,z)) \\
&\,+ \int_{(0,T)\times\T^d} d\z\, \Big( g(u(\z)) \sigma(u(\z)) \xi(\z) 
- (g\sigma)'(u(\z))\sigma(u(\z)) \cdumb^{a(u(\z))}(r) \Big) 
 \label{eq:energy1} \\
&\, - \int_{(0,T)\times\T^d} d\z\, (g'a\sigma)(u(\z)) 
\Big(\sigma(u(\z)) \Pi_\z[\derivativecherry;\z](\z)
+ 2 \nabla \Pi_\z[\lolly;\z](\z) \cdot \nu(\z) \Big) \label{eq:energy2} \qquad \\
&\,+ \int_{(0,T)\times\T^d} d\z\, g'(u(\z)) \sigma^2(u(\z)) 
\Big( \cdumb^{a(u(\z))}(r) - a(u(\z)) \ccherry^{a(u(\z))}(r) \Big) \\
&\,+ \int_{(0,T)\times\T^d} d\z\, g(u(\z)) \sigma'(u(\z)) \sigma(u(\z)) 
\Big( \cdumb^{a(u(\z))}(r) - C^{a(u(\z))} \Big) \, .
\end{align}
Since $G\geq0$ and $g'(u)a(u)|\nu|^2\geq0$, 
the claim thus follows from estimating the right hand side. 

Note that $G(v)=|v|^p/(p(p-1))$, hence the first right hand side term is estimated by 
$\|u_0\|_{L^p(\T^d)}^p$. 
For the last two terms of the right hand side, 
note that $|g'(u)\sigma^2(u)|\lesssim|\sigma^2(u)/u|$ 
and $|g(u)\sigma'(u)\sigma(u)|\lesssim|\sigma'(u)\sigma(u)|$ 
by $p>1$ and the compact support of $\sigma$ from Assumption~\ref{ass:nonlinearities}, 
hence both terms are estimated by a constant as a consequence of Assumption~\ref{ass:noise}. 

To estimate \eqref{eq:energy1} and \eqref{eq:energy2}
we aim to apply Lemma~\ref{lem:reconstruction4}. 
For \eqref{eq:energy2} define 
\begin{equation}
F_\x(\z)\coloneqq 
(g'a\sigma^2)(u(\x)) \Pi_\x[\derivativecherry;\x](\z)
+ 2(g'a\sigma)(u(\x)) \nabla\Pi_\x[\lolly;\x](\z) \cdot \nu(\x) \, ,
\end{equation}
such that 
\begin{equation}
\eqref{eq:energy2} = \int_{(0,T)\times\T^d} d\z \, F_\z(\z) \, .
\end{equation}
Analogous to the proof of Lemma~\ref{lem:reconstruction2} and Lemma~\ref{lem:reconstruction3} (setting $K_1=1$ and $\alolly=g'(u)$) 
we have for $\y\in\mathfrak{K}\subseteq\R\times\R^d$ for compact $\mathfrak{K}$, $0<\lambda\leq\sqrt{T}$, and $\varphi\in\mathcal{B}$ 
\begin{align}
&\int_{D_\y} d\x \Big| \int_{(0,\infty)\times\R^d} d\z\, 
\big( F_{\x+\y}(\z) - F_\x(\z) \big) \varphi_\x^\lambda(\z) \Big| \\
&\lesssim 
\lambda^{2\alpha-2} \|\y\|_\s^\alpha 
[\derivativecherry]_{2\alpha-2} [\Theta(u)]_{B_{1,\infty}^\alpha} 
\big( \|a^{-\mathfrak{e}(\cherrysmall)}(g'a\sigma^2)'\|_\infty + \|a^{-1-\mathfrak{e}(\cherrysmall)}g'a'a\sigma^2\|_\infty\big) \\
&\,+ \lambda^{\alpha-1} \|\y\|_\s^{\beta-1} 
[\lolly]_\alpha [\Theta'(u)\nu]_{B_{1,\infty}^{\beta-1}} \|a^{-\mathfrak{e}(\lollysmall)}g'a\sigma\|_\infty \\
&\,+ \lambda^{\alpha-1} \|\y\|_\s^{\alpha/2}
[\lolly]_\alpha [\Theta(u)]_{B_{1,\infty}^\alpha}^{1/2} \|\Theta'(u)\nu\|_{L^2} 
\big( \|a^{-\mathfrak{e}(\lollysmall)} (g'\!a\sigma)'\|_\infty + \|a^{-1-\mathfrak{e}(\lollysmall)}g'\!a'a\sigma\|_\infty\big) \, .
\end{align}
Using that Assumption~\ref{ass:nonlinearities}~(ii) and \eqref{eq:restrictionM} imply 
%
\begin{equation}
\max\Big\{
\Big|\frac{(g'a\sigma^2)'}{a^{\mathfrak{e}(\cherrysmall)}}\Big| \, , \ 
\Big|\frac{g'a'\sigma^2}{a^{\mathfrak{e}(\cherrysmall)}}\Big| \, , \ 
\big|\frac{g'\sigma}{a^{\mathfrak{e}(\lollysmall)-1}}\big| \, , \ 
\Big|\frac{(g'a\sigma)'}{a^{\mathfrak{e}(\lollysmall)}}\Big| \, , \ 
\Big|\frac{g'a'\sigma}{a^{\mathfrak{e}(\lollysmall)}}\Big| \Big\} \lesssim 1 \, , 
\tag{XI}
\end{equation}
this right hand side is estimated by 
\begin{align}
\lambda^{2\alpha-2} (\lambda+\|\y\|_\s)^{\bar\alpha}
\big( [\derivativecherry]_{2\alpha-2} [\Theta(u)]_{B_{1,\infty}^\alpha}
&+ [\lolly]_\alpha [\Theta'(u)\nu]_{B_{1,\infty}^{\beta-1}} \\
&+ [\lolly]_\alpha [\Theta(u)]_{B_{1,\infty}^\alpha} \|\Theta'(u)\nu\|_{L^2} \big) \, ,
\end{align}
where $\bar\alpha=\min\{\alpha,\beta-\alpha,1-\alpha/2\}$.
Furthermore, using \eqref{eq:Theta} and the modelbound \eqref{eq:modelbound} combined with 
$\Pi_\x[\derivativecherry]=\Pi_\y[\derivativecherry]$ and $\nabla\Pi_\x[\lolly]=\nabla\Pi_\y[\lolly]$ (see \eqref{eq:trafoderivativecherry} and \eqref{eq:trafololly}), 
yields for $\lambda\leq\sqrt{T}$ and $0\leq\bar t\leq T$
\begin{align}
&\int_{(0,T)\times\T^d} d\x \Big| \int_{(0,\infty)\times\R^d} d\z \, 
F_\x(\z) \psi_{\x-(\bar{t},0)}^{\lambda}(\z) \Big| \\
&\lesssim \lambda^{2\alpha-2} 
[\derivativecherry]_{2\alpha-2} T \|a^{-\mathfrak{e}(\cherrysmall)}g'a\sigma^2\|_\infty 
+ \lambda^{\alpha-1} 
[\lolly]_\alpha \|\Theta'(u)\nu\|_{L^1} \|a^{-\mathfrak{e}(\lollysmall)}g'a\sigma\|_\infty \\
&\lesssim 
\lambda^{2\alpha-2} [\derivativecherry]_{2\alpha-2} T 
+ \lambda^{\alpha-1} [\lolly]_\alpha \|\Theta'(u)\nu\|_{L^1} \\
&\lesssim
\lambda^{2\alpha-2} \big([\derivativecherry]_{2\alpha-2} T 
+ (\sqrt{T})^{1-\alpha} [\lolly]_\alpha \|\Theta'(u)\nu\|_{L^1} \big) \, ,
\end{align}
where in the second inequality we used that Assumption~\ref{ass:nonlinearities}~(ii) and \eqref{eq:restrictionM} imply 
%
\begin{equation}
\max\Big\{
\Big|\frac{g'\sigma^2}{a^{\mathfrak{e}(\cherrysmall)-1}}\Big| \, , \
\Big|\frac{g'\sigma}{a^{\mathfrak{e}(\lollysmall)-1}}\big| \Big\} \lesssim 1 \, .
\tag{XII}
\end{equation}
Thus Lemma~\ref{lem:reconstruction4} can be applied with 
$\theta_1=2\alpha-2$, $\theta_2=\min\{\alpha,\beta-\alpha,1-\alpha/2\}$, 
$\theta_3=2\alpha-2$, and $\theta_4=0$, to obtain 
\begin{align}
|\eqref{eq:energy2}|
&\lesssim [\derivativecherry]_{2\alpha-2} [\Theta(u)]_{B_{1,\infty}^\alpha}
+ [\lolly]_\alpha [\Theta'(u)\nu]_{B_{1,\infty}^{\beta-1}}
+ [\lolly]_\alpha [\Theta(u)]_{B_{1,\infty}^\alpha} \|\Theta'(u)\nu\|_{L^2} \\
&\,+ [\derivativecherry]_{2\alpha-2} T (\sqrt{T})^{2\alpha-2}
+ [\lolly]_{\alpha} \|\Theta'(u)\nu\|_{L^1} (\sqrt{T})^{\alpha-1} \, .
\end{align}
Here we have directly used that $\theta_1+\theta_2>0$ and $T\leq1$.

We turn to \eqref{eq:energy1} where we define (recall that we write $\fnoise=f(u)$) 
\begin{equation}
F_\x(\z) 
\coloneqq \gnoise(\x)\snoise(\x) \xi(\z)
+ (g\sigma)'(u(\x)) \big( \snoise(\x) \Pi_\x[\dumb;\x](\z) + \nu(\x) \cdot \Pi_\x[\xnoise](\z) \big)
\end{equation}
such that 
\begin{equation}
\eqref{eq:energy1} 
= \int_{(0,T)\times\T^d} d\z \, F_\z(\z) \, .
\end{equation}
By definition 
\begin{align}
&F_{\x+\y} - F_\x \\
&= \gnoise(\x+\y)\snoise(\x+\y) \xi
+ (g\sigma)'(u(\x+\y)) \snoise(\x+\y) \Pi_{\x+\y}[\dumb;\x+\y] \\
&\hphantom{= \gnoise(\x+\y)\snoise(\x+\y) \xi \ }
+ (g\sigma)'(u(\x+\y)) \nu(\x+\y) \cdot \Pi_{\x+\y}[\xnoise] \\
&\,- \gnoise(\x)\snoise(\x) \xi
- (g\sigma)'(u(\x)) \snoise(\x) \Pi_\x[\dumb;\x]
- (g\sigma)'(u(\x)) \nu(\x) \cdot \Pi_\x[\xnoise] \, .
\end{align}
Using \eqref{eq:trafodumb} in form of 
$\Pi_{\x}[\dumb;\x] = \Pi_{\x+\y}[\dumb;\x] + \Pi_{\x}[\lolly;\x](\x+\y) \xi$
and \eqref{eq:trafoxnoise} in form of 
$\Pi_{\x}[\xnoise] = \Pi_{\x+\y}[\xnoise] + y \xi$,
we obtain 
\begin{align}
&F_{\x+\y} - F_\x \\
&= \Big( \gnoise(\x\!+\!\y)\snoise(\x\!+\!\y) 
\!-\! \gnoise(\x)\snoise(\x) 
\!-\! (g\sigma)'(u(\x)) \big(\snoise(\x) \Pi_{\x}[\lolly;\x](\x\!+\!\y) \!+\! \nu(\x) \!\cdot\! y \big) \Big) \xi \\
&\,+ (g\sigma)'(u(\x+\y)) \snoise(\x+\y) \Pi_{\x+\y}[\dumb;\x+\y] 
- (g\sigma)'(u(\x)) \snoise(\x) \Pi_{\x+\y}[\dumb;\x] \\
&\,+ \Big( (g\sigma)'(u(\x+\y)) - (g\sigma)'(u(\x)) \Big) \nu(\x+\y) 
\cdot \Pi_{\x+\y}[\xnoise] \\
&\,+ (g\sigma)'(u(\x)) \Big(\nu(\x+\y)-\nu(\x) \Big) \cdot \Pi_{\x+\y}[\xnoise] \, .
\end{align}
In the second, third, and fourth right hand side term we can replace every instance of $u$ by $\Theta(u)$ and every instance of $\nu$ by $\Theta'(u)\nu$.
Thus for $\y\in\mathfrak{K}\subseteq\R\times\R^d$ for compact $\mathfrak{K}$, $0<\lambda\leq\sqrt{T}$, and $\varphi\in\mathcal{B}$
\begin{align}
&\int_{D_\y} d\x \Big| \int_{(0,\infty)\times\R^d} d\z \, 
\big( F_{\x+\y}(\z) - F_\x(\z) \big) \varphi_\x^\lambda(\z) \Big| \\
&\lesssim [(g\sigma)(\mathcal{U})]_{B^{\beta}_{1,\infty}} \|\y\|_\s^{\beta} [\xi]_{\alpha-2} \lambda^{\alpha-2} \\
&\,+ [\Theta(u)]_{B^\alpha_{1,\infty}} \|\y\|_\s^\alpha [\dumb]_{2\alpha-2} \lambda^{\alpha-2}(\lambda+\|\y\|_\s)^\alpha \\
&\quad\times \big( \|a^{-\mathfrak{e}(\dumbsmall)}(g\sigma)''\sigma\|_\infty
+ \|a^{-\mathfrak{e}(\dumbsmall)} (g\sigma)' \sigma'\|_\infty 
+ \|a^{-1-\mathfrak{e}(\dumbsmall)} (g\sigma)' \sigma a'\|_\infty \big) \\
&\,+ \|(g\sigma)''\|_\infty [\Theta(u)]_{B_{1,\infty}^\alpha}^{1/2} \|\y\|_\s^{\alpha/2} \|\Theta'(u)\nu\|_{L^2} 
[\xnoise]_{\alpha-1} \lambda^{\alpha-2}(\lambda+\|\y\|_\s) \\
&\,+ \|(g\sigma)'\|_\infty 
[\Theta'(u)\nu]_{B_{1,\infty}^{\beta-1}} \|\y\|_\s^{\beta-1}
[\xnoise]_{\alpha-1} \lambda^{\alpha-2}(\lambda+\|\y\|_\s) \, ,
\end{align}
where we used the modelbound \eqref{eq:offdiagonal}, 
for the second and third term also the fundamental theorem of calculus, 
and for the third term additionally Cauchy-Schwarz and $\|\Theta(u)\|_{L^\infty}\lesssim1$ 
as in the estimate of \eqref{eq:mt09} in the proof of Lemma~\ref{lem:reconstruction1}.
Using that Assumption~\ref{ass:nonlinearities}~(ii) and \eqref{eq:restrictionM} imply 
%
\begin{equation}
\max\Big\{
\Big|\frac{(g\sigma)''\sigma}{a^{\mathfrak{e}(\dumbsmall)}}\Big| \, , \ 
\Big|\frac{(g\sigma)'\sigma'}{a^{\mathfrak{e}(\dumbsmall)}}\Big| \, , \
\Big|\frac{(g\sigma)'a'\sigma}{a^{1+\mathfrak{e}(\dumbsmall)}}\Big| \, , \
\big|(g\sigma)''\big| \, , \
\big|(g\sigma)'\big| \Big\} \lesssim 1 \, , 
\tag{XIII}
\label{eq:XIII}
\end{equation}
this right hand side is estimated by
\begin{align}
&\lambda^{\alpha-2} (\lambda+\|\y\|_\s)^{\beta} 
\big(  [(g\sigma)(\mathcal{U})]_{B_{1,\infty}^{\beta}} [\xi]_{\alpha-2} 
+ [\Theta'(u)\nu]_{B_{1,\infty}^{\beta-1}} [\xnoise]_{\alpha-1} \big) \\
&+ \lambda^{\alpha-2} (\lambda+\|\y\|_\s)^{2\alpha} 
[\Theta(u)]_{B_{1,\infty}^\alpha} [\dumb]_{2\alpha-2} \\
&+ \lambda^{\alpha-2} (\lambda+\|\y\|_\s)^{1+\alpha/2} 
[\Theta(u)]_{B_{1,\infty}^\alpha}^{1/2} \|\Theta'(u)\nu\|_{L^2} [\xnoise]_{\alpha-1} \, .
\end{align}
Furthermore, using \eqref{eq:Theta} and the modelbound \eqref{eq:offdiagonal} yields for $0<\lambda\leq\sqrt{T}$ and $0\leq\bar t\leq T$
\begin{align}
&\int_{(0,T)\times\T^d} d\x \Big| \int_{(0,\infty)\times\R^d} d\z \, 
F_\x(\z) \psi_{\x-(\bar{t},0)}^{\lambda}(\z) \Big| \\
&\lesssim \lambda^{\alpha-2} [\xi]_{\alpha-2} T \|g\sigma\|_\infty \\
&\,+ \lambda^{\alpha-2} (\lambda + \sqrt{\bar t\,}\,)^{\alpha} 
[\dumb]_{2\alpha-2} T \|a^{-\mathfrak{e}(\dumbsmall)} (g\sigma)' \sigma\|_\infty \\
&\,+ \lambda^{\alpha-2} (\lambda + \sqrt{\bar t\,}\,)
[\xnoise]_{\alpha-1} \|\Theta'(u)\nu\|_{L^1} \|(g\sigma)'\|_\infty \\
&\lesssim \lambda^{\alpha-2} [\xi]_{\alpha-2} T 
+ \lambda^{\alpha-2} (\lambda + \sqrt{\bar t\,}\,)^{\alpha} [\dumb]_{2\alpha-2} T 
+ \lambda^{\alpha-2} (\lambda + \sqrt{\bar t\,}\,)
[\xnoise]_{\alpha-1} \|\Theta'(u)\nu\|_{L^1} \\
&\lesssim
 \lambda^{\alpha-2} \big( [\xi]_{\alpha-2} T 
+(\sqrt{T})^{\alpha} [\dumb]_{2\alpha-2} T 
+\sqrt{T} \, [\xnoise]_{\alpha-1} \|\Theta'(u)\nu\|_{L^1} \big) \, ,
\end{align}
where in the second inequality we have used that Assumption~\ref{ass:nonlinearities}~(ii) and \eqref{eq:restrictionM} imply 
%
\begin{equation}
\max\Big\{\big|g\sigma\big| \, , \
\Big|\frac{(g\sigma)'\sigma}{a^{\mathfrak{e}(\dumbsmall)}}\Big| \, \
\big|(g\sigma)'\big| \Big\} \lesssim 1\, . 
\tag{XIV}
\end{equation}
Thus Lemma~\ref{lem:reconstruction4} can be applied with $\theta_1=\alpha-2$, 
$\theta_2=\min\{\beta,2\alpha,1+\alpha/2\}$, 
$\theta_3 = \alpha-2$, and
$\theta_4=0$, to obtain 
\begin{align}
|\eqref{eq:energy1}|
&\lesssim [(g\sigma)(\mathcal{U})]_{B_{1,\infty}^{\beta}} [\xi]_{\alpha-2} 
+ [\Theta'(u)\nu]_{B_{1,\infty}^{\beta-1}} [\xnoise]_{\alpha-1} \\
&\,+ [\Theta(u)]_{B_{1,\infty}^\alpha} [\dumb]_{2\alpha-2} 
+ [\Theta(u)]_{B_{1,\infty}^\alpha}^{1/2} \|\Theta'(u)\nu\|_{L^2} [\xnoise]_{\alpha-1} \\
&\,+ [\xi]_{\alpha-2} T (\sqrt{T})^{\alpha-2}
+ [\dumb]_{2\alpha-2} T (\sqrt{T})^{2\alpha-2}
+ [\xnoise]_{\alpha-1} \|\Theta'(u)\nu\|_{L^1} (\sqrt{T})^{\alpha-1} \, .
\end{align}
Here we have directly used that $\theta_1+\theta_2>0$ and $T\leq1$.
\end{proof}

\section{Real interpolation and proof of the main result}\label{sec:interpolation}

The main result of this section is to prove Theorem~\ref{thm:modelledness}. 
This is achieved by real interpolation based on Propositions~\ref{prop:u<} and \ref{prop:u>}, more precisely by what is called the $K$-method, 
see \cite{BL76} (or \cite{Lun18} for a more pedagogical introduction). 
Before giving the actual proof, we heuristically outline the strategy in the following.

If we had not estimated $[\mathcal{U^>}]_{B_{1,\infty}^{2\alpha}}$  
but the actual Besov norms
$\|u^<\|_{B_{1,\infty}^0}\lesssim \delta^{1/(M-1)}$ and 
$\|u^>\|_{B_{1,\infty}^{2\alpha}}\lesssim \delta^{-\alpha}$, 
we could define 
\begin{equation}
K(\lambda, u, B_{1,\infty}^0,B_{1,\infty}^{2\alpha})
= \inf_{u_1\in B_{1,\infty}^0,\,u_2\in B_{1,\infty}^{2\alpha},\,u_1+u_2=u} \big( \|u_1\|_{B_{1,\infty}^0} + \lambda \|u_2\|_{B_{1,\infty}^{2\alpha}} \big) \, , 
\end{equation}
which would be estimated by $\delta^{1/(M-1)}+\lambda\delta^{-\alpha}$ for all $\delta\in(0,1)$ since $u=u^<+u^>$. 
Choosing $\delta=\lambda^{(M-1)/(1+M\alpha-\alpha)}$, 
this would be further estimated by $\lambda^{1/(1+M\alpha-\alpha)}$, and thus
\begin{equation}
\sup_{\lambda>0} \lambda^{-\frac{1}{1+M\alpha-\alpha}} K(\lambda,u,B_{1,\infty}^0,B_{1,\infty}^{2\alpha}) < \infty \, .
\end{equation}
This in turn is the norm of the interpolation space $(B_{1,\infty}^0,B_{1,\infty}^{2\alpha})_{\theta,\infty}$ for $\theta=1/(1+M\alpha-\alpha)$, 
which by \cite[Theorem~6.4.5 (1)]{BL76} coincides with $B_{1,\infty}^{\theta2\alpha}$ (with equivalent norms). 

We need an extension of such real interpolation spaces to basepoint dependent seminorms, 
which is conveniently achieved via the following equivalence. 

\begin{lemma}\label{lem:equivalence}
Let $\mathcal{B}_1$ denote the set of smooth functions $\zeta$ on $\R^d$ supported in the unit ball with all its derivatives up to order $2$ bounded by $1$, 
and such that $\int_{\R^d} dy\, \zeta(y)y^k=0$ for $|k|_1\leq1$.
Furthermore let $u_\x(\y)$ be real valued and continuous in $\x,\y\in [0,T]\times\R^d$ with $\nabla_y u_\x(\y)$ continuous, and jointly $1$-periodic in its spatial arguments, i.e.~$u_\x(\y)=u_{\x+(0,k)}(\y+(0,k))$ for all $k\in\Z^d$. 
Then\footnote{here we write $\langle\cdot,\cdot\rangle$ to denote the $L^2$ inner product on $R^d$} 
for $\gamma\in(1,2\alpha)$
(recall that $\x=(t,x)$ and $\y=(s,y)$)
\begin{align}
&\sup_{0<\|\y\|_\s\leq R} \|\y\|_\s^{-\gamma} \int_{D_\y} d\x \, 
\big| u_\x(\x+\y) - u_\x(\x+(0,y)) \big| \\
&+\sup_{0<\lambda\leq R} \lambda^{-\gamma} 
\sup_{\zeta\in\mathcal{B}_1} 
\int_{[0,T]\times\T^d} d\x \, \big| \langle u_\x(t,\cdot),\zeta_x^\lambda\rangle \big| \\
&\lesssim \sup_{0<\|\y\|_\s\leq R} \|\y\|_\s^{-\gamma} \int_{D_\y} d\x \, 
|u_\x(\x+\y)-u_\x(\x)-\nabla u_\x(\x)\cdot y| \\
&\lesssim \sup_{0<\|\y\|_\s\leq R} \|\y\|_\s^{-\gamma} \int_{D_\y} d\x \, 
\big| u_\x(\x+\y) - u_\x(\x+(0,y)) \big| \\
&+\sup_{0<\lambda\leq R} \lambda^{-\gamma} 
\sup_{\zeta\in\mathcal{B}_1} 
\int_{[0,T]\times\T^d} d\x \, |\langle u_\x(t,\cdot), \zeta_x^\lambda \rangle| \\
&\,+ R^{2\alpha-\gamma} \!\!
\sup_{0<\lambda\leq R}\sup_{0<| y|\leq R} \lambda^{-\alpha}| y|^{-\alpha} 
\sup_{\zeta\in\mathcal{B}_1}
\int_{[0,T]\times\T^d}\!\! d\x 
\big| \big\langle u_\x(t,\cdot)-u_{\x+(0, y)}(t,\cdot), \zeta_{x+ y}^{\lambda} \big\rangle\big| \, .
\end{align}
Furthermore let $u_\x\coloneqq u - \sigma(u(\x)) \Pi_\x[\lolly;\x]$ under the assumption of Theorem~\ref{thm:main}. 
Then
\begin{align}
&\sup_{0<\lambda\leq R}\sup_{0<| y|\leq R} \lambda^{-\alpha}| y|^{-\alpha} 
\sup_{\zeta\in\mathcal{B}_1}
\int_{[0,T]\times\T^d} d\x 
\big| \big\langle u_\x(t,\cdot)-u_{\x+(0, y)}(t,\cdot), \zeta_{x+ y}^{\lambda} \big\rangle\big| \\
&\lesssimdata 1 \\
&\,+ (R^{\gamma-\alpha} + R^{1-\alpha}\sqrt{T}) 
\sup_{0<\|\y\|_\s\leq R} \|\y\|_\s^{-\gamma} \int_{D_\y} d\x \, 
\big|u_\x(\x+\y)-u_\x(\x)-\nabla u_\x(\x)\cdot y\big| \, .
\end{align}
The dependence of the implicit constant in $\lesssimdata$ on the data is as in Theorem~\ref{thm:main}. 
\end{lemma}

\begin{proof}
The three inequalities are proved in the following three steps, respectively. 

\textbf{Step 1.} We start with the first inequality. 
The first left hand side term is estimated as desired as an immediate consequence of 
\begin{align}
u_\x(\x+\y)-u_\x(\x+(0,y))
&= u_\x(\x+\y) - u_\x(\x) - \nabla u_\x(\x)\cdot y \\
&\,-\big(u_\x(\x+(0,y)) - u_\x(\x) - \nabla u_\x(\x)\cdot y\big) \, ,
\end{align}
the triangle inequality, and $D_\y\subseteq[0,T]\times\T^d$. 

We turn to the second left hand side term.
Let $\zeta\in\mathcal{B}_1$ and $0<\lambda\leq R$. Then
\begin{align}
&\int_{[0,T]\times\T^d} d\x |\langle u_\x(t,\cdot) , \zeta_x^\lambda\rangle| \\
&= \int_{[0,T]\times\T^d} d\x \Big| \int_{\R^d} dy \, u_\x(\x+(0,y)) \zeta^\lambda(y) \Big| \\
&= \int_{[0,T]\times\T^d} d\x \Big| \int_{\R^d} dy \, \big(u_\x(\x+(0,y))-u_\x(\x)-\nabla u_\x(\x)\cdot y\big) \zeta^\lambda(y) \Big| \\
&\leq \int_{\R^d} dy \, |\zeta^\lambda(y)| 
\int_{[0,T]\times\T^d} d\x \big| u_\x(\x+(0,y))-u_\x(\x)-\nabla u_\x(\x)\cdot y\big| \, ,
\end{align}
where in the second equality we used that $\int dy \zeta(y)y^k = 0$ for $|k|_1\leq1$ 
and in the inequality we used the triangle inequality and Fubini. 
Note that by the support of $\zeta$ we have effectively $|y|<\lambda\leq R$, 
hence the claim follows from 
$\int_{\R^d} dy \, |\zeta^\lambda(y)| |y|^\gamma \lesssim \lambda^{\gamma}$.

\textbf{Step 2.}
We turn to the second inequality, where by the triangle inequality
\begin{align}
\big| u_\x(\x+\y)-u_\x(\x)-\nabla u_\x(\x)\cdot y \big|
&\leq \big| u_\x(\x+\y)-u_\x(\x+(0,y))\big| \\
&\,+ \big|u_\x(\x+(0,y)-\nabla u_\x(\x)\cdot y \big| \, .
\end{align}
It remains to estimate the second right hand side term. 
Let $\zeta$ be smooth, compactly supported in the unit ball, 
and satisfy $\int_{\R^d} dy\,\zeta(y)=1$ and $\int_{\R^d} dy\,\zeta(y)y^k=0$ for $|k|_1=1$. 
Then for $n_0\in\N$ to be chosen we have 
\begin{align}
u_\x(\x+(0,y)) - u_\x(\x) - \nabla u_\x(\x)\cdot y
&=
\big\langle u_\x(t,\cdot), 
\zeta_{x+y}^{2^{-n_0}} - \zeta_x^{2^{-n_0}} + \nabla \zeta_x^{2^{-n_0}} \cdot y\big\rangle \\
&\,+ \sum_{n=0}^\infty \big\langle u_\x(t,\cdot), \zeta_{x+y}^{2^{-n_0-n-1}} - \zeta_{x+y}^{2^{-n_0-n}} \big\rangle \\
&\,- \sum_{n=0}^\infty \big\langle u_\x(t,\cdot), \zeta_{x}^{2^{-n_0-n-1}} - \zeta_{x}^{2^{-n_0-n}} \big\rangle \\
&\,- \sum_{n=0}^\infty \big\langle \nabla u_\x(t,\cdot)\cdot y, \zeta_{x}^{2^{-n_0-n-1}} - \zeta_{x}^{2^{-n_0-n}} \big\rangle \, .
\end{align}
The first right hand side term equals 
\begin{equation}
\int_{\R^d} dz\, u_\x(t,z) 2^{n_0 d} \Big(
\zeta\big(2^{n_0}(z-x-y)\big) - \zeta\big(2^{n_0}(z-x)\big) - \nabla\zeta\big(2^{n_0}(z-x)\big) \cdot(-y)2^{n_0} \Big) \, ,
\end{equation}
which using $f(1)-f(0)-f'(0) = \int_0^1 d\vartheta\, (1-\vartheta) f''(\vartheta)$ applied to $f(\vartheta) = \zeta(2^{n_0}(z-x-\vartheta y))$ we rewrite as 
\begin{align}
&\int_{\R^d} \!dz\, u_\x(t,z) 2^{n_0 d}
\!\int_0^1\! d\vartheta\,(1\!-\!\vartheta) \sum_{|k_1+k_2|_1=2} \!\partial^{k_1+k_2}\zeta\big(2^{n_0}(z\!-\!x\!-\!\vartheta y)\big) (-y)^{k_1+k_2} 2^{2n_0} \\
&= \int_0^1 d\vartheta\,(1-\vartheta) \sum_{|k_1+k_2|_1=2} \big\langle u_\x(t,\cdot), \partial^{k_1+k_2} \zeta_{x+\vartheta y}^{2^{-n_0}} \big\rangle (-y)^{k_1+k_2} \, .
\end{align}
Rewriting
\begin{align}
\big\langle u_\x(t,\cdot),\partial^{k_1+k_2}\zeta_{x+vy}^{2^{-n_0}}\big\rangle
&= \big\langle u_{\x+(0,\vartheta y)}(t,\cdot),\partial^{k_1+k_2}\zeta_{x+\vartheta y}^{2^{-n_0}}\big\rangle \\
&\,+ \big\langle u_\x(t,\cdot) - u_{\x+(0,vy)}(t,\cdot),\partial^{k_1+k_2}\zeta_{x+\vartheta y}^{2^{-n_0}}\big\rangle
\end{align}
and noting that $\partial^{k_1+k_2}\zeta\in C\mathcal{B}_1$ for some $C>0$ 
(since $|k_1+k_2|_1=2$), 
we thus obtain 
\begin{align}
&\int_{D_\y} d\x \big| 
\big\langle u_\x(t,\cdot), 
\zeta_{x+y}^{2^{-n_0}} - \zeta_x^{2^{-n_0}} + \nabla \zeta_x^{2^{-n_0}} \cdot y\big\rangle \big| \\
&\lesssim 2^{-n_0(\gamma-2)} |y|^2 \sup_{0<\lambda\leq R} \lambda^{-\gamma} \sup_{\tilde\zeta\in\mathcal{B}_1} \int_{D_\y} d\x \, |\langle u_\x(t,\cdot),\tilde\zeta_x^\lambda\rangle| \\
&+\! 2^{-n_0(\alpha-2)} |y|^{\alpha+2} \!\!\!
\sup_{0<\lambda\leq R}\sup_{0<|\tilde y|\leq R} \! (\lambda|\tilde y|)^{-\alpha} \!
\sup_{\tilde\zeta\in\mathcal{B}_1} \!
\int_{D_\y}\!\!d\x 
\big| \big\langle u_\x(t,\cdot)-u_{\x+(0,\tilde y)}(t,\cdot), \tilde\zeta_{x+\tilde y}^{\lambda} \big\rangle\big| ,
\end{align}
provided $2^{-n_0}\leq R$ and $|y|\leq R$. 

We turn to estimating $\langle \partial^k u_\x(t,\cdot) y^k , \zeta_{x+\vartheta y}^{2^{-n_0-n-1}} - \zeta_{x+\vartheta y}^{2^{-n_0-n}} \rangle$ for $|k|_1\leq1$ and $\vartheta\in\{0,1\}$. 
Denoting $\hat\zeta\coloneqq\zeta^{1/2}-\zeta$ we have 
$\zeta_{x+\vartheta y}^{2^{-n_0-n-1}}-\zeta_{x+\vartheta y}^{2^{-n_0-n}} = \hat\zeta_{x+\vartheta y}^{2^{-n_0-n}}$ 
and $\partial^k\hat\zeta\in\mathcal{B}_1$ for $|k|_1\leq1$ 
(here we use $\int dy \, \zeta(y)y^k=0$ for $|k|_1=1$). 
Hence 
\begin{align}
&\int_{D_\y} d\x \big| \big\langle \partial^k u_\x(t,\cdot)y^k , 
\zeta_{x+\vartheta y}^{2^{-n_0-n-1}}-\zeta_{x+\vartheta y}^{2^{-n_0-n}} \big\rangle\big| \\
&\leq \int_{D_\y} d\x \big| \big\langle u_{\x+(0,\vartheta y)}(t,\cdot) , 
\partial^k\hat\zeta_{x+\vartheta y}^{2^{-n_0-n}}\big\rangle y^k \big| \\
&\,+ \int_{D_\y} d\x \big| \big\langle u_\x(t,\cdot) - u_{\x+(0,\vartheta y)}(t,\cdot) , 
\partial^k \hat\zeta_{x+\vartheta y}^{2^{-n_0-n}}\big\rangle y^k \big| \\
&\lesssim 2^{-(n_0+n)(\gamma-|k|_1)} |y|^{|k|_1} \sup_{0<\lambda\leq R} \lambda^{-\gamma} 
\sup_{\tilde\zeta\in\mathcal{B}_1} 
\int_{D_\y} d\x |\langle u_\x(t,\cdot), \tilde\zeta_{x}^\lambda \rangle| \\
&\,+ 2^{-(n_0+n)(\alpha-|k|_1)} |\vartheta y|^\alpha |y|^{|k|_1} \\
&\quad\times
\sup_{0<\lambda\leq R} \sup_{0<|\tilde y|\leq R} \lambda^{-\alpha} |\tilde y|^{-\alpha}
\sup_{\tilde\zeta\in\mathcal{B}_1} \int_{D_\y} d\x \big|\big\langle u_\x(t,\cdot) - u_{\x+(0,\tilde y)}(t,\cdot), \tilde\zeta_{x+\tilde y}^\lambda \big\rangle\big| \, ,
\end{align}
again provided $2^{-n_0}\leq R$ and $|y|\leq R$. 
For such $n_0$ and $y$ we thus have
\begin{align}
&\int_{D_\y} d\x \Big| \sum_{n=0}^\infty \big\langle u_\x(t,\cdot), \zeta_{x+y}^{2^{-n_0-n-1}} - \zeta_{x+y}^{2^{-n_0-n}} \big\rangle 
- \sum_{n=0}^\infty \big\langle u_\x(t,\cdot), \zeta_{x}^{2^{-n_0-n-1}} - \zeta_{x}^{2^{-n_0-n}} \big\rangle \\
&\qquad- \sum_{n=0}^\infty \big\langle \nabla u_\x(t,\cdot)\cdot y, \zeta_{x}^{2^{-n_0-n-1}} - \zeta_{x}^{2^{-n_0-n}} \big\rangle \Big| \\
&\lesssim (2^{-n_0\gamma} + 2^{-n_0(\gamma-1)}|y|) 
\sup_{0<\lambda\leq R}\lambda^{-\gamma}
\sup_{\tilde\zeta\in\mathcal{B}_1} 
\int_{D_\y} d\x |\langle u_\x(t,\cdot), \tilde\zeta_{x}^\lambda \rangle| \\
&\,+ 2^{-n_0\alpha}|y|^\alpha 
\sup_{0<\lambda\leq R} \sup_{0<|\tilde y|\leq R} \lambda^{-\alpha} |\tilde y|^{-\alpha}
\sup_{\tilde\zeta\in\mathcal{B}_1} \int_{D_\y} d\x \big|\big\langle u_\x(t,\cdot) - u_{\x+(0,\tilde y)}(t,\cdot), \tilde\zeta_{x+\tilde y}^\lambda \big\rangle\big| \, ,
\end{align}
where we used $\alpha>0$ and $\gamma>1$ such that 
$\sum_{n=0}^\infty (2^{-n\gamma} + 2^{-n(\gamma-1)} + 2^{-n\alpha}) \lesssim 1$. 

For $|y|\leq R$ we thus choose $n_0$ such that $2^{-n_0}\leq R$ and $2^{-n_0}\sim|y|$ and the claim follows. 

\textbf{Step 3.}
We turn to the last inequality. 
First observe that by the definition of $u_\x$ and by \eqref{eq:trafololly} in form of 
$\Pi_{\x}[\lolly;\x] = \Pi_{\x+(0,y)}[\lolly;\x] + \Pi_{\x}[\lolly;\x](\x+(0,y))$ 
\begin{align}
u_\x - u_{\x+(0,y)} 
&= \sigma(u(\x+(0,y))) \Pi_{\x+(0,y)}[\lolly;\x+(0,y)] - \sigma(u(\x))\Pi_\x[\lolly;\x] \\
&= \sigma(u(\x+(0,y))) \Pi_{\x+(0,y)}[\lolly;\x+(0,y)] - \sigma(u(\x))\Pi_{\x+(0,y)}[\lolly;\x] \\
&\,- \sigma(u(\x)) \Pi_{\x}[\lolly;\x](\x+(0,y)) \, .
\end{align}
For $\zeta\in\mathcal{B}_1$ we thus have 
\begin{align}
&\big\langle u_\x(t,\cdot) - u_{\x+(0,y)}(t,\cdot) , \zeta_{x+y}^\lambda \big\rangle \\
&= \big\langle \sigma(u(\x+(0,y))) \Pi_{\x+(0,y)}[\lolly;\x+(0,y)](t,\cdot) - \sigma(u(\x))\Pi_{\x+(0,y)}[\lolly;\x](t,\cdot) , \zeta_{x+y}^\lambda \big\rangle \, .
\end{align}
Denoting $G(v)\coloneqq \sigma(v) \Pi_{\x+(0,y)}[\lolly;a(v)](t,\cdot)$ and observing that $G(v)=G(\Theta(v))$ for $\Theta$ from \eqref{eq:Theta}, 
we have by the fundamental theorem of calculus 
\begin{align}
&\big\langle u_\x(t,\cdot) - u_{\x+(0,y)}(t,\cdot) , \zeta_{x+y}^\lambda \big\rangle \\
&= \int_0^1 d\vartheta\, \big\langle G'\big(\vartheta \Theta(u(\x+(0,y)))+(1-\vartheta)\Theta(u(\x))\big) , \zeta_{x+y}^\lambda\big\rangle \\
&\quad\times \big(\Theta(u(\x+(0,y)))-\Theta(u(\x))\big) \, .
\end{align}
Using 
\begin{equation}
G'(v) 
= \sigma'(v) \Pi_{\x+(0,y)}[\lolly;a(v)](t,\cdot) + \sigma(v)a'(v)\partial_a\Pi_{\x+(0,y)}[\lolly;a(v)](t,\cdot) 
\end{equation}
we obtain from the model estimate \eqref{eq:lollybound} 
\begin{align}
&\int_{[0,T]\times\T^d} d\x \big| \big\langle u_\x(t,\cdot) - u_{\x+(0,y)}(t,\cdot) , \zeta_{x+y}^\lambda \big\rangle \big| \\
&\lesssim \lambda^\alpha [\lolly]_\alpha \big(\|\sigma'\!/\!a^{\mathfrak{e}(\lollysmall)}\|_\infty \!+\! \|\sigma a'\!/\!a^{1+\mathfrak{e}(\lollysmall)}\|_\infty\big)
\!\int_{[0,T]\times\T^d} \! d\x \big| \Theta(u(\x\!+\!(0,y)))\!-\!\Theta(u(\x)) \big| \, .
\end{align}
Assumption~\ref{ass:nonlinearities}~(ii) and \eqref{eq:restrictionM} yield 
%
\begin{equation}
\max\Big\{
\Big|\frac{\sigma'}{a^{\mathfrak{e}(\lollysmall)}}\Big| \, , \ 
\Big|\frac{\sigma a'}{a^{1+\mathfrak{e}(\lollysmall)}}\Big| \Big\} \lesssim 1\, ,
\tag{XV}\label{eq:XV}
\end{equation}
and together with the definition of $[\Theta(u)]_{B_{1,\infty}^\alpha}$ we deduce 
\begin{equation}
\int_{[0,T]\times\T^d} d\x \big| \big\langle u_\x(t,\cdot) - u_{\x+(0,y)}(t,\cdot) , \zeta_{x+y}^\lambda \big\rangle \big|
\lesssim \lambda^\alpha |y|^\alpha [\lolly]_\alpha [\Theta(u)]_{B_{1,\infty}^\alpha} \, .
\end{equation}
The claim follows from 
$[\Theta(u)]_{B_{1,\infty}^\alpha}
\lesssim (R^{\gamma-\alpha}+R^{1-\alpha}\sqrt{T}) [\mathcal{U}]_{B_{1,\infty}^{\gamma}} +1$, 
see Lemma~\ref{lem:absorb}, and noting that $u_\x$ is defined such that 
\begin{equation}
u_\x(\x+\y)-u_\x(\x)-\nabla u_\x(\x)\cdot y
= u(\x+\y)-u(\x)-\sigma(u(\x))\Pi_\x[\lolly;\x](\x+\y)-\nu(\x)\cdot y \, .
\end{equation}
\end{proof}

We can finally give the proof of Theorem~\ref{thm:modelledness}.

\begin{proof}[Proof of Theorem~\ref{thm:modelledness}]
Let $0<R,T\leq1$.
Recall $u_\x=u - \sigma(u(\x)) \Pi_\x[\lolly;\x]$ from Lemma~\ref{lem:equivalence} defined such that
\begin{equation}
u(\x+\y)-u(\x)-\sigma(u(\x))\Pi_\x[\lolly;\x](\x+\y)-\nu(\x)\cdot y
= u_\x(\x+\y)-u_\x(\x)-\nabla u_\x(\x)\cdot y \, .
\end{equation}
Hence by Lemma~\ref{lem:equivalence} we have for $\gamma\in(1,2\alpha)$ and $R$ small enough 
\begin{align}
&[\mathcal{U}]_{B_{1,\infty}^\gamma} \\
&=\sup_{0<\|\y\|_\s\leq R}
\|\y\|_\s^{-\gamma} \int_{D_\y} d\x \big| u(\x+\y)-u(\x)-\sigma(u(\x))\Pi_\x[\lolly;\x](\x+\y)-\nu(\x)\cdot y \big| \\
&\lesssimdata 1 
+ \sup_{0<\|\y\|_\s\leq R} \|\y\|_\s^{-\gamma} \int_{D_\y} d\x \, 
\big| u_\x(\x+\y) - u_\x(\x+(0,y)) \big| \\
&\hphantom{\lesssimdata 1\ }
+\sup_{0<\lambda\leq R} \lambda^{-\gamma} 
\sup_{\zeta\in\mathcal{B}_1} \int_{[0,T]\times\T^d} d\x \big| \langle u_\x(t,\cdot), \zeta_x^\lambda \rangle \big| \, .
\end{align}
In Step~1 we estimate the time increments and in Step~2 we estimate the last right hand side term; 
in Step~3 we combine both estimates and conclude. 
In all steps we use for $\delta\in(0,1)$ the splitting $u=u^<+u^>$ and $\Pi[\lolly]=\Pi^<[\lolly]+\Pi^>[\lolly]$ into small and large velocities as in 
the beginning of Section~\ref{sec:splitting}.

\textbf{Step 1.} 
Let $u_\x^<\coloneqq u^< - \sigma(u(\x)) \Pi^<_\x[\lolly;\x]$ and 
$u_\x^>\coloneqq u^> - \sigma(u(\x)) \Pi^>_\x[\lolly;\x]$. 
Then $u_\x=u_\x^<+u_\x^>$, and by the triangle inequality 
\begin{align}
&\int_{D_\y} d\x \big| u_\x(\x+\y) - u_\x(\x+(0,y)) \big| \\
&\leq \int_{D_\y} d\x \big( |u^<_\x(\x+\y)| + |u^<_\x(\x+(0,y))| + |u^>_\x(\x+\y)-u^>_\x(\x+(0,y))| \big) \, .
\end{align}
Again by the triangle inequality, 
\begin{align}
&\int_{D_\y} d\x \big| u^>_\x(\x+\y)-u^>_\x(\x+(0,y)) \big| \\
&\leq \int_{D_\y} d\x \big| u^>(\x+\y)-u^>(\x)-\sigma(u(\x))\Pi^>_\x[\lolly;\x](\x+\y) - \nu^>(\x) \cdot y \big| \\
&\,+ \int_{D_\y} d\x \big| u^>(\x+(0,y))-u^>(\x)-\sigma(u(\x))\Pi^>_\x[\lolly;\x](\x+(0,y)) - \nu^>(\x)\cdot y \big| \, .
\end{align}
Denoting 
$[\![u^<]\!]\coloneqq \sup_{\|\y\|_\s\leq R} \int_{D_\y} d\x |u^<_\x(\x+\y)|$ 
this yields for $\|\y\|_\s\leq R$
\begin{equation}
\int_{D_\y} d\x \big| u_\x(\x+\y)-u_\x(\x+(0,y)) \big| 
\leq 2 [\![u^<]\!]
+ 2 [\mathcal{U}^>]_{B_{1,\infty}^{2\alpha}} \|\y\|^{2\alpha} \, .
\end{equation}
Since $\delta\in(0,1)$ was arbitrary (but for convenience dyadic, see \eqref{eq:decompose_phi>}) in the decomposition $u=u^<+u^>$ from Section~\ref{sec:splitting}, we deduce for $\gamma\in(1,2\alpha]$ and $\|\y\|_\s\leq R$
\begin{align}
&\|\y\|_\s^{-\gamma} \int_{D_\y} d\x \big| u_\x(\x+\y)-u_\x(\x+(0,y)) \big| \\
&\lesssim \sup_{0<\lambda\leq R} \lambda^{-\gamma} \inf_{\substack{0<\delta<1\\ \delta\textnormal{ dyadic}}}
\big( [\![u^<]\!] + \lambda^{2\alpha} [\mathcal{U}^>]_{B_{1,\infty}^{2\alpha}} \big) \, .
\end{align}

\textbf{Step 2.}
Let $\lambda\leq R$ and $\zeta\in\mathcal{B}_1$. 
Splitting once more $u_\x=u^<_\x+u^>_\x$ the triangle inequality yields
\begin{align}
\int d\x \big| \langle u_\x(t,\cdot),\zeta^\lambda_x\rangle \big| 
\leq \int d\x \int_{\R^d} dz \big| u_\x^<(\x+(0,z))\zeta^\lambda(z) \big|
+ \int d\x \big| \langle u^>_\x(t,\cdot), \zeta^\lambda_x \rangle \big| \, .
\end{align}
For the first right hand side term note that $|z|\leq\lambda\leq R$ by the support of $\zeta$, 
hence it is bounded by $[\![u^<]\!]$. 
The second right hand side term is by Lemma~\ref{lem:equivalence} estimated by 
\begin{equation}
\lambda^{2\alpha} \sup_{0<\|\y\|_\s\leq R} \|\y\|_\s^{-2\alpha} \int_{D_\y} d\x \big|
u^>_\x(\x+\y)-u^>_\x(\x)-\nabla u^>_\x(\x)\cdot y \big| 
\leq \lambda^{2\alpha} [\mathcal{U}^>]_{B_{1,\infty}^{2\alpha}} \, .
\end{equation}
Hence we obtain altogether (again since $\delta$ was arbitrary)
\begin{equation}
\lambda^{-\gamma} 
\sup_{\zeta\in\mathcal{B}_1} \int_{[0,T]\times\T^d} d\x \big| \langle u_\x(t,\cdot), \zeta_x^\lambda \rangle \big| \lesssim \sup_{0<\lambda\leq R} \lambda^{-\gamma} \inf_{\substack{0<\delta<1\\ \delta \textnormal{ dyadic}}}
\big( [\![u^<]\!] + \lambda^{2\alpha} [\mathcal{U}^>]_{B_{1,\infty}^{2\alpha}} \big) \, .
\end{equation}

\textbf{Step 3.}
Combining Steps~1 and 2 we deduce for $\gamma\in(1,2\alpha)$
\begin{align}
[\mathcal{U}]_{B_{1,\infty}^\gamma} 
\lesssimdata 1+ \sup_{0<\lambda\leq R} \lambda^{-\gamma} \inf_{\substack{0<\delta<1\\ \delta\textnormal{ dyadic}}}
\big( [\![u^<]\!] + \lambda^{2\alpha} [\mathcal{U}^>]_{B_{1,\infty}^{2\alpha}} \big) \, .
\end{align}
By Proposition~\ref{prop:u<} and Lemma~\ref{lem:Pi<} we estimate 
\begin{align}
[\![u^<]\!]
&= \sup_{0<\|\y\|_\s\leq R} \int_{D_\y} d\x \, \big|u^<(\x+\y)-\sigma(u(\x))\Pi^<_\x[\lolly;\x](\x+\y) \big| 
\lesssimdata \delta^{1/(M-1)} \, .
\end{align}
Furthermore, by Proposition~\ref{prop:u>} we have for $\beta\in(2-\alpha,2\alpha)$ and $\epsilon>0$ small enough
\begin{equation}
[\mathcal{U}^>]_{B_{1,\infty}^{2\alpha}} 
\lesssimdata \delta^{-\alpha-\epsilon} \big( 1+ (R+T)^{1-\alpha} [\mathcal{U}]_{B_{1,\infty}^\beta} \big) \, ,
\end{equation}
and thus 
\begin{equation}
[\mathcal{U}]_{B_{1,\infty}^\gamma} 
\lesssimdata 1 + \sup_{0<\lambda\leq R} \lambda^{-\gamma} \inf_{\substack{0<\delta<1\\ \delta\textnormal{ dyadic}}}
\big( \delta^{1/(M-1)} + \lambda^{2\alpha} \delta^{-\alpha-\epsilon} (1+ (R+T)^{1-\alpha}[\mathcal{U}]_{B_{1,\infty}^\beta}) \big) \, .
\end{equation}
Choosing $\delta$ such that $\delta^{1/(M-1)}=\lambda^{2\alpha} \delta^{-\alpha-\epsilon}$, 
i.e.~$\delta=\lambda^{2\alpha(M-1)/(1+(M-1)(\alpha+\epsilon))}\in(0,1)$ and hence 
$\delta^{1/(M-1)} = \lambda^{2\alpha/(1+(M-1)(\alpha+\epsilon))}$ 
(since this choice need not be dyadic, actually choose $\delta$ to be the closest dyadic), yields
\begin{equation}
[\mathcal{U}]_{B_{1,\infty}^\gamma} 
\lesssimdata 1 + \sup_{0<\lambda\leq R} \lambda^{\frac{2\alpha}{1+(M-1)(\alpha+\epsilon)}-\gamma}
\big(1+(R+T)^{1-\alpha}[\mathcal{U}]_{B_{1,\infty}^\beta} \big) \, .
\end{equation}
Using that \eqref{eq:restrictionM} is equivalent to 
\begin{equation}
\frac{2\alpha}{1+(M-1)\alpha}> 2-\alpha 
\end{equation}
we can choose $\beta=\gamma\in(2-\alpha,2\alpha/(1+(M-1)(\alpha+\epsilon)]$ to deduce for $R,T$ small enough the desired 
\begin{equation}
[\mathcal{U}]_{B_{1,\infty}^\beta} 
\lesssimdata 1 \, .
\end{equation}
\end{proof}

Theorem~\ref{thm:main} is now a simple consequence of Theorem~\ref{thm:modelledness}. 

\begin{proof}[Proof of Theorem~\ref{thm:main}]
By the triangle inequality and the definitions of $[\mathcal{U}]_{B_{1,\infty}^\beta}$ for $\beta\in(1,2)$ and $[\lolly]_\alpha$ we have
\begin{align}
&\int_{D_\y} d\x\, \big|u(\x+\y)-u(\x)\big| \\
&\leq \int_{D_\y} d\x\, \big|u(\x+\y)-u(\x)-\sigma(u(\x))\Pi_\x[\lolly;\x](\x+\y)-\nu(\x)\cdot y\big| \\
&\,+ \int_{D_\y} d\x\, \big|\sigma(u(\x))\Pi_\x[\lolly;\x](\x+\y)-\nu(\x)\cdot y\big| \\
&\leq \|\y\|_\s^\beta [\mathcal{U}]_{B_{1,\infty}^\beta} 
+ \|\y\|_\s^\alpha \|\sigma/a^{\mathfrak{e}(\lollysmall)}\|_\infty T [\lolly]_\alpha
+ \|\nu\|_{L^1}\|\y\|_\s \, .
\end{align}
Choosing $\beta$ as in Theorem~\ref{thm:modelledness} and appealing to 
%
\begin{equation}
\Big|\frac{\sigma}{a^{\mathfrak{e}(\lollysmall)}}\Big| \lesssim 1\, ,
\tag{XVI}\label{eq:XVI}
\end{equation}
it remains to estimate $\|\nu\|_{L^1}$. 
For this we rewrite 
\begin{align}
\nu(\x)\cdot y
&=u(\x+(0,y))-u(\x)-\sigma(u(\x))\Pi_\x[\lolly;\x](\x+(0,y)) \\
&\,-\big(u(\x+(0,y))-u(\x)-\sigma(u(\x))\Pi_\x[\lolly;\x](\x+(0,y))-\nu(\x)\cdot y\big) \, ,
\end{align}
which by the triangle inequality and the definitions of $[\lolly]_\alpha$ and $[\mathcal{U}]_{B_{1,\infty}^{\beta}}$ yields
\begin{equation}
\int_{(0,T)\times\T^d} \! d\x \, | \nu(\x)\cdot y| 
\leq 2 T \! \sup_{0<t\leq T}\int_{\T^d}dx\,|u(t,x)| 
+ \|\sigma\!/\!a^{\mathfrak{e}(\lollysmall)}\|_\infty T [\lolly]_\alpha |y|^\alpha 
+ [\mathcal{U}]_{B_{1,\infty}^{\beta}} |y|^{\beta} \, .
\end{equation}
Choosing $y=T e_i$ for $e_i\in\R^d$ the $i$-th unit vector with $|y|=T\leq R$ yields 
\begin{equation}
\int_{(0,T)\times\T^d}d\x\,|\nu(\x)|
\lesssimdata \sup_{0<t\leq T}\int_{\T^d}dx\,|u(t,x)| + [\mathcal{U}]_{B_{1,\infty}^\beta} +1 \, .
\end{equation}
Actually, $\|u(t,\cdot)\|_{L^1(\T^d)}$ is estimated in the proof of Proposition~\ref{prop:energy}
by $C_2$ from Proposition~\ref{prop:u2}, which in turn is estimated in the proof of Proposition~\ref{prop:u>} by $1+[\mathcal{U}]_{B_{1,\infty}^\beta}$ for $\beta\in(2-\alpha,2\alpha)$. 
Thus for such $\beta$ it holds $\|\nu\|_{L^1}\lesssim 1+[\mathcal{U}]_{B_{1,\infty}^\beta}$, and we conclude by appealing once more to Theorem~\ref{thm:modelledness}. 
\end{proof}


\appendix

\section{Auxiliary results}

\subsection{Equivalence of classical and kinetic formulation}\label{app:kinetic}

In this Subsection we argue that sufficiently smooth solutions $u$ of \eqref{eq:spm} satisfy its kinetic formulation \eqref{eq:kinetic} and vice versa. 

To see this, let $u$ be a classical solution of \eqref{eq:spm} and $\phi\in C_c^\infty((0,T)\times\T^d\times\R)$.
Then by the definition of the kinetic function $\chi$, the left hand side of \eqref{eq:kinetic} coincides with 
\begin{equation}
-\int_{[0,T]\times\T^d}dt\,dx\int_0^{u(t,x)}dv\, (\partial_t+a(v)\Delta)\phi(t,x,v) \, ,
\end{equation}
which we can rewrite as 
\begin{align}
&-\int_{[0,T]\times\T^d}dt\,dx \Big(
\partial_t \int_0^{u(t,x)}dv\, \phi(t,x,v) - \phi(t,x,u(t,x))\partial_t u(t,x) \\
&\qquad+ \nabla\cdot \int_0^{u(t,x)} dv\, a(v)\nabla\phi(t,x,v) - a(u(t,x))\nabla\phi(t,x,u(t,x))\nabla u(t,x) \Big) \, . 
\end{align}
The first summand vanishes since $\phi(0,x,v)=\phi(T,x,v)=0$,
and the third summand vanishes by periodicity. 
We are thus left with 
\begin{equation}
\int_{[0,T]\times\T^d}dt\,dx \Big(
\phi(t,x,u(t,x))\partial_t u(t,x) 
+ a(u(t,x))\nabla\phi(t,x,u(t,x))\nabla u(t,x) \Big) \, .
\end{equation}
We use the chain rule to rewrite $\nabla\phi(t,x,u(t,x)) = \nabla (\phi(\cdot_t,\cdot_x,u(\cdot_t,\cdot_x)))(t,x) - \partial_v \phi(t,x,u(t,x)) \nabla u(t,x)$, 
integrate by parts, and plug in \eqref{eq:spm} to arrive at the right hand side of \eqref{eq:kinetic}. 

The same computation read in the opposite direction shows that if $u$ is smooth enough and such that $\chi$ satisfies \eqref{eq:kinetic} for all test functions $\phi$, then $u$ satisfies \eqref{eq:spm}. 

\subsection{Heuristics for the counterterm}\label{app:heuristics_counterterm}

Let us briefly motivate the choice of counterterm in Section~\ref{sec:mild}.
If a function $u$ is modelled according to 
\begin{equation}
u = u(\x) + \sigma(u(\x)) \Pi_\x[\lolly;\x] + \nu(\x) \cdot (\cdot-x) + \mathcal{O}(|\cdot-\x|^{2\alpha}) \, ,
\end{equation}
then its composition with a sufficiently smooth function $f$ is modelled according to 
\begin{equation}
f(u) = f(u(\x)) + f'(u(\x)) \Big( \sigma(u(\x)) \Pi_\x[\lolly;\x] + \nu(\x) \cdot (\cdot-x) \Big) + \mathcal{O}(|\cdot-\x|^{2\alpha}) \, ,
\end{equation}
as can be seen by Taylor. 
Multiplying the right hand side with $\xi$ 
and inserting the necessary renormalization for it to be stable 
yields
\begin{align}
f(u(\x)) \xi 
+ f'(u(\x)) \Big( \sigma(u(\x)) \Pi_\x[\dumb;\x] + \nu(\x) \cdot \Pi_\x[\xnoise] \Big) 
+ \mathcal{O}(|\cdot-\x|^{3\alpha-2}) \, .
\end{align}
The correct choice of counterterm can now be read off the reconstruction of this expression, 
which informally is its diagonal evaluation (i.e.~setting $\x$ to coincide with the active variable). 
Since $\Pi_\x[\dumb;\x](\x) = -\cdumb^{a(u(\x))}(t)$ and $\Pi_\x[\xnoise](\x)=0$ this reconstruction is given by 
\begin{equation}
f(u) \xi - f'(u) \sigma(u) \cdumb^{a(u)} \, ,
\end{equation}
which coincides with the choice in \eqref{eq:ren_intro} for $f(u)=\varphi^>(a(u)/\delta)\Psi(a(u))\sigma(u)$. 

\subsection{Test functions}

\begin{lemma}\label{lem:estK}
The kernel 
$G(t,x)\coloneqq \sum_{n\in\Z^d}\Phi(t,x+n) \zeta(x)$ 
with $\Phi$ the heat kernel defined in \eqref{eq:heatkernel} and $\zeta$ from \eqref{eq:defK} 
satisfies for $0\neq\x=(t,x)\in\R\times\R^d$ and $l\in\N$, $k\in\N^d$
\begin{align}
\big| \partial_t^l\partial_x^k G(\x) \big| 
\lesssim 1+ \|\x\|_\s^{-d-2l-|k|_1} \, .
\end{align}
The implicit constant depends only on the dimension $d$ 
and the derivatives (up to order $k$) 
of $\zeta$.
\end{lemma}
\begin{proof}
As a first step we claim that the heat kernel $\Phi$ defined in \eqref{eq:heatkernel} satisfies 
\begin{equation}
\big| \partial_t^l\partial_x^k \Phi(\x) \big| \lesssim \|\x\|_\s^{-d-2l-|k|_1} \textnormal{ for } 0\neq\x\in\R\times\R^d \, .
\end{equation}
Indeed, note that $\Phi$ satisfies the scaling invariance
$\Phi(\S^{1/\lambda}\x) = \lambda^d \Phi(\x)$. 
By choosing $\lambda\coloneqq\|\x\|_\s$ and since $\Phi(\S^{1/\|\x\|_\s}\x)$ is 
a smooth function on a compact set we obtain the desired estimate for $\Phi$ 
and all $\x\in\R\times\R^d$, $\x\neq0$.

We turn to $G$, where by the binomial formula 
\begin{equation}
\partial_t^l\partial_x^k G(\x)
= \sum_{k_1+k_2=k}\sum_{n\in\Z^d} \tbinom{k_1+k_2}{k_1} 
\partial_t^{l}\partial_x^{k_1} \Phi(t,x+n) 
\partial_x^{k_2} \zeta(x) \, .
\end{equation}
By the triangle inequality this can be split into\footnote{$k_1\leq k$ for $k_1,k\in\N^d$ has to be understood component wise}
\begin{align}
|\partial_t^l\partial_x^k G(\x)|
&\lesssim 
\sum_{k_1\leq k} \sum_{n\in\{-1,0,1\}^d}
| \partial_t^{l}\partial_x^{k_1} \Phi(t,x+n) | \\
&\,+ 
\sum_{k_1\leq k} \sum_{n\in\Z^d\setminus\{-1,0,1\}^d} 
| \partial_t^{l}\partial_x^{k_1} \Phi(t,x+n) | \, .
\end{align}
The first right hand side term is by the just established estimate of $\Phi$ estimated by 
\begin{equation}
\sum_{k_1\leq k} \sum_{n\in\{-1,0,1\}^d} 
\|(t,x+n)\|_\s^{-d-2l-|k_1|_1} 
\lesssim 
\sum_{k_1\leq k} 
\|(t,x)\|_\s^{-d-2l-|k_1|_1} \, ,
\end{equation}
where we used that we only need to consider $x$ in a set slightly larger than $[-1/2,1/2]^d$ by the support of $\zeta$.
To estimate the second right hand side term we distinguish the case $2l+|k_1|_1=0$ from $2l+|k_1|_1\geq1$. 
The contribution of the former case to the second right hand side term equals 
$\sum_{n\in\Z^d\setminus\{-1,0,1\}^d} \Phi(t,x+n)$ which is estimated by 
$\int_{\R^d} dx \, \Phi(t,x) = 1$,
where we used again that $x$ is close to $[-1/2,1/2]^d$; 
the contributions of the latter case can be estimated by using the already established estimate of $\Phi$ by 
\begin{equation}
\sum_{n\in\Z^d\setminus\{-1,0,1\}^d} 
\|(t,x+n)\|_\s^{-d-2l-|k_1|_1} 
\lesssim 
\sum_{n\in\Z^d\setminus\{-1,0,1\}^d} 
|x+n|^{-d-2l-|k_1|_1} 
\lesssim 1 \, ,
\end{equation}
where in the last inequality we used once more that $x$ is close to 
$[-1/2,1/2]^d$.

Summarizing, we obtained that 
$|\partial_t^l\partial_x^k G(\x)|\lesssim 1+
\sum_{k_1\leq k} \|\x\|_\s^{-d-2l-|k_1|_1}$.
We conclude by noting that 
if $\|\x\|_\s<1$ then $\sum_{k_1\leq k} \|\x\|_\s^{-d-2l-|k_1|_1}\lesssim \|\x\|_\s^{-d-2l-|k|_1}$, 
whereas if $\|\x\|_\s\geq1$ then $\sum_{k_1\leq k} \|\x\|_\s^{-d-2l-|k_1|_1}\lesssim1$.
\end{proof}

\begin{lemma}\label{lem:multtest}
There exists a constant $C>0$ such that the following holds. 
\begin{itemize}
\item[(i)] For $\varphi,\psi\in\mathcal{B}$, $\x,\y\in\R\times\R^d$, and $0<\lambda\leq r$ it holds
\begin{equation}
\varphi_\x^\lambda \, \psi_\y^r = r^{-D} \rho_\x^\lambda
\quad\textnormal{for some } \rho\in C\mathcal{B} \, . 
\end{equation}
\item[(ii)] For $\x,\z\in\R\times\R^d$, $l,m,n\in\N$, $k\in\N^d$, $q\in\Z$, and $\bar{a}>0$, 
the kernel $K^{(n)}=K_q^{(n)}$ defined in \eqref{eq:defK} and \eqref{eq:splittingk} satisfies 
\begin{equation}
\partial_{\bar{a}}^m\partial_t^l\partial_x^k \Kn(\bar{a},\x-\z) 
= 2^{-n(2-2l-|k|_1)} \bar{a}^{l-m} \, \rho_{(\bar{a}t,x)}^{2^{-n}}(\bar{a}r,z)
\, \tilde\varphi(2^q\bar{a})
\end{equation}
for some $\rho\in C\mathcal{B}$ supported on non-positive times and 
for some smooth $\tilde\varphi\colon\R\to\R$ supported in $(1/2,2)$.\footnote{Actually, for $m>0$ it is a linear combination of such $\rho$ and $\tilde\varphi$, which we suppress for notational convenience.} 
\end{itemize}
\end{lemma}
\begin{proof}
For \textit{(i)} define $\rho(\z) \coloneqq \varphi(\z) \, \psi(\S^{1/r}(\S^\lambda\z+\x-\y))$.
Then 
\begin{align}
\rho_\x^\lambda(\z) &= \lambda^{-D} \rho(\S^{1/\lambda}(\z-\x)) \\
&= \lambda^{-D} \varphi(\S^{1/\lambda}(\z-\x)) \, \psi(\S^{1/r}(\z-\x+\x-\y)) \\
&= \varphi_\x^\lambda(\z) \, r^D \psi_\y^r(\z) \, .
\end{align}
Furthermore, $\rho$ is smooth and its support is contained in that of $\varphi$, and
\begin{align}
\big|\partial_t^l\partial_x^k \rho(\z)\big| 
&= \Big|\sum_{\substack{l_1+l_2=l \\k_1+k_2=k}} \tbinom{l}{l_1} \tbinom{k}{k_1} 
\partial_t^{l_1}\partial_x^{k_1}\varphi(\z) \, \partial_t^{l_2}\partial_x^{k_2} \psi(\S^{1/r}(\S^\lambda\z+\x-\y)) (\tfrac{\lambda}{r})^{2l_2+|k_2|_1} \Big| \\
&\leq \sum_{\substack{l_1+l_2=l \\k_1+k_2=k}} \tbinom{l}{l_1} \tbinom{k}{k_1} 
=2^{l+|k|_1} \, ,
\end{align}
hence $\rho\in C\mathcal{B}$ for a sufficiently large $C$. 

For \textit{(ii)} recall the kernel $G$ from Lemma~\ref{lem:estK} and $\eta$ from \eqref{eq:splittingk}, 
and define 
$\rho(\z)\coloneqq 2^{-nd} \, \eta(-\z)G(-\S^{2^{-n}}\z)$. 
Then for $\varphi$ from \eqref{eq:decompose_phi>} it holds
\begin{align}
\rho_{(\bar{a}t,x)}^{2^{-n}}(\bar{a}r,z) \varphi(2^q\bar{a})
&= 2^{nD} \rho(\S^{2^n}\Da (\z-\x)) \varphi(2^q\bar{a}) \\
&= 2^{2n} \eta(\S^{2^n}\Da (\x-\z)) G(\Da(\x-\z)) \varphi(2^q\bar{a}) \\
&= 2^{2n} \Kn(\bar{a},\x-\z) \, ,
\end{align}
which proves the claim for $m=l=0$ and $k=0$ with $\tilde\varphi=\varphi$, 
provided $\rho\in C\mathcal{B}$ which we establish next. 
The support of $\rho$ is contained in 
$(-\,\mathrm{supp}\,\eta) \cap ((-\infty,0)\times\R^d)$, it is smooth, and 
\begin{align}
\big|\partial_t^l\partial_x^k \rho(-\z)\big| 
&= 2^{-nd}\Big| \sum_{\substack{l_1+l_2=l \\k_1+k_2=k}} \tbinom{l}{l_1}\tbinom{k}{k_1} 
\partial_t^{l_1}\partial_x^{k_1} \eta(\z) \, \partial_t^{l_2}\partial_x^{k_2} G(\S^{2^{-n}}\z) 2^{-n(2l_2+|k_2|_1)} \Big| \\
&\lesssim 2^{-nd} \!\!\! \sum_{\substack{l_1+l_2=l \\k_1+k_2=k}} \! \tbinom{l}{l_1}\tbinom{k}{k_1} 
\1_{\mathrm{supp}\,\eta}(\z)
(1\!+\!\|\S^{2^{-n}}\!\z\|_\s^{-d-2l_2-|k_2|_1}) 2^{-n(2l_2+|k_2|_1)} \\
&\leq \sum_{\substack{l_1+l_2=l \\k_1+k_2=k}} \tbinom{l}{l_1}\tbinom{k}{k_1} 
\1_{\mathrm{supp}\,\eta}(\z)
(1+\|\z\|_\s^{-d-2l_2-|k_2|_1}) \, , 
\end{align}
where in the first inequality we have used Lemma~\ref{lem:estK} 
and in the second inequality we used $n\in\N$.
Since $\|\z\|_\s \sim1$ on $\mathrm{supp}\,\eta$
we obtain $|\partial_t^l\partial_x^k\rho(\z)|\lesssim 1$ and thus 
$\rho\in C\mathcal{B}$ for a sufficiently large $C$. 
The claim for general $m,l,k$ is a consequence of the already established: 
applying the product and chain rule to the claim without derivatives yields
\begin{align}
&\partial_{\bar{a}}^m\partial_t^l\partial_x^k \Kn(\bar{a},\x-\z) 
= 2^{-2n} \sum_{m_1+m_2+m_3=m} \tfrac{m!}{m_1!m_2!m_3!} \\
&\, \times \tfrac{l!}{(l-m_1)!} \bar{a}^{l-m_1} 
(\partial_t^{l+m_2} \partial_x^k \rho)_{(\bar{a}t,x)}^{2^{-n}}(\bar{a}r,z) 2^{n(2l+2m_2+|k|_1)} (r-t)^{m_2} 
(\partial_{\bar{a}}^{m_3}\tilde\varphi)(2^q\bar{a}) 2^{qm_3} \, .
\end{align}
We conclude by noting that 
$(2^{2n}(r-t))^{m_2}\sim\bar{a}^{-m_2}$ which is a consequence 
of the support of $\rho$ being contained in the support of $\eta$, 
and by noting that $2^{qm_3}\sim\bar{a}^{-m_3}$ which is a consequence 
of the support of $\tilde\varphi$.
\end{proof}

\subsection{Reconstruction in an \texorpdfstring{$L^1$}{L1} setting}

We recall here a version of the reconstruction theorem in an $L^1$-setting which we used frequently. 
\begin{theorem}[Reconstruction]\label{thm:reconstruction}
Let $F_\x(\z;\y)$ for 
$\x\in[0,T]\times\R^d$, $\y\in[0,T]\times\R^d$, and $\z\in[0,\infty)\times\R^d$
be jointly $1$-periodic in its spatial arguments, 
i.e.~$F_{\x}(\z;\y) = F_{\x+(0,k)}(\z+(0,k);\y+(0,k))$ for all $k\in\Z^d$. 
Assume that there exist 
a smooth $\varphi$ compactly supported in the past with $\int\varphi\neq0$ and 
$\theta_1,\theta_2\in\R$ with $\theta_1+\theta_2>0$, $\theta_2\geq0$
such that 
\begin{align}
&\int_{D_{\y'}\cap D_{\y''}} d\x \Big| \int_{(0,\infty)\times\R^d} d\z \, 
\big(F_{\x+\y'}(\z;\x+\y'') - F_{\x}(\z;\x+\y'') \big) \varphi_\x^\lambda (\z) \Big| \\
&\leq C \lambda^{\theta_1} (\lambda+\|\y'\|_\s)^{\theta_2} \, ,
\end{align}
for $\y'\in\R\times\R^d$, $\y''\in[0,\infty)\times\R^d$, 
and dyadic\footnote{i.e.~of the form $2^{-n}$ for $n\in\Z$} $\lambda\leq\Lambda$ 
for some dyadic $\Lambda$. 

Then 
\begin{align}
&\int_{D_\y} d\x \Big| \int_{(0,\infty)\times\R^d} d\z \, 
\big( F_{\z}(\z;\x+\S^\vartheta\y) - F_{\x}(\z;\x+\S^\vartheta\y) \big) 
\psi^\Lambda_{\x+\S^\vartheta\y-\bar\x}(\z) \Big| \\
&\lesssim C \Lambda^{\theta_1} (\Lambda+\|\S^\vartheta\y-\bar\x\|_\s)^{\theta_2} 
\end{align}
for $\y\in\R\times\R^d$, $\bar\x\in[0,\infty)\times\R^d$, $\vartheta\in[0,1]$, 
and smooth $\psi$ compactly supported in the past, 
where the implicit constant depends only on $d$ 
and $\|\psi\|_{C^r}$ for the smallest $r\in\N$ with $\theta_1+r>0$. 
The assumption is only used for $\|\y'\|_\s<\vartheta\|\y\|_\s+\|\bar\x\|_\s+3\Lambda$ 
and $\|\y''\|_\s<\|\bar\x\|_\s+3\Lambda$. 
\end{theorem}

\begin{proof}
Let us first mention that the definition of $D_\y$ (see Theorem~\ref{thm:main}) ensures that all arguments of 
$F$ in the assumption are contained in the domain of $F$. 
To check that all arguments of $F$ in the conclusion are contained in the domain of $F$ we argue as follows: 
first, $\x\in D_\y$ implies that the time component of $\x+\S^\vartheta\y$ is contained in $[0,T]$ for all $\vartheta\in[0,1]$; 
furthermore, $F_\cdot(\z;\x+\S^\vartheta\y)$ vanishes by assumption 
unless $\z$ has a non-negative time component, 
and $\Psi_{\x+\S^\vartheta\y-\bar\x}(\z)$ vanishes by assumption unless $\z-\x-\S^\vartheta\y+\bar\x$ has a negative time component which by $\x\in D_\y$ and $\bar\x$ having a non-negative time component implies that 
also the time component of $\z$ is contained in $[0,T]$. 

We now turn to the reconstruction proof proper. 
A concise proof of the reconstruction theorem in a Besov setting can be found in \cite{BL23}. 
We are not precisely in their setup 
since we deal with time components in $[0,T]$, 
have the additional dependence on $\x+\S^\vartheta\y$ 
(which is not merely an additional parameter as we integrate over $\x$), 
and work with kernels only looking into the past. 
We thus briefly mention how to adapt the proof of \cite[Proposition~4.10]{BL23}, 
which is the main step in the proof of the reconstruction bound we want to obtain. 

In a first step\footnote{recall that we write $F(\varphi)$ short for 
$\int_{(0,\infty)\times\R^d}d\z\,F(\z)\varphi(\z)$} 
$(F_\cdot(\,\cdot\,;\x+\S^\vartheta\y)-F_\x(\,\cdot\,;\x+\S^\vartheta\y))(\psi_{\x+\S^\vartheta\y-\bar\x}^\Lambda)$ is rewritten
as a sum of four terms $a,b,c,d$, 
the first one of which is given by 
\begin{equation}
a(\x) \coloneqq
\int \!d\z' \int \!d\x' \big( F_{\z'}(\,\cdot\,;\x+\S^\vartheta\y)-F_{\x'}(\,\cdot\,;\x+\S^\vartheta\y) \big) 
(\hat{\varphi}_{\x'}^\Lambda) 
\hat{\varphi}_{\z'}^{2\Lambda}(\x') \psi_{\x+\S^\vartheta\y-\bar\x}^\Lambda(\z') \, .
\end{equation}
We neglect the remaining three terms for now and get back to them below. 
Here, $\hat\varphi\in\mathcal{B}$ is defined in \cite[Proposition~5.4]{BL23} as a linear combination of rescaled versions of $\varphi$ and is thus supported in the past. 
The integral over $\z'$ restricts effectively to the support 
$B_\s^-(\x+\S^\vartheta\y-\bar\x, \Lambda)$ of $\psi_{\x+\S^\vartheta\y-\bar\x}^\Lambda$, 
and the integral over $\x'$ restricts effectively to the support 
$B_\s^-(\z', 2\Lambda)$ of $\varphi_{\z'}^{2\Lambda}$, 
where $B_\s^-(\x,R)$ denotes the parabolic ball with radius $R$ centered at $\x$ looking only into the past. 
Using the triangle inequality and bounding the suprema of $\hat\varphi$ and $\psi$ by $1$ we obtain 
\begin{equation}
\begin{split}
|a(\x)| \lesssim 
\Lambda^{-2D} \int_{B_\s^-(\x+\S^\vartheta\y-\bar\x,\Lambda)} d\z' 
&\int_{B_\s^-(\z',2\Lambda)} d\x' \, \\
&\big| \big( F_{\z'}(\,\cdot\,;\x+\S^\vartheta\y) 
- F_{\x'}(\,\cdot\,;\x+\S^\vartheta\y)\big) (\hat\varphi_{\x'}^\Lambda) \big| \, .
\end{split}
\end{equation}
Changing variables according to $\tilde\x=-\x'+\z'$ and $\tilde\z=\z'-(\x+\S^\vartheta\y-\bar\x)$ yields
\begin{equation}
\begin{split}
|a(\x)| &\lesssim 
\Lambda^{-2D} \int_{B_\s^-(0,\Lambda)} d\tilde\z \int_{-B_\s^-(0,2\Lambda)} d\tilde\x \\
&\ \big| \big( F_{\tilde\z+\x+\S^\vartheta\y-\bar\x}(\cdot;\x\!+\!\S^\vartheta\y) - F_{\tilde\z+\x+\S^\vartheta\y-\bar\x-\tilde\x}(\cdot;\x\!+\!\S^\vartheta\y)\big) (\hat\varphi_{\tilde\z+\x+\S^\vartheta\y-\bar\x-\tilde\x}^\Lambda) \big| .
\end{split}
\end{equation}
By Minkowski's inequality 
\begin{align}
\int_{D_\y} d\x \, |a(\x)| 
&\lesssim \Lambda^{-2D} 
\int_{B_\s^-(0,\Lambda)} d\tilde\z 
\int_{-B_\s^-(0,2\Lambda)} d\tilde\x 
\int_{D_\y+\tilde\z+\S^\vartheta\y-\bar\x-\tilde\x} d\tilde{\mathbf{w}} \\
&\quad\big| \big( F_{\tilde{\mathbf{w}}+\tilde\x}(\,\cdot\,;\tilde{\mathbf{w}}-\tilde\z+\bar\x+\tilde\x) 
- F_{\tilde{\mathbf{w}}}(\,\cdot\,;\tilde{\mathbf{w}}-\tilde\z+\bar\x+\tilde\x) \big) 
(\hat\varphi_{\tilde{\mathbf{w}}}^\Lambda) \big| \, ,
\end{align}
where we changed variables according to $\tilde{\mathbf{w}}=\tilde\z+\x+\S^\vartheta\y-\bar\x-\tilde\x$.

To apply the assumption 
(with $\y'=\tilde\x$ and $\y''=-\tilde\z+\bar\x+\tilde\x$) 
we verify that all arguments of $F$ belong to its domain. 
First, note that the time component of 
$\tilde{\mathbf{w}}-\tilde\z+\bar\x+\tilde\x\in D_\y+\S^\vartheta\y$ 
is contained in $[0,T]$ by definition of $D_\y$. 
We now argue that in the above integral effectively $\tilde{\mathbf{w}}\in D_{\tilde\x}$ 
and first focus on the time variable. 
The implicit $(\cdot)$-variable of $F$ has effectively a non-negative time component by assumption, 
and since $\hat\varphi$ is supported in the past 
this transfers to $\tilde{\mathbf{w}}$ having effectively a non-negative time component. 
Furthermore, $\tilde\z\in B_\s^-(0,\Lambda)$ has a negative time component, 
$-\bar\x$ has by assumption a non-positive time component, 
and $-\tilde\x\in B_\s^-(0,2\Lambda)$ has a negative time component as well; 
hence $\tilde{\mathbf{w}}\in D_\y+\tilde\z+\S^\vartheta\y-\bar\x-\tilde\x$ has a time component $\leq T$. 
Once more since $\tilde\x$ has a positive time component 
also $\tilde{\mathbf{w}}+\tilde\x$ has a positive time component, 
and by $\tilde{\mathbf{w}}+\tilde\x\in D_\y+\tilde\z+\S^\vartheta\y-\bar\x$ 
and $\tilde\z-\bar\x$ having a negative time component 
we deduce that $\tilde{\mathbf{w}}+\tilde\x$ has a time component $\leq T$.
Altogether the time components of $\tilde{\mathbf{w}}$ and $\tilde{\mathbf{w}}+\tilde\x$ 
are contained in $[0,T]$. 
We turn to the spatial variables, 
where we remark that the function 
$\tilde{\mathbf{w}}
\mapsto 
F_{\tilde{\mathbf{w}}+\tilde\x}(\,\cdot\,;\tilde{\mathbf{w}}+\tilde\y) 
(\varphi_{\tilde{\mathbf{w}}})$ 
is periodic in space by assumption (for all $\tilde\x$, $\tilde\y$, and $\varphi$). 
Thus shifting the torus $\T^d$ by the spatial components of $\tilde\z+\S^\vartheta\y-\bar\x-\tilde\x$ 
has no effect. 
Altogether this yields the desired $\tilde{\mathbf{w}}\in D_{\tilde\x}$. 

The assumption thus applies and yields the estimate
\begin{equation}
\int_{D_\y} d\x \, |a(\x)| 
\lesssim \Lambda^{-2D} 
\int_{B_\s^-(0,\Lambda)} d\tilde\z 
\int_{-B_\s^-(0,2\Lambda)} d\tilde\x \, 
C \Lambda^{\theta_1} (\Lambda+\|\tilde\x\|_\s)^{\theta_2}
\lesssim C \Lambda^{\theta_1+\theta_2} \, .
\end{equation}

\medskip

We turn to the remaining three terms $b,c,d$. 
The first of them is given by 
\begin{equation}
b(\x)\coloneqq
\int \! d\z' \int\! d\x' \, \big(F_{\x'}(\,\cdot\,;\x+\S^\vartheta\y) - F_{\x}(\,\cdot\,;\x+\S^\vartheta\y)\big) 
(\hat\varphi_{\x'}^\Lambda) \hat\varphi_{\z'}^{2\Lambda}(\x') \psi_{\x+\S^\vartheta\y-\bar\x}^\Lambda(\z') \, .
\end{equation}
The two integrals effectively restrict to the same sets as for $a$. 
As for $a$ we bound the suprema of $\hat\varphi$ and $\psi$ by $1$ and note that 
$B_\s^-(\z',2\Lambda)\subseteq B_\s^-(\x+\S^\vartheta\y-\bar\x,3\Lambda)$ 
for $\z'\in B_\s^-(\x+\S^\vartheta\y-\bar\x,\Lambda)$, so that we obtain 
\begin{equation}
|b(\x)| \lesssim
\Lambda^{-D} \int_{B_\s^-(\x+\S^\vartheta\y-\bar\x,3\Lambda)} d\x' \, 
\big| \big( F_{\x'}(\,\cdot\,;\x+\S^\vartheta\y) - F_{\x}(\,\cdot\,;\x+\S^\vartheta\y) \big) (\hat\varphi_{\x'}^\Lambda) \big| \, .
\end{equation}
Changing variables according to $\tilde\x=-\x'+\x+\S^\vartheta\y-\bar\x$ yields
\begin{equation}
\begin{split}
|b(\x)| &\lesssim
\Lambda^{-D} \int_{-B_\s^-(0,3\Lambda)} d\tilde\x \\
&\quad
\big| \big( F_{\x+\S^\vartheta\y-\bar\x-\tilde\x}(\,\cdot\,;\x+\S^\vartheta\y) - F_{\x}(\,\cdot\,;\x+\S^\vartheta\y) \big) (\hat\varphi_{\x+\S^\vartheta\y-\bar\x-\tilde\x}^\Lambda) \big| \, .
\end{split}
\end{equation}
By Minkowski's inequality 
\begin{align}
\int_{D_\y} d\x \, |b(\x)| 
&\lesssim \Lambda^{-D} \int_{-B_\s^-(0,3\Lambda)} d\tilde\x
\int_{D_\y+\S^\vartheta\y-\bar\x-\tilde\x} d\tilde{\mathbf{w}} \\
&\quad\big| \big( F_{\tilde{\mathbf{w}}}(\,\cdot\,;\tilde{\mathbf{w}}+\bar\x+\tilde\x) - F_{\tilde{\mathbf{w}}-\S^\vartheta\y+\bar\x+\tilde\x}(\,\cdot\,;\tilde{\mathbf{w}}+\bar\x+\tilde\x) \big) (\hat\varphi_{\tilde{\mathbf{w}}}^\Lambda) \big| \, ,
\end{align}
where we changed variables according to $\tilde{\mathbf{w}}=\x+\S^\vartheta\y-\bar\x-\tilde\x$. 
The analogous arguments as for $a$ show that the assumption 
(with $\y'=-\S^\vartheta\y+\bar\x+\tilde\x$ and $\y''=\bar\x+\tilde\x$)
can be applied which yields
\begin{align}
\int_{D_\y} d\x \, |b(\x)| 
&\lesssim
\Lambda^{-D} \int_{-B_\s^-(0,3\Lambda)} d\tilde\x \, 
C\Lambda^{\theta_1} (\Lambda+\|\S^\vartheta\y-\bar\x-\tilde\x\|_\s)^{\theta_2} \\
&\lesssim C \Lambda^{\theta_1} (\Lambda+\|\S^\vartheta\y-\bar\x\|_\s)^{\theta_2} \, .
\end{align}

We turn to $c$ which is given by 
\begin{equation}
c(\x)\coloneqq 
\!\sum_{\lambda\leq\Lambda}\! \int\! d\z'\!\! \int \!d\x' \big( F_{\z'}(\cdot;\x+\S^\vartheta\y) - F_{\x'}(\cdot;\x+\S^\vartheta\y)\big)
(\hat\varphi_{\x'}^{\lambda}) \check\varphi_{\z'}^\lambda(\x') \psi_{\x+\S^\vartheta\y-\bar\x}^\Lambda(\z') ,
\end{equation}
where the sum is taken over dyadic $\lambda$, and 
$\check\varphi\in\mathcal{B}$ is defined in \cite[Section~4.3]{BL23} as a linear combination of rescaled versions of $\hat\varphi$ and is thus also supported in the past. 
The first integral restricts as for $a,b$ to $\z'\in B_\s^-(\x+\S^\vartheta\y-\bar\x,\Lambda)$, 
whereas the second integral restricts to $\x'\in B_\s^-(\z',\lambda)$. 
Bounding the suprema of $\check\varphi$ and $\psi$ by $1$ yields
\begin{equation}
\begin{split}
|c(\x)| \lesssim 
\Lambda^{-D} &\sum_{\lambda\leq\Lambda} \lambda^{-D} 
\int_{B_\s^-(\x+\S^\vartheta\y-\bar\x,\Lambda)}d\z' \int_{B_\s^-(\z',\lambda)} d\x' \\
&\quad
\big| \big( F_{\z'}(\,\cdot\,;\x+\S^\vartheta\y) - F_{\x'}(\,\cdot\,;\x+\S^\vartheta\y) \big)
(\hat\varphi_{\x'}^\lambda) \big| \, .
\end{split}
\end{equation}
Changing variables as for $a$ yields
\begin{align}
\int_{D_\y} d\x \, |c(\x)| 
&\lesssim \Lambda^{-D} \sum_{\lambda\leq\Lambda} \lambda^{-D} 
\int_{B_\s^-(0,\Lambda)} d\tilde\z \int_{-B_\s^-(0,\lambda)} d\tilde\x
\int_{D_\y+\tilde\z+\S^\vartheta\y-\bar\x-\tilde\x} d\tilde{\mathbf{w}} \\
&\quad\big| \big( F_{\tilde{\mathbf{w}}+\tilde\x}(\,\cdot\,;\tilde{\mathbf{w}}-\tilde\z+\bar\x+\tilde\x) - 
F_{\tilde{\mathbf{w}}}(\,\cdot\,;\tilde{\mathbf{w}}-\tilde\z+\bar\x+\tilde\x) \big)
(\hat\varphi_{\tilde{\mathbf{w}}}^\lambda) \big| \, .
\end{align}
Again the assumption 
(with $\y'=\tilde\x$ and $\y''=-\tilde\z+\bar\x+\tilde\x$)
can be applied to the effect of 
\begin{align}
\int_{D_\y} d\x \, |c(\x)| 
&\lesssim \Lambda^{-D} \sum_{\lambda\leq\Lambda} \lambda^{-D} 
\int_{B_\s^-(0,\Lambda)} d\tilde\z \int_{-B_\s^-(0,\lambda)} d\tilde\x \,
C \lambda^{\theta_1} (\lambda+\|\tilde\x\|_\s)^{\theta_2} \\
&\lesssim C \sum_{\lambda\leq\Lambda} \lambda^{\theta_1+\theta_2} 
\lesssim C \Lambda^{\theta_1+\theta_2} \, 
\end{align}
where in the last inequality we have used that by assumption $\theta_1+\theta_2>0$. 

Finally we turn to $d$, which is given by 
\begin{equation}
d(\x) \coloneqq
\!\sum_{\lambda\leq\Lambda}\! \int\! d\z'\!\! \int\! d\x' 
\big( F_{\x'}(\cdot;\x+\S^\vartheta\y) - F_\x(\cdot;\x+\S^\vartheta\y) \big)(\hat\varphi_{\x'}^\lambda)
\check\varphi_{\z'}^\lambda(\x') \psi_{\x+\S^\vartheta\y-\bar\x}^\Lambda(\z') \, .
\end{equation}
The two integrals effectively restrict to the same sets as for $c$. 
Changing variables according to 
$\tilde\z=\z'-\x-\S^\vartheta\y+\bar\x$ and then $\tilde\x=-\x'+\x+\S^\vartheta\y-\bar\x$, yields 
\begin{align}
d(\x) &= 
\sum_{\lambda\leq\Lambda} \int_{B_\s^-(0,\Lambda)} d\tilde\z 
\int_{-B_\s^-(\tilde\z,\lambda)} d\tilde\x \\
&\quad\big(F_{\x\!+\!\S^\vartheta\y-\bar\x-\tilde\x}(\,\cdot\,;\x\!+\!\S^\vartheta\y) - F_{\x}(\,\cdot\,;\x+\S^\vartheta\y) \big)(\hat\varphi_{\x+\S^\vartheta\y-\bar\x-\tilde\x}^\lambda) 
\check\varphi_{\tilde\z}^\lambda(-\tilde\x) \psi^\Lambda(\tilde\z) \, .
\end{align}
By the support of $\check\varphi$ we can replace the integration domain from 
$\tilde\x\in-B_\s^-(\tilde\z,\lambda)$ to $\tilde\x\in -B_\s^-(0,2\Lambda)$, so that
\begin{align}
d(\x) &= 
\sum_{\lambda\leq\Lambda} 
\int_{-B_\s^-(0,2\Lambda)} d\tilde\x \\
&\quad
\big(F_{\x+\S^\vartheta\y-\bar\x-\tilde\x}(\,\cdot\,;\x\!+\!\S^\vartheta\y) - F_{\x}(\,\cdot\,;\x\!+\!\S^\vartheta\y) \big)(\hat\varphi_{\x+\S^\vartheta\y-\bar\x-\tilde\x}^\lambda) 
(\check\varphi^\lambda * \psi^\Lambda)(-\tilde\x) .
\end{align}
Using Minkowski's inequality, 
$|\check\varphi^\lambda*\psi^\Lambda|\lesssim \Lambda^{-r-D} \lambda^r$ 
(see \cite[Lemma~4.8]{BL23}), 
and changing variables according to $\tilde{\mathbf{w}}=\x+\S^\vartheta\y-\bar\x-\tilde\x$ yields
\begin{align}
\int_{D_\y}d\tilde\x \, |d(\x)| 
&\lesssim \Lambda^{-r-D} \sum_{\lambda\leq\Lambda} \lambda^r 
\int_{-B_\s^-(0,2\Lambda)} d\tilde\x \int_{D_\y+\S^\vartheta\y-\bar\x-\tilde\x} d\tilde{\mathbf{w}} \\ 
&\quad\big| \big( F_{\tilde{\mathbf{w}}}(\,\cdot\,;\tilde{\mathbf{w}}+\bar\x+\tilde\x) 
- F_{\tilde{\mathbf{w}}-\S^\vartheta\y+\bar\x+\tilde\x}(\,\cdot\,;\tilde{\mathbf{w}}+\bar\x+\tilde\x) 
(\hat\varphi_{\tilde{\mathbf{w}}}^\lambda) \big| \, .
\end{align}
As for $b$ we argue that we can apply the assumption 
(with $\y'=-\S^\vartheta\y+\bar\x+\tilde\x$ and $\y''=\bar\x+\tilde\x$) 
and obtain 
\begin{align}
\int_{D_\y}d\tilde\x \, |d(\x)| 
&\lesssim \Lambda^{-r-D} \sum_{\lambda\leq\Lambda} \lambda^r 
\int_{-B_\s^-(0,2\Lambda)} d\tilde\x \, 
C \lambda^{\theta_1} (\lambda+\|\S^\vartheta\y-\bar\x-\tilde\x\|_\s)^{\theta_2} \\
&\lesssim C \Lambda^{-r} (\Lambda+\|v\y-\bar\x\|_\s)^{\theta_2} 
\sum_{\lambda\leq\Lambda} \lambda^{\theta_1+r} 
\lesssim C \Lambda^{\theta_1} (\Lambda+\|\S^\vartheta\y-\bar\x\|_\s)^{\theta_2} \, ,
\end{align}
where in the last inequality we used that by assumption $\theta_1+r>0$. 
\end{proof}

\subsection{Proof of Lemma \texorpdfstring{\ref{lem:absorb}}{back to mathcal(U)}}\label{app:absorb}

Before we give the proof of Lemma~\ref{lem:absorb} 
we prove the following auxiliary lemma.

\begin{lemma}\label{lem:absorb_aux}
Let $\alpha\in(2/3,1)$, $\beta\in(2-\alpha,2\alpha]$, $0<R,T\leq1$, 
and recall $\nu$ defined in \eqref{eq:gubinelliderivative} for some sufficiently smooth $u$.
Then 
\begin{enumerate}[label=(\roman*)]
\item \label{ita:u_alpha}
for $f\in C^1(\R)$
\begin{equation}
[f(u)]_{B_{1,\infty}^\alpha}
\leq [f(\mathcal{U})]_{B_{1,\infty}^{\beta}} R^{\beta-\alpha} 
+ \|f'\sigma/a^{\mathfrak{e}(\lollysmall)}\|_\infty T [\lolly]_\alpha 
+ \|f'(u)\nu\|_{L^1} R^{1-\alpha} \, ,
\end{equation}
\item \label{ita:f(U)_2alpha}
for $f\in C^2(\R)$ with $f(v)=const$ for $|v|\geq C$
\begin{align}
&[f(\mathcal{U})]_{B_{1,\infty}^{\beta}} \\
&\leq \|f'\|_\infty [\mathcal{U}]_{B_{1,\infty}^{\beta}} \\
&\,+ \|f''\|_\infty \|\sigma/a^{\mathfrak{e}(\lollysmall)}\|_\infty [\lolly]_\alpha 
\big(R^{2\alpha-\beta} \|\sigma/a^{\mathfrak{e}(\lollysmall)}\|_\infty T [\lolly]_\alpha 
+ \|\1_{|u|\leq2C}\nu\|_{L^1} R^{1+\alpha-\beta} \big) \\
&\,+ R^{2\alpha-\beta} 2^{2-2\alpha} \|f'\|_\infty^{2-2\alpha} \|f''\|_\infty^{2\alpha-1} \|\1_{|u|\leq2C}\nu\|_{L^1}^{2-2\alpha} \|\1_{|u|\leq2C}\nu\|_{L^2}^{2(2\alpha-1)} \\
&\,+ 4/C \|f\|_\infty [\mathcal{U}]_{B_{1,\infty}^\beta} \\
&\,+\frac{2^{1+3\beta/(2\alpha)}}{C^{\beta/\alpha}} \|f\|_\infty 
\big(T [\lolly]_\alpha^{\beta/\alpha} \|\sigma\!/\!a^{\mathfrak{e}(\lollysmall)}\|_\infty^{\beta/\alpha} 
+ T^{1-\beta/(2\alpha)} \|\1_{|u|\leq C}\nu\|_{L^2}^{\beta/\alpha} R^{\beta/\alpha-\beta} \big)\, ,
\end{align}
%
%
\item \label{ita:nu_2alpha-1}
for $f\in C^2(\R)$ and $\Theta$ from \eqref{eq:Theta}
\begin{align}
[f'(u)\nu]_{B_{1,\infty}^{\beta-1}} 
&\leq \sqrt{d} \big(6 [f(\mathcal{U})]_{B_{1,\infty}^{\beta}} \\
&\,+ R^{2\alpha-\beta} (\|(f'\!\sigma)'/a^{\mathfrak{e}(\lollysmall)}\|_\infty + \|f'\!\sigma/a^{1+\mathfrak{e}(\lollysmall)}\|_\infty) [\Theta(u)]_{B_{1,\infty}^\alpha} [\lolly]_\alpha\big) \, .
\end{align}
\end{enumerate}
\end{lemma}
\begin{proof}
\ref{ita:u_alpha}.
By the triangle inequality and the definitions of 
$[f(\mathcal{U})]_{B_{1,\infty}^{\beta}}$ and
$[\lolly]_\alpha$ we have
\begin{align}
&\int_{D_\y} d\x \big| f(u(\x+\y))-f(u(\x)) \big| \\
&\leq \int_{D_\y} d\x \big| f(u(\x+\y))-f(u(\x))-f'(u(\x))\big(\sigma(u(\x))\Pi_\x[\lolly;\x](\x+\y)+\nu(\x)\cdot y \big) \big| \\
&\,+ \int_{D_\y} d\x\,|f'(u(\x))|\, \big| \sigma(u(\x))\Pi_\x[\lolly;\x](\x+\y)+\nu(\x)\cdot y \big| \\
&\leq
\|\y\|_\s^{\beta} [f(\mathcal{U})]_{B_{1,\infty}^{\beta}}
+ \|\y\|_\s^\alpha \|f'\sigma/a^{\mathfrak{e}(\lollysmall)}\|_\infty T [\lolly]_\alpha 
+ \|f'(u)\nu\|_{L^1} \|\y\|_\s \, ,
\end{align}
hence the claim follows from the definition of $[\cdot]_{B_{1,\infty}^\alpha}$.

\ref{ita:f(U)_2alpha}.
For 
\begin{equation}\label{eq:f(U)}
f(u(\x+\y)) - f(u(\x)) - f'(u(\x))\big(\sigma(u(\x))\Pi_\x[\lolly;\x](\x+\y)+\nu(\x)\cdot y\big)
\end{equation}
we split 
\begin{equation}
\int_{D_\y}d\x\,|\eqref{eq:f(U)}|
=\int_{D_\y}d\x\,|\eqref{eq:f(U)}|\1_{|u(\x)|\leq2C}
+\int_{D_\y}d\x\,|\eqref{eq:f(U)}|\1_{|u(\x)|>2C} \, .
\end{equation}
For the former integral we rewrite
\begin{align}
\eqref{eq:f(U)} 
&=f(u(\x+\y)) - f\big(u(\x)+\sigma(u(\x))\Pi_\x[\lolly;\x](\x+\y)+\nu(\x)\cdot y\big) \label{eq:mt11}\\
&\,+ f\big(u(\x)+\sigma(u(\x))\Pi_\x[\lolly;\x](\x+\y)+\nu(\x)\cdot y\big) \\
&\,- f\big(u(\x)+\nu(\x)\cdot y\big) - f'(u(\x))\sigma(u(\x)) \Pi_\x[\lolly;\x](\x+\y) \label{eq:mt12}\\
&\,+f\big(u(\x)+\nu(\x)\cdot y\big) - f(u(\x)) - f'(u(\x))\nu(\x)\cdot y \, . \label{eq:mt13}
\end{align}
For the first right hand side term we have by the fundamental theorem of calculus
\begin{align}
\eqref{eq:mt11} 
=\! \int_0^1\! d\vartheta \, f'\big(\vartheta u(\x\!+\!\y)+(1-\vartheta)(u(\x)+\sigma(u(\x))\Pi_\x[\lolly;\x](\x\!+\!\y)+\nu(\x)\cdot y)\big)& \\
\quad\times\big(u(\x+\y)-u(\x)-\sigma(u(\x))\Pi_\x[\lolly;\x](\x+\y)-\nu(\x)\cdot y \big)& \, ,
\end{align}
and hence by 
the triangle inequality
and the definition of $[\mathcal{U}]_{B_{1,\infty}^{\beta}}$ that 
\begin{equation}
\int_{D_\y} d\x\, | \eqref{eq:mt11} | \1_{|u(\x)|\leq2C}
\leq \|f'\|_\infty [\mathcal{U}]_{B_{1,\infty}^{\beta}} \|\y\|_\s^{\beta} \, .
\end{equation}
For the second right hand side term we appeal twice to the fundamental theorem of calculus to obtain 
\begin{align}
\eqref{eq:mt12}
&=\!\int_0^1\! d\vartheta \,
f'\big( u(\x)\!+\!\vartheta \sigma(u(\x))\Pi_\x[\lolly;\x](\x\!+\!\y)\!+\!\nu(\x) \!\cdot\! y\big) 
\sigma(u(\x))\Pi_\x[\lolly;\x](\x\!+\!\y) \\
&\,- f'(u(\x))\sigma(u(\x))\Pi_\x[\lolly;\x](\x+\y) \\
&= \int_0^1 d\vartheta \int_0^1 d\bar{\vartheta} \,
f''\big( u(\x)+\vartheta\bar{\vartheta} \sigma(u(\x))\Pi_\x[\lolly;\x](\x+\y)+\bar{\vartheta}\nu(\x)\cdot y\big) \\
&\quad\times \sigma(u(\x))\Pi_\x[\lolly;\x](\x+\y) 
\big( \vartheta\sigma(u(\x))\Pi_\x[\lolly;\x](\x+\y) + \nu(\x)\cdot y\big) \, ,
\end{align}
so that by the triangle inequality and the definition of $[\lolly]_\alpha$ 
\begin{align}
&\int_{D_\y} d\x |\eqref{eq:mt12}| \1_{|u(\x)|\leq2C} \\
&\leq \|f''\|_\infty \|\sigma/a^{\mathfrak{e}(\lollysmall)}\|_\infty [\lolly]_\alpha \|\y\|_\s^\alpha
\big( \|\sigma/a^{\mathfrak{e}(\lollysmall)}\|_\infty T [\lolly]_\alpha \|\y\|_\s^\alpha + \|\1_{|u|\leq2C}\nu\|_{L^1} \|\y\|_\s \big) \, .
\end{align}
The third right hand side term is analogously estimated by 
\begin{equation}
\int_{D_\y}d\x |\eqref{eq:mt13}| \1_{|u(\x)|\leq2C}
\leq \|f''\|_\infty \|\1_{|u|\leq2C}\nu\|_{L^2}^2 \|\y\|_\s^2 \, .
\end{equation}
Alternatively, we can appeal just once to the fundamental theorem of calculus and break the resulting terms up by the triangle inequality to obtain 
\begin{equation}
\int_{D_\y}d\x |\eqref{eq:mt13}| \1_{|u(\x)|\leq2C}
\leq 2\|f'\|_\infty \|\1_{|u|\leq2C}\nu\|_{L^1} \|\y\|_\s \, .
\end{equation}
Interpolating between these two estimates yields 
\begin{align}
&\int_{D_\y}d\x |\eqref{eq:mt13}| \1_{|u(\x)|\leq2C} \\
&\leq 2^{2-2\alpha} \|f'\|_\infty^{2-2\alpha} \|\1_{|u|\leq2C}\nu\|_{L^1}^{2-2\alpha}
\|f''\|_\infty^{2\alpha-1} \|\1_{|u|\leq2C}\nu\|_{L^2}^{2(2\alpha-1)} \|\y\|_\s^{2\alpha} \, .
\end{align}
We turn to the contributions from $\1_{|u(\x)|>2C}$. Note that $|u(\x)|>2C$ implies by the assumption on $f$ that $f'(u(\x))=0$, and thus 
\begin{align}
|\eqref{eq:f(U)}| \1_{|u(\x)|>2C}
&= |f(u(\x+\y))-f(u(\x))| \1_{|u(\x)|>2C} \\
&= |f(u(\x+\y))-f(u(\x))| \1_{|u(\x)|>2C} \1_{|u(\x+\y)|\leq C} \, ,
\end{align}
where in the last equality we used once more the assumption on $f$. 
Thus 
\begin{align}
&\int_{D_\y}d\x\,|\eqref{eq:f(U)}| \1_{|u(\x)|>2C} \\
&\leq 2\|f\|_\infty \int_{D_\y}d\x\, 
\1_{|u(\x)|>2C} 
\1_{|u(\x+\y)+\sigma(u(\x+\y))\Pi_{\x+\y}[\lollysmall;\x+\y](\x)-\nu(\x+\y)\cdot y|\leq3C/2} \quad\qquad \label{eq:mt14} \\
&\,+ 2\|f\|_\infty \int_{D_\y}d\x\, 
\1_{|u(\x+\y)|\leq C} 
\1_{|u(\x+\y)+\sigma(u(\x+\y))\Pi_{\x+\y}[\lollysmall;\x+\y](\x)-\nu(\x+\y)\cdot y|>3C/2} \label{eq:mt15} \, .
\end{align}
By a change of variables, using $D_\y+\y\subseteq D_{-\y}$, and by the inverse triangle inequality, we deduce
\begin{align}
\eqref{eq:mt14} 
&\leq 2\|f\|_\infty \! \int_{D_{-\y}}\!d\x\, 
\1_{|u(\x-\y)-u(\x)-\sigma(u(\x))\Pi_{\x}[\lollysmall;\x](\x-\y)-\nu(\x)\cdot (-y)|>C/2} \\
&\leq 2\|f\|_\infty \! \int_{D_{-\y}}\!d\x\, 
\tfrac{2}{C} |u(\x\!-\!\y)\!-\!u(\x)-\sigma(u(\x))\Pi_{\x}[\lolly;\x](\x\!-\!\y)\!-\!\nu(\x)\cdot (-y)| \\
&\leq 4/C \|f\|_\infty [\mathcal{U}]_{B_{1,\infty}^\beta} \|\y\|_\s^\beta \, .
\end{align}
The same argumentation reveals 
\begin{align}
\eqref{eq:mt15} 
&\leq 2\|f\|_\infty \! \int_{D_{-\y}}\!d\x\, 
\1_{|u(\x)|\leq C} \1_{|\sigma(u(\x))\Pi_{\x}[\lollysmall;\x](\x-\y)-\nu(\x)\cdot y|>C/2} \, .
\end{align}
On the one hand, this is bounded by $2\|f\|_\infty T$, while on the other hand it is bounded by 
\begin{align}
&2\|f\|_\infty \! \int_{D_{-\y}}\!d\x\, 
\1_{|u(\x)|\leq C} \tfrac{4}{C^2} |\sigma(u(\x))\Pi_{\x}[\lolly;\x](\x-\y)-\nu(\x)\cdot y|^2 \\
&\leq16/C^2 \|f\|_\infty \big( T [\lolly]_\alpha^2 \|\sigma/a^{\mathfrak{e}(\lollysmall)}\|_\infty^2 \|\y\|_\s^{2\alpha}
+ \|\1_{|u|\leq C}\nu\|_{L^2}^2 \|\y\|_\s^2 \big) \, .
\end{align}
Interpolating between both estimates shows that $\eqref{eq:mt15}$ is bounded by 
\begin{align}
\frac{2^{1+3\beta/(2\alpha)}}{C^{\beta/\alpha}} \|f\|_\infty 
\big(T [\lolly]_\alpha^{\beta/\alpha} \|\sigma/a^{\mathfrak{e}(\lollysmall)}\|_\infty^{\beta/\alpha} \|\y\|_\s^{\beta}
+ T^{1-\beta/(2\alpha)} \|\1_{|u|\leq C}\nu\|_{L^2}^{\beta/\alpha} \|\y\|_\s^{\beta/\alpha} \big) \, .
\end{align}
Combining the estimates of \eqref{eq:mt11} -- \eqref{eq:mt15} with the definition of $[f(\mathcal{U})]_{B_{1,\infty}^{\beta}}$ yields the claim. 

\ref{ita:nu_2alpha-1}.
We introduce the shorthand notation
\begin{equation}
f(\mathcal{U})_\x(\x+\y) \coloneqq
f(u(\x+\y))-f(u(\x))-f'(u(\x))\big(\sigma(u(\x))\Pi_\x[\lolly;\x](\x+\y)+\nu(\x)\cdot y \big) \, .
\end{equation}
Then 
\begin{align}
&\big(f'(u(\x+\y))\nu(\x+\y)-f'(u(\x))\nu(\x)\big)\cdot z \\
&= f(\mathcal{U})_\x(\x+\y+\z) - f(\mathcal{U})_\x(\x+\y) - f(\mathcal{U})_{\x+\y}(\x+\y+\z) \\
&\,+f'(u(\x))\sigma(u(\x)) \big( \Pi_\x[\lolly;\x](\x+\y+\z) - \Pi_\x[\lolly;\x](\x+\y) \big) \\
&\,- f'(u(\x+\y))\sigma(u(\x+\y))\Pi_{\x+\y}[\lolly;\x+\y](\x+\y+\z) \\
&= f(\mathcal{U})_\x(\x+\y+\z) - f(\mathcal{U})_\x(\x+\y) - f(\mathcal{U})_{\x+\y}(\x+\y+\z) \\
&\,+f'(u(\x))\sigma(u(\x)) \Pi_{\x+\y}[\lolly;\x](\x+\y+\z) \\
&\,- f'(u(\x+\y))\sigma(u(\x+\y)) \Pi_{\x+\y}[\lolly;\x+\y](\x+\y+\z) \, ,
\end{align}
where in the second equality we have used $\Pi_\x[\lolly;\x](\x+\y+\z)-\Pi_\x[\lolly;\x](\x+\y)=\Pi_{\x+\y}[\lolly;\x](\x+\y+\z)$ as a consequence of \eqref{eq:trafololly}.
We choose $\z\coloneqq e_i \|\y\|_\s$ for the $d$ unit vectors $e_1,\dots,e_d$ of $\R^d$, 
and note that $D_\y=D_{\y+\z}$ as $\z$ has no time component. 
Thus 
\begin{align}
\int_{D_\y}\! d\x\, |f(\mathcal{U})_\x(\x+\y+\z)| 
&=\!\int_{D_{\y+\z}} \! d\x\, |f(\mathcal{U})_\x(\x+\y+\z)| \\
&\leq [f(\mathcal{U})]_{B_{1,\infty}^{\beta}} \! \|\y+\z\|_\s^{\beta} 
\leq 4 [f(\mathcal{U})]_{B_{1,\infty}^{\beta}} \! \|\y\|_\s^{\beta} \, .
\end{align}
Similarly, since $D_\y+\y\subseteq [0,T]\times\T^d$ by the definition of $D_\y$ (see Theorem~\ref{thm:main}), 
and again since $[0,T]\times\T^d=D_\z$ as $\z$ has no time component, we deduce
\begin{align}
\int_{D_\y} d\x \, |f(\mathcal{U})_{\x+\y}(\x+\y+\z)| 
&=\int_{D_\y+\y} d\x \, |f(\mathcal{U})_\x(\x+\z)| \\
&\leq\int_{D_\z} d\x \, |f(\mathcal{U})_\x(\x+\z)| 
\leq [f(\mathcal{U})]_{B_{1,\infty}^{\beta}} \|\y\|_\s^{\beta} \, .
\end{align}
Of course, also 
\begin{equation}
\int_{D_\y} |f(\mathcal{U})_\x(\x+\y)| \leq [f(\mathcal{U})]_{B_{1,\infty}^{\beta}} \|\y\|_\s^{\beta} \, .
\end{equation}
Furthermore, 
by \eqref{eq:Theta} and the fundamental theorem of calculus we have
\begin{align}
&(f'\!\sigma)(u(\x))\Pi_{\x+\y}[\lolly;\x](\x+\y+\z)
- (f'\!\sigma)(u(\x+\y))\Pi_{\x+\y}[\lolly;\x+\y](\x+\y+\z) \\
&= \int_0^1 d\vartheta \, \partial_v \big( f'(v) \sigma(v) \Pi_{\x+\y}[\lolly;a(v)](\x+\y+\z) \big)_{|v=\vartheta \Theta(u(\x))+(1-\vartheta)\Theta(u(\x+\y))} \\
&\quad\times \big(\Theta(u(\x))-\Theta(u(\x+\y))\big) \, ,
\end{align}
and thus
\begin{align}
&\int_{D_\y}\! d\x\, \Big| 
(f'\!\sigma)(u(\x))\Pi_{\x+\y}[\lolly;\x](\x\!+\!\y\!+\!\z)
\!-\! (f'\!\sigma)(u(\x\!+\!\y))\Pi_{\x+\y}[\lolly;\x\!+\!\y](\x\!+\!\y\!+\!\z) \Big| \\
&\leq \big(\|(f'\!\sigma)'/a^{\mathfrak{e}(\lollysmall)}\|_\infty + \|f'\!\sigma/a^{1+\mathfrak{e}(\lollysmall)}\|_\infty \big) 
[\lolly]_\alpha \|\y\|_\s^\alpha [\Theta(u)]_{B_{1,\infty}^\alpha} \|\y\|_\s^\alpha \, ,
\end{align}
where we have used the definitions of $[\lolly]_\alpha$ and $[\cdot]_{B_{1,\infty}^\alpha}$.
Altogether we proved 
\begin{align}
&\int_{D_\y}d\x \, \big|f'(u(\x+\y))\nu_i(\x+\y)-f'(u(\x))\nu_i(\x)\big| \|\y\|_\s \\
&\leq \|\y\|_\s^{\beta} 
\big( 6 [f(\mathcal{U})]_{B_{1,\infty}^{\beta}} 
\!+\! \|\y\|_\s^{2\alpha-\beta} (\|(f'\!\sigma)'/a^{\mathfrak{e}(\lollysmall)}\|_\infty \!+\! \|f'\!\sigma/a^{1+\mathfrak{e}(\lollysmall)}\|_\infty) [\Theta(u)]_{B_{1,\infty}^\alpha} [\lolly]_\alpha \big) 
\end{align}
for $i=1,\dots,d$ which implies the desired estimate. 
\end{proof}

Equipped with Lemma~\ref{lem:absorb_aux} we are now in position to prove Lemma~\ref{lem:absorb}.

\begin{proof}[Proof of Lemma~\ref{lem:absorb}]
Combining Lemma~\ref{lem:absorb_aux}~\ref{ita:f(U)_2alpha} with 
$|\sigma/a^{\mathfrak{e}(\lollysmall)}|\lesssim 1$
from equation \eqref{eq:XVI}, 
using $\|f\|_\infty,\|f'\|_\infty,\|f''\|_\infty\lesssim1$, $R,T\leq1$, and $\|\cdot\|_{L^1}\leq\sqrt{T}\|\cdot\|_{L^2}$
we obtain 
\begin{align}
[f(\mathcal{U})]_{B_{1,\infty}^{\beta}} 
&\lesssimdata [\mathcal{U}]_{B_{1,\infty}^{\beta}} 
+ \|\1_{|u|\leq2C}\nu\|_{L^2}^{2\alpha} + \|\1_{|u|\leq2C}\nu\|_{L^2}^{\beta/\alpha} + 1 \, .\label{eq:f(U)_2alpha}
\end{align}
Plugging this into Lemma~\ref{lem:absorb_aux}~\ref{ita:u_alpha} and using once more 
$|\sigma/a^{\mathfrak{e}(\lollysmall)}|\lesssim 1$
from equation \eqref{eq:XVI}, 
$\|f'\|_\infty\lesssim1$, and the assumption of $f$, yields
\begin{align}
[f(u)]_{B_{1,\infty}^\alpha} 
&\lesssimdata R^{\beta-\alpha} \big([\mathcal{U}]_{B_{1,\infty}^{\beta}} + \|\1_{|u|\leq2C}\nu\|_{L^2}^{2\alpha} + \|\1_{|u|\leq2C}\nu\|_{L^2}^{\beta/\alpha}\big) \\
&\,+ R^{1-\alpha} \sqrt{T} \|\1_{|u|\leq2C}\nu\|_{L^2} 
+ 1 \, . \label{eq:u_alpha}
\end{align}
Similarly, using additionally 
$|\sigma'/a^{\mathfrak{e}(\lollysmall)}|\lesssim1$
from equation \eqref{eq:III}
and 
$|\sigma/a^{1+\mathfrak{e}(\lollysmall)}|\lesssim1$
from equation \eqref{eq:V}, 
Lemma~\ref{lem:absorb_aux}~\ref{ita:nu_2alpha-1} implies for $C$ large enough ($\Theta$ is by construction constant outside a sufficiently large ball)
\begin{equation}\label{eq:nu_2alpha-1}
[f'(u)\nu]_{B_{1,\infty}^{\beta-1}} 
\lesssimdata [\mathcal{U}]_{B_{1,\infty}^{\beta}} 
+ \|\1_{|u|\leq2C}\nu\|_{L^2}^{2\alpha} + \|\1_{|u|\leq2C}\nu\|_{L^2}^{\beta/\alpha} + 1 \, .
\end{equation}

We're finally in position to tackle Item~\ref{it:nu^2_0}.
Note that for $1<p'\leq3-M$
\begin{equation}
\|\1_{|u|\leq2C}\nu\|_{L^2}^2
\lesssim \int_{[0,T]\times\T^d} d\z \, |u(\z)|^{p'-2} \,a(u(\z)) |\nu(\z)|^2
\lesssim 
\|u_0\|_{L^{p'}(\T^d)}^{p'} + C_{\eqref{eq:u2part3}} + 1 \, .
\end{equation}
The first inequality is an immediate consequence of Assumption~\ref{ass:nonlinearities}~(ii),
while the second inequality is exactly the content of Proposition~\ref{prop:energy}.
Choosing $p'\leq p$ with $p$ from Assumption~\ref{ass:initial} and plugging in $C_{\eqref{eq:u2part3}}$ from Proposition~\ref{prop:energy} we obtain 
\begin{align}
\|\1_{|u|\leq2C}\nu\|_{L^2}^2
\lesssimdata 
	1
	&+ [(g\sigma)(\mathcal{U})]_{B_{1,\infty}^{\beta}} 
	+ [\Theta(u)]_{B_{1,\infty}^\alpha}
	+ [\Theta'(u)\nu]_{B_{1,\infty}^{\beta-1}} \\
	&+ [\Theta(u)]_{B_{1,\infty}^\alpha}^{1/2} \|\Theta'(u)\nu\|_{L^2}
	+ \|\Theta'(u)\nu\|_{L^1} (\sqrt{T})^{\alpha-1} \, ,
\end{align}
where we recall that $g(v)=\int_0^vdw\,|w|^{p'-2}$. 
Using Cauchy-Schwarz in form of $\|\cdot\|_{L^1}\leq\sqrt{T}\|\cdot\|_{L^2}$, 
plugging in \eqref{eq:f(U)_2alpha} for $f=g\sigma$ 
(which by Assumption~\ref{ass:nonlinearities}~(ii) and \eqref{eq:restrictionM} is in $C^2(\R)$ and compactly supported), 
\eqref{eq:u_alpha} and \eqref{eq:nu_2alpha-1} for $f=\Theta$ 
(which by construction is constant outside a sufficiently large ball), 
and using $R,T\leq1$ yields for $C$ sufficiently large
\begin{align}
\|\1_{|u|\leq2C}\nu\|_{L^2}^2
&\lesssimdata 
	1
	+[\mathcal{U}]_{B_{1,\infty}^{\beta}} 
	+\|\1_{|u|\leq2C}\nu\|_{L^2}^{\max\{2\alpha,\beta/\alpha\}} \\
	&\quad+ \big([\mathcal{U}]_{B_{1,\infty}^{\beta}} 
		+\|\1_{|u|\leq2C}\nu\|_{L^2}^{\max\{2\alpha,\beta/\alpha\}}\big)^{1/2} 
	\|\1_{|u|\leq2C}\nu\|_{L^2} \\
&\lesssim 
	1
	+[\mathcal{U}]_{B_{1,\infty}^{\beta}} 
	+[\mathcal{U}]_{B_{1,\infty}^{\beta}}^{1/2} \|\1_{|u|\leq2C}\nu\|_{L^2}
	+ \|\1_{|u|\leq2C}\nu\|_{L^2}^{\max\{2\alpha,\beta/\alpha\}/2+1} \, .
\end{align}
Observe that $\max\{2\alpha,\beta/\alpha\}/2+1<2$ by assumption. 
We conclude the proof of Item~\ref{it:nu^2_0} by absorbing $\|\1_{|u|\leq2C}\nu\|_{L^2}$ as follows. 
We claim that 
$x\leq C_1 x^{\epsilon_1} + C_2 x^{\epsilon_2} +C_3$ for $0<\epsilon_1,\epsilon_2<1$ and some constants $C_1,C_2,C_3>0$ implies 
$x\lesssim C_1^{1/(1-\epsilon_1)} + C_2^{1/(1-\epsilon_2)} + C_3$. 
To prove this claim we distinguish two cases. 
In the case $x\geq(3C_1)^{1/(1-\epsilon_1)}$ and $x\geq(3C_2)^{1/(1-\epsilon_2)}$ 
we have $x^{\epsilon_1}\leq x/(3C_1)$ and $x^{\epsilon_2}\leq x/(3C_2)$ 
and thus 
\begin{equation}
x\leq C_1 x^{\epsilon_1} + C_2 x^{\epsilon_2} +C_3
\leq x/3 + x/3 +C_3 \, .
\end{equation}
In particular $x\leq3 C_3 \lesssim C_1^{1/(1-\epsilon_1)} + C_2^{1/(1-\epsilon_2)}+C_3$.
We conclude with noting that the remaining case is 
$x<(3C_1)^{1/(1-\epsilon_1)}$ or $x<(3C_2)^{1/(1-\epsilon_2)}$.

This establishes $\|\1_{|u|\leq2C}\nu\|_{L^2}^2\lesssim [\mathcal{U}]_{B_{1,\infty}^\beta} +1$, 
and Item~\ref{it:nu^2_0} follows from $|f'(u)|\lesssim\1_{|u|\leq2C}$. 
The remaining Items~\ref{it:u_alpha}, \ref{it:f(U)_2alpha}, and \ref{it:nu_2alpha-1}
follow by plugging the already established bound on $\|\1_{|u|\leq2C}\nu\|_{L^2}^2$ into \eqref{eq:f(U)_2alpha} -- \eqref{eq:nu_2alpha-1} and using $R,T\leq1$.
\end{proof}

\section*{Acknowledgements}

The authors thank Lucas Broux, Salvador Esquivel, Simon Gabriel, Benjamin Gess, Fabian Höfer, and Rhys Steele for illuminating discussions. 
Both authors are funded by 
the European Research Council (ERC) under the European Union’s Horizon 2020 research and innovation programme (Grant agreement No.~101045082), and by 
the Deutsche Forschungsgemeinschaft (DFG, German Research Foundation) under Germany's Excellence Strategy EXC 2044-390685587, Mathematics Münster: Dynamics-Geometry-Structure.

%
%
%
%
\pdfbookmark{References}{references}
\addtocontents{toc}{\protect\contentsline{section}{References}{\thepage}{references.0}}
%
%
\bibliographystyle{alphaurl}
\small
\bibliography{references}{}
%
%
\end{document}